\documentclass[12pt]{article}

\usepackage{amsmath, amsthm, amssymb, stmaryrd}
\usepackage[all]{xy}
\usepackage{fullpage}
\usepackage{titlesec}
\usepackage{mathrsfs}
\usepackage{amsbsy}
\usepackage{turnstile}
\usepackage{bbm}
\usepackage{yhmath}
\usepackage{tensor}
\usepackage{graphicx}
\usepackage{enumitem}
\usepackage{calligra}
\usepackage{verbatim}
\usepackage{setspace}
\usepackage{hhline}
\usepackage{array}
\usepackage{cancel}
\usepackage{trimclip}
\usepackage{adjustbox}
\usepackage{bbm}

\xyoption{rotate}

\usepackage[pdftex,bookmarks=true,pdfborderstyle={/S/D/D[1 0]/W 0.5}]{hyperref}
\usepackage[svgnames]{xcolor}
\hypersetup{
    linkbordercolor = {PaleTurquoise},
    citebordercolor = {PaleGreen},
}
\usepackage{titletoc}

\titlecontents{section}
[0pt]                                               
{}%
{\contentsmargin{0pt}                               
    \thecontentslabel\enspace%
    }
{\contentsmargin{0pt}}                        
{\titlerule*[.5pc]{.}\contentspage}                 
[]                                                  

\makeatletter
\newcommand*{\@old@slash}{}\let\@old@slash\slash
\def\slash{\relax\ifmmode\delimiter"502F30E\mathopen{}\else\@old@slash\fi}
\makeatother

\titleformat{\section}{\normalsize\bfseries}{\thesection}{1em}{}
\titleformat{\subsection}{\normalsize\bfseries}{\thesubsection}{1em}{}

\numberwithin{equation}{subsection}

\theoremstyle{plain}

\newtheorem{prop}[subsection]{Proposition}
\newtheorem{lem}[subsection]{Lemma}
\newtheorem{cor}[subsection]{Corollary}
\newtheorem{thm}[subsection]{Theorem}

\theoremstyle{definition}

\newtheorem{defn}[subsection]{Definition}
\newtheorem{exa}[subsection]{Example}
\newtheorem{rem}[subsection]{Remark}
\newtheorem{para}[subsection]{}

\newcommand*{\emptybox}{\leavevmode\hbox{}}

\newdir{ >}{{}*!/-10pt/@{>}}

\DeclareMathAlphabet{\mathpzc}{OT1}{pzc}{m}{it}
\DeclareMathAlphabet{\mathcalligra}{T1}{calligra}{m}{n}

\newcommand{\sfV}{\mathsf{V}}

\newcommand{\cP}{\ensuremath{\mathcal{P}}}

\newcommand{\A}{\ensuremath{\mathscr{A}}}
\newcommand{\B}{\ensuremath{\mathscr{B}}}
\newcommand{\C}{\ensuremath{\mathscr{C}}}
\newcommand{\D}{\ensuremath{\mathscr{D}}}
\newcommand{\E}{\ensuremath{\mathscr{E}}}
\newcommand{\F}{\ensuremath{\mathscr{F}}}
\newcommand{\G}{\ensuremath{\mathscr{G}}}

\newcommand{\normalJ}{\ensuremath{\mathscr{J}}}
\newcommand{\J}{\ensuremath{{\kern -0.4ex \mathscr{J}}}}

\newcommand{\K}{\ensuremath{\mathscr{K}}}

\newcommand{\R}{\ensuremath{\mathscr{R}}}
\newcommand{\sS}{\ensuremath{\mathscr{S}}}
\newcommand{\T}{\ensuremath{\mathscr{T}}}

\newcommand{\V}{\ensuremath{\mathscr{V}}}
\newcommand{\W}{\ensuremath{\mathscr{W}}}
\newcommand{\X}{\ensuremath{\mathscr{X}}}

\newcommand{\shortunderline}[1]{\underline{#1\mkern-5mu}\mkern5mu}

\newcommand{\CAT}{\ensuremath{\operatorname{\textnormal{\text{CAT}}}}}

\newcommand{\VCAT}{\ensuremath{\V\textnormal{-\text{CAT}}}}

\newcommand{\ALGCAT}{\ensuremath{\operatorname{\textnormal{\text{ALGCAT}}}}}
\newcommand{\ALGCATJ}{\ensuremath{\ALGCAT_{\kern -0.6ex \normalJ}}}
\newcommand{\ADA}{\ensuremath{\operatorname{\textnormal{\text{ADA}}}}}
\newcommand{\ADAJ}{\ensuremath{\ADA_{\kern -0.6ex \normalJ}}}

\newcommand{\MCAT}{\ensuremath{\operatorname{\textnormal{\text{MCAT}}}}}

\newcommand{\TWOCAT}{\ensuremath{\operatorname{\textnormal{\text{2CAT}}}}}

\newcommand{\NN}{\ensuremath{\mathbb{N}}}

\newcommand{\TT}{\ensuremath{\mathbb{T}}}

\newcommand{\ob}{\ensuremath{\operatorname{\textnormal{\textsf{ob}}}}}

\newcommand{\Th}{\ensuremath{\textnormal{Th}}}
\newcommand{\ThJ}{\ensuremath{\Th_{\kern -0.5ex \normalJ}}}

\newcommand{\otimesJ}[2]{\ensuremath{#1 \kern-0.15ex \otimes_{\kern -1ex \normalJ} \kern-0.5ex #2}}
\newcommand{\totimes}{\mathbin{\tilde{\otimes}}}
\newcommand{\totimesJ}[2]{\ensuremath{#1 \kern-0.15ex \totimes_{\kern -1ex \normalJ} \kern-0.5ex #2}}

\newcommand{\Vfp}{\ensuremath{\V_{\kern -0.5ex fp}}}
\newcommand{\TTperpJ}{\ensuremath{\TT^\perp_{\kern-.1ex\scriptscriptstyle\J}}}
\newcommand{\TTperpJwrt}[1]{\ensuremath{\TT^\perp_{\kern-.1ex\scriptscriptstyle\J#1}}}

\newcommand{\Ev}{\ensuremath{\textnormal{\textsf{Ev}}}}

\newcommand{\y}{\ensuremath{\textnormal{\textsf{y}}}}

\newcommand{\VCATJ}{\VCAT_{\kern -0.7ex \normalJ}}
\newcommand{\PhiJ}{\Phi_{\kern -0.8ex \normalJ}}
\newcommand{\MndJ}{\Mnd_{\kern -0.8ex \normalJ}}

\newcommand{\Set}{\ensuremath{\operatorname{\textnormal{\text{Set}}}}}

\newcommand{\SET}{\ensuremath{\operatorname{\textnormal{\text{SET}}}}}

\newcommand{\Top}{\ensuremath{\operatorname{\textnormal{\text{Top}}}}}

\newcommand{\Mnd}{\ensuremath{\operatorname{\textnormal{\text{Mnd}}}}}

\newcommand{\Alg}[1]{\ensuremath{#1\kern -.5ex\operatorname{\textnormal{-Alg}}}}
\newcommand{\CAlg}[1]{\ensuremath{#1\kern -.5ex\operatorname{\textnormal{-\text{CAlg}}}}}
\newcommand{\CAlgPair}[2]{\ensuremath{(#1,#2)\kern -.5ex\operatorname{\textnormal{-CPair}}}}
\newcommand{\AlgPair}[2]{\ensuremath{(#1,#2)\kern -.5ex\operatorname{\textnormal{-Pair}}}}
\newcommand{\CinftyRing}{\ensuremath{C^\infty\kern -.5ex\operatorname{\textnormal{-\text{Ring}}}}}

\newcommand{\CRProf}{\ensuremath{\operatorname{\textnormal{\text{CRProf}}}}}
\newcommand{\CRProfJ}{\ensuremath{\CRProf_{\kern -0.6ex \normalJ}}}
\newcommand{\Algs}{\ensuremath{\operatorname{\textnormal{\text{Alg}}}}}
\newcommand{\AlgsJ}{\Algs_{\kern -0.6ex \normalJ}}

\newcommand{\BifoldAlg}{\ensuremath{\operatorname{\textnormal{\text{BAlg}}}}}
\newcommand{\BifoldAlgJ}{\BifoldAlg_{\kern -0.6ex \normalJ}}

\newcommand{\Adjn}[6]{\xymatrix {#1 \ar@/_0.5pc/[rr]_{#2}^(0.4){#4}^(0.6){#5}^{\top} & & #6 \ar@/_0.5pc/[ll]_{#3}}}
\newcommand{\Equiv}[6]{\xymatrix {#1 \ar@/_0.5pc/[rr]_{#2}^(0.4){#4}^(0.6){#5}^{\sim} & & #6 \ar@/_0.5pc/[ll]_{#3}}}
\newcommand{\EquivAlt}[6]{\xymatrix {#1 \ar@{<-}@/_0.5pc/[rr]_{#3}^(0.4){#4}^(0.6){#5}^{\sim} & & #6 \ar@{<-}@/_0.5pc/[ll]_{#2}}}
\newcommand{\AlgDual}[4]{\xymatrix {#1:#2 \ar@/_0.2pc/[r] & #3:#4 \ar@/_0.2pc/[l] }}

\def\modto{\mathop{\rlap{\hspace{1.35ex}\raisebox{0.1ex}{$\shortmid$}}{\longrightarrow}}\nolimits}

\newcommand{\op}{\ensuremath{\mathsf{op}}}
\newcommand{\co}{\ensuremath{\mathsf{co}}}
\newcommand{\coop}{\ensuremath{\mathsf{coop}}}

\newcommand{\pushoutcorner}{\ar@{}[dr]|(.3)\ulcorner}
\newcommand{\pullbackcorner}{\ar@{}[dr]|(.3)\lrcorner}

\newcommand{\cmt}[1]{}

\newcommand{\strong}{\ensuremath{\mathrm{strong}}}

\newcommand{\rev}{\ensuremath{\mathsf{rev}}}

\newcommand{\lrsub}[3]{\ensuremath{\tensor[_{#1}]{\textnormal{#2}}{_{#3}}}}

\newcommand{\LRCAT}[2]{\lrsub{#1}{\textnormal{CAT}}{#2}}
\newcommand{\LCAT}[1]{\LRCAT{#1}{}}
\newcommand{\RCAT}[1]{\LRCAT{}{#1}}

\newcommand{\GCAT}[2]{\lrsub{#1}{\textnormal{GCAT}}{#2}}
\newcommand{\LGCAT}[1]{\GCAT{#1}{}}
\newcommand{\RGCAT}[1]{\GCAT{}{#1}}

\newcommand{\ACT}[2]{\lrsub{#1}{\textnormal{ACT}}{#2}}
\newcommand{\LACT}[1]{\ACT{#1}{}}

\newcommand{\LRMOD}[2]{\lrsub{#1}{\textnormal{MOD}}{#2}}
\newcommand{\LMOD}[1]{\LRMOD{#1}{}}
\newcommand{\RMOD}[1]{\LRMOD{}{#1}}

\newcommand{\LRGMOD}[2]{\lrsub{#1}{\textnormal{GMOD}}{#2}}
\newcommand{\LGMOD}[1]{\LRGMOD{#1}{}}
\newcommand{\RGMOD}[1]{\LRGMOD{}{#1}}

\newcommand{\oplax}{\ensuremath{\mathrm{oplax}}}

\newcommand{\gc}{,}

\newcommand{\gct}{\kern -0.6ex\cdot\kern-0.4ex}

\newcommand{\gcl}{\kern -0.1ex\cdot\kern-0.2ex}

\newcommand{\dpr}{\star}

\newcommand{\bigdot}{{\mathbin{\textnormal{\raisebox{-0.9ex}[0ex][0ex]{\scalebox{2.5}{\kern -0.1ex$\cdot$\kern -0.1ex}}}}}}

\newcommand{\smsub}[2]{#1_{\scriptscriptstyle#2}}

\newcommand{\otsub}[1]{\smsub{\otimes}{#1}}

\newcommand{\dV}{\mathrm{V}}
\newcommand{\tinydV}{\scriptscriptstyle\mathrm{V}}

\newcommand{\CleftV}{\tensor[_\V]{\kern -0.5ex\C}{}}

\newcommand{\gid}{\mathsf{i}}

\begin{document}

\author{\normalsize  Rory B. B. Lucyshyn-Wright\thanks{We acknowledge the support of the Natural Sciences and Engineering Research Council of Canada (NSERC), [funding reference numbers RGPIN-2019-05274, RGPAS-2019-00087, DGECR-2019-00273].  Cette recherche a \'et\'e financ\'ee par le Conseil de recherches en sciences naturelles et en g\'enie du Canada (CRSNG), [num\'eros de r\'ef\'erence RGPIN-2019-05274, RGPAS-2019-00087, DGECR-2019-00273].}\let\thefootnote\relax
\\
\small Brandon University, Brandon, Manitoba, Canada}

\title{\large \textbf{$\V$-graded categories and $\V$-$\W$-bigraded categories:\\Functor categories and bifunctors over non-symmetric bases}}

\date{}

\maketitle

\abstract{
In the well-known settings of category theory enriched in a monoidal category $\V$, the use of $\V$-enriched functor categories and bifunctors demands that $\V$ be equipped with a symmetry, braiding, or duoidal structure. In this paper, we establish a theory of functor categories and bifunctors that is applicable relative to an arbitrary monoidal category $\V$ and applies both to $\V$-enriched categories and also to $\V$-actegories. We accomplish this by working in the setting of \textit{($\V$-)graded categories}, which generalize both $\V$-enriched categories and $\V$-actegories and were introduced by Wood under the name \textit{large $\V$-categories}. We develop a general framework for graded functor categories and graded bifunctors taking values in \textit{bigraded categories}, noting that $\V$ itself is canonically bigraded. We show that $\V$-graded modules (or profunctors) are examples of graded bifunctors and that $\V$-graded presheaf categories are examples of $\V$-graded functor categories. In the special case where $\V$ is normal duoidal, we compare the above graded concepts with the enriched bifunctors and functor categories of Garner and L\'opez Franco. Along the way, we study several foundational aspects of graded categories, including a contravariant change of base process for graded categories and a formalism of commutative diagrams in graded categories that arises by freely embedding each $\V$-graded category into a $\V$-actegory.
}

\newpage

\tableofcontents

\section{Introduction} \label{sec:intro}

Categories graded by a monoidal category $\V$, or \textit{($\V$-)graded categories}, were introduced by Wood \cite{Wood:thesis,Wood:Vindexed} under the name \textit{large $\V$-categories} and have been discussed also in \cite{KLSS:CatsEnrTwoSides,Mell:Param,Campbell:SkewEnriched,Levy:LocGraded,McDUu,McDUu:StrMnd,Gav:Thesis} and, in special cases, in \cite{Levy:AdjunctionModels,EgMoSi,GKOS}. Here we employ the term \textit{$\V$-graded category} as an abbreviation of the term \textit{locally $\V$-graded category} introduced by Levy \cite{Levy:LocGraded}, while the term \textit{procategory} is also used \cite{KLSS:CatsEnrTwoSides}. Categories graded by $\V$ provide a simultaneous generalization of $\V$-enriched categories and \textit{$\V$-actegories}, in the terminology of McCrudden \cite{McCr:CatsRepsCoa}, which are categories equipped with an action of $\V$ in the sense of B\'enabou \cite[(2.3)]{Ben:Bicats} and appear in noncommutative algebra and representation theory under the name \textit{module categories} \cite{EGNO:TensorCategories}. Indeed, the 2-categories of $\V$-categories and of $\V$-actegories embed into the 2-category of $\V$-graded categories, whose 1-cells we call \textit{$\V$-graded functors}.  The 1-cells between $\V$-actegories that we consider here are functors laxly preserving the action, which are also called \textit{strong functors} and play an important role in computer science, as discussed in \cite{McDUu:StrMnd} and the references therein. In particular, we may conveniently regard an arbitrary monoidal category $\V$ itself as a $\V$-actegory and so as a $\V$-graded category, even when $\V$ is \textit{not} assumed closed, and thus we recover the notion of strong endofunctor $F:\V \to \V$ that originates with Kock \cite{Kock:Comm}. But $\V$-graded categories offer added flexibility over $\V$-actegories, since $\V$-enriched categories and full subcategories of $\V$-actegories may be regarded as $\V$-graded categories.

Categories graded by a monoidal category $\V$ may be defined succinctly as categories enriched in the functor category $\hat{\V} := [\V^\op,\SET]$ equipped with its Day convolution monoidal structure \cite{Day:ClCatsFunc}, where $\SET$ is the category of (possibly large) sets. But the notion of $\V$-graded category also admits a direct, elementwise definition that makes it an elementary and intuitive categorical concept in its own right. Explicitly, a $\V$-graded category $\C$ consists of a set $\ob\C$ and sets $\C(A,B)X$ $(A,B \in \ob\C, X \in \ob\V)$ whose elements we call \textit{graded morphisms from $A$ to $B$ with grade $X$}, plus some further data as we recall in \S \ref{sec:graded_cats}. In this paper, we write $\C(X \gc A;B)$ as a notation for $\C(A,B)X$, and we write $f:X \gc A \to B$ to mean that $f \in \C(X \gc A;B)$. For example, every $\V$-actegory $\C$ underlies a $\V$-graded category in which a graded morphism $f:X \gc A \to B$ is a morphism $f:X.A \to B$ in $\C$, where $X.A$ is the object obtained via the action. 

The fact that an arbitrary monoidal category $\V$ may be regarded as a $\V$-graded category is one of several ways in which $\V$-graded categories are convenient when working with monoidal categories $\V$ on which few assumptions (if any) are to be made. After all, $\V$-graded categories are categories enriched in the above monoidal category $\hat{\V} = [\V^\op,\SET]$, which is biclosed and has $\SET$-small limits and colimits, without any assumptions on $\V$. Hence the full force of known results in the theory of categories enriched in a biclosed, complete, and cocomplete monoidal category is applicable to $\V$-graded categories by way of the base of enrichment $\hat{\V}$, and this allows the methods of categories enriched in a locally cocomplete and (bi)closed bicategory to be applied; see \cite{Str:EnrCatsCoh,BCSW:VarEnr,GP} for example. In particular, this entails that $\V$-graded categories admit a calculus of $\hat{\V}$-enriched modules and weighted colimits \cite{Str:EnrCatsCoh,BCSW:VarEnr}, and even that certain $\V$-graded categories of $\hat{\V}$-enriched presheaves can be formed, following Street \cite[\S 3]{Str:EnrCatsCoh}. When $\V$ is assumed \textit{symmetric} monoidal, the biclosed monoidal category $\hat{\V}$ is also symmetric, so that the full-blown methods of enrichment in a complete and cocomplete symmetric monoidal closed category \cite{Ke:Ba} are applicable to $\V$-graded categories; thus if $\V$ is symmetric, or more generally \textit{braided} \cite{JoyStr}, then one may form monoidal products $\A \otimes \B$ of $\V$-graded categories $\A$ and $\B$ and so consider $\V$-graded bifunctors $\A \otimes \B \to \C$, and one may form $\V$-graded functor categories $[\A,\B]$ as was discussed by Wood \cite{Wood:thesis}.

But for an arbitrary monoidal category $\V$, these familiar constructions of $\V$-graded functor categories and monoidal products of $\V$-graded categories are \textit{not} available, and $\hat{\V}$-modules are a basic notion of their own and are \textit{not} defined as $\V$-graded bifunctors. Indeed, the well-known definitions of enriched functor categories \cite{DayKe:EnrFuncCats} and bifunctors \cite[III.4]{EiKe} depend on assuming that the base of enrichment is equipped with a symmetry or braiding, or more generally a \textit{normal duoidal structure} \cite[\S 2, \S 3.2--3.4]{GaLf}.

In this paper, we define a framework of $\V$-graded functor categories and bifunctors for \textit{arbitrary} monoidal categories $\V$, \textit{not} assumed symmetric, and \textit{not} assumed duoidal. To do this, we consider also the monoidal category $\V^\rev$ that is called the \textit{reverse} of $\V$ and has the same underlying ordinary category as $\V$ but is equipped with the monoidal product $\otsub{\V^\rev}$ given by $X \otsub{\V^\rev} Y = Y \otimes X$. We simultaneously consider both $\V$-graded categories and $\V^\rev$-graded categories, calling the former \textit{left} $\V$-graded categories and the latter \textit{right} $\V$-graded categories, as they generalize left $\V$-actegories and right $\V$-actegories, respectively. In fact, we work in a more general setting by fixing two monoidal categories $\V$ and $\W$ and considering left $\V$-graded categories and right $\W$-graded categories. Crucially, we also define \textit{$\V$-$\W$-bigraded categories} as left $(\V \times \W^\rev)$-graded categories, so that each $\V$-$\W$-bigraded category has an underlying left $\V$-graded category and an underlying right $\W$-graded category. For example, $\V$ itself underlies a $\V$-$\V$-bigraded category, as does $\hat{\V}$.  In the special case where $\V$ is braided, which we do not assume, every $\V$-graded category is $\V$-$\V$-bigraded.

Given a left $\V$-graded category $\A$ and a $\V$-$\W$-bigraded category $\C$, we show that (left) $\V$-graded functors $F:\A \to \C$ are the objects of a right $\W$-graded category that we denote by $[\A,\C]$ or $\tensor[^\V]{[\A,\C]}{_\W}$. Similarly, we show that if $\B$ is a right $\W$-graded category and $\C$ is a $\V$-$\W$-bigraded category, then right $\W$-graded functors $F:\B \to \C$ are the objects of a left $\V$-graded category $[\B,\C] = \tensor[_\V]{[\B,\C]}{^\W}$. For example, if $\A$ is a left $\V$-graded category then $[\A,\V]$ is a right $\V$-graded category, while if $\B$ is a right $\V$-graded category then $[\B,\V]$ is a left $\V$-graded category.

Given a left $\V$-graded category $\A$, a right $\W$-graded category $\B$, and a $\V$-$\W$-bigraded category $\C$ we define a \textit{graded bifunctor} $F:\A,\B \to \C$ to consist of left $\V$-graded functors $F(-,B):\A \to \C$ $(B \in \ob\B)$ and right $\W$-graded functors $F(A,-):\B \to \C$ $(A \in \ob\A)$ that agree on objects and satisfy a certain compatibility condition. Graded bifunctors $F:\A,\B \to \C$ are the objects of a category $\tensor[_\V]{\textnormal{\textsf{GBif}}}{_\W}(\A,\B;\C)$. Writing $\LGCAT{\V}$, $\RGCAT{\W}$, and $\GCAT{\V}{\W}$ for the 2-categories of left $\V$-graded categories, right $\W$-graded categories, and $\V$-$\W$-bigraded categories, respectively, we show that there are isomorphisms
$$\LGCAT{\V}(\A,\tensor[_\V]{[\B,\C]}{^\W}) \cong \tensor[_\V]{\textnormal{\textsf{GBif}}}{_\W}(\A,\B;\C) \cong \RGCAT{\W}(\B,\tensor[^\V]{[\A,\C]}{_\W})$$
2-natural in $\A \in \LGCAT{\V}$, $\B \in \RGCAT{\W}$, $\C \in \GCAT{\V}{\W}$, so that a graded bifunctor $F:\A,\B \to \C$ is equivalently given by a left $\V$-graded functor $\A \to \tensor[_\V]{[\B,\C]}{^\W}$ or a right $\W$-graded functor $\B \to \tensor[^\V]{[\A,\C]}{_\W}$.

Given a left $\V$-graded category $\A$ and a right $\W$-graded category $\B$, we define a $\V$-$\W$-bigraded category $\A \boxtimes \B$ that we call the \textit{bigraded product} of $\A$ and $\B$, whose objects are pairs $(A,B)$ with $A \in \ob\A$ and $B \in \ob\B$. We show that
$$\tensor[_\V]{\textnormal{GCAT}}{_\W}(\A \boxtimes \B,\C) \cong \tensor[_\V]{\textnormal{\textsf{GBif}}}{_\W}(\A,\B;\C)$$
2-naturally in $\A \in \LGCAT{\V}$, $\B \in \RGCAT{\W}$, $\C \in \GCAT{\V}{\W}$, so that a $\V$-$\W$-bigraded functor $F:\A \boxtimes \B \to \C$ is equivalently given by a graded bifunctor $F:\A,\B \to \C$. Thus we show that
$$\LGCAT{\V}(\A,\tensor[_\V]{[\B,\C]}{^\W}) \cong \tensor[_\V]{\textnormal{GCAT}}{_\W}(\A \boxtimes \B,\C) \cong \RGCAT{\W}(\B,\tensor[^\V]{[\A,\C]}{_\W})$$
2-naturally in $\A \in \LGCAT{\V}$, $\B \in \RGCAT{\W}$, $\C \in \GCAT{\V}{\W}$. 

In the special case where $\V$ is \textit{symmetric} monoidal and we take $\W = \V$, there is no essential distinction between left and right $\V$-graded categories, while every $\V$-graded category is canonically $\V$-$\V$-bigraded, and we recover the usual $\hat{\V}$-enriched concepts for the symmetric monoidal category $\hat{\V} = [\V^\op,\textnormal{SET}]$, though $\A \boxtimes \B$ does not coincide with the monoidal product of $\hat{\V}$-categories $\A \otimes \B$. More generally, if $\V$ is a \textit{duoidal category} \cite{AgMa} that is \textit{normal} in the sense of \cite{GaLf}, as is the case for any braided monoidal category, then $\hat{\V}$ carries the structure of a normal duoidal category by \cite{BooSt}, so in this case one can also apply the work of Garner and L\'opez Franco \cite{GaLf} to define $\hat{\V}$-enriched bifunctors and functor categories, and we show that these are recovered as instances of the above concepts. Furthermore, for a normal duoidal category $\V$ we show that the notions of $\V$-enriched bifunctor and functor category of Garner and L\'opez Franco \cite[\S 3.2--3.4]{GaLf} can be regarded as examples of the above graded bifunctors and functor categories.

For an arbitrary monoidal category $\V$, we refer to $\hat{\V}$-enriched modules (or $\hat{\V}$-enriched profunctors) as \textit{$\V$-graded modules}. We show that $\V$-graded modules are examples of graded bifunctors. In detail, every right $\V$-graded category $\B$ determines a left $\V$-graded category $\B^\circ$ that we call the \textit{formal opposite} of $\B$, and we show that if $\A$ and $\B$ are right $\V$-graded categories then a $\V$-graded module $M:\A \modto \B$ is equivalently given by a graded bifunctor $M:\B^\circ,\A \to \hat{\V}$ or a $\V$-$\V$-bigraded functor $M:\B^\circ \boxtimes \A \to \hat{\V}$, where we regard $\hat{\V}$ as a $\V$-$\V$-bigraded category.

Given a right $\V$-graded category $\B$, we show that Street's $\hat{\V}$-enriched presheaf category $\cP\B$ \cite[\S 3]{Str:EnrCatsCoh} is recovered up to isomorphism as an example of a graded functor category, namely the right $\V$-graded category $[\B^\circ,\hat{\V}] = \tensor[^\V]{[\B^\circ,\hat{\V}]}{_\V}$ whose objects are left $\V$-graded functors $F:\B^\circ \to \hat{\V}$ with the above notations. Indeed, we deduce this as a consequence of the above results, which yield 2-natural correspondences between right $\V$-graded functors $\A \to [\B^\circ,\hat{\V}]$, graded bifunctors $\B^\circ,\A \to \hat{\V}$, and $\V$-graded modules $\A \modto \B$. Under these correspondences, the identity module on $\B$ corresponds to a fully faithful $\V$-graded functor $\y:\B \to [\B^\circ,\hat{\V}]$ that we identify with Street's Yoneda embedding \cite[\S 3]{Str:EnrCatsCoh}, and we thus obtain a \textit{$\V$-graded Yoneda lemma} to the effect that $[\B^\circ,\hat{\V}](\y B,F) \cong FB$, $\V$-graded naturally in $B \in \B$ and $F \in [\B^\circ,\hat{\V}]$.

Our results are enabled by a novel and convenient formalism of commutative diagrams in $\V$-graded categories that we establish by freely embedding every $\V$-graded category $\C$ into a $\V$-actegory that we call the \textit{enveloping actegory}, which we construct as a free cocompletion of $\C$ with respect to copowers (or tensors) by objects of $\V$, using free cocompletion for a class of weights \cite{Ke:Ba,PCW:Cocompl} in enriched category theory.

We now comment on the relation between this work and two other general frameworks. Firstly, categories and procategories \textit{enriched on two sides} \cite{KLSS:CatsEnrTwoSides} (from one bicategory to another) generalize $\V$-enriched and $\V$-graded categories, respectively; they are defined in \cite[\S 2.2, 6.2]{KLSS:CatsEnrTwoSides} by directly describing their data, which are somewhat more complex than those of a $\V$-enriched or $\V$-graded category. However, in the special case where $\V$ and $\W$ are monoidal categories, procategories from $\V$ to $\W$ are equivalently $(\W \times \V^\op)$-graded categories, in view of \cite[Proposition 6.11]{KLSS:CatsEnrTwoSides}, so that $\V$-$\W$-bigraded categories are equivalently procategories from $\W^{\rev\,\op}$ to $\V$. Furthermore, the bigraded product can be regarded as a special case of an operation on two-sided procategories discussed in \cite[\S 6.11]{KLSS:CatsEnrTwoSides}. But functor categories for categories and procategories enriched on two sides have not been developed, so it would be interesting to investigate whether the constructions in this paper might generalize to that setting; we leave this for future work. Secondly, Wood \cite{Wood:thesis,Wood:Vindexed} defined not only the notion of $\V$-graded category but also a more general (and more complex) concept under the name of \textit{$\V$-indexed category}; these are categories enriched in a cartesian closed category of \textit{$\V$-indexed sets}, which is defined as a fibre product over $\CAT$ of the finite limit theory of monoids and a certain category fibred over $\CAT$ whose fibres are presheaf categories. Because the resulting base of enrichment is cartesian closed, $\V$-indexed functor categories can be formed, but Wood writes in \cite{Wood:Vindexed} that ``It is fairly difficult to explicitly calculate functor categories in'' the 2-category of $\V$-indexed categories. This contrasts with the explicit constructions of $\V$-graded functor categories in this paper. Regardless, it would be interesting to investigate the relation between functor categories in these two settings; we leave this for future work also.

We now provide an overview of the organization of the paper and the supporting methods of $\V$-graded categories that we develop along the way.  In \S \ref{sec:lr_vcats_vacts}, we review some background material on $\V$-enriched categories, $\V$-actegories, and Day convolution. In \S \ref{sec:graded_cats} we provide an introduction to (left and right) $\V$-graded categories, with examples including simplicially, topologically, and cubically graded categories. In \S \ref{sec:ch_base_grcats} we study a contravariant change of base process for graded categories along opmonoidal functors. In \S \ref{sec:bigraded}, we introduce $\V$-$\W$-bigraded categories. In \S \ref{sec:emb} we freely embed each $\V$-graded category into a $\V$-actegory, and we discuss the resulting formalism of commutative diagrams in $\V$-graded categories. In \S\ref{sec:bigr_sq} we develop supporting lemmas on graded commutative squares in bigraded categories. In \S \ref{sec:gr_func_cat} we define graded functor categories valued in bigraded categories. In \S \ref{sec:bifunctors} we define graded bifunctors and bigraded products, and we show that the latter represent the former. In \S \ref{sec:reln_gr_func_cat_bifunc} we establish results relating graded functor categories and graded bifunctors. In \S \ref{sec:enr_mod_psh_cocompl}, we recall background material on enriched modules/profunctors in the non-symmetric setting as well as Street's presheaf category and Yoneda lemma. In \S \ref{sec:vgr_mod_psh_yon} we discuss $\V$-graded modules, $\V$-graded presheaf categories, and the $\V$-graded Yoneda lemma. In \S \ref{sec:duoidal}, we treat the special case where $\V$ is normal duoidal and study the relation of the above concepts to the enriched functor categories and bifunctors of Garner and L\'opez Franco \cite{GaLf}. In an appendix (\S \ref{sec:weighted_colims}), we recall background material on free cocompletion for a class of weights in the non-symmetric setting, which we use only in \ref{para:copower_cocompl} (and, in turn, in the proof of \ref{thm:copower_cocompl_vgr}).

\medskip

\noindent\textit{Acknowledgments.} We thank the anonymous referee for comments leading to enhancements of the paper. We also thank Nathanael Arkor for suggestions and discussions on $\V$-graded categories. We thank Am\'elie Comtois and Richard Blute for various discussions on $\V$-graded categories.

\begin{para}[\textbf{Size conventions and notation}]\label{para:set_small}
One sufficient set-theoretic setting for this paper can be formulated in terms of (Grothendieck) universes. In terms of notation, we identify each universe with its associated category of sets.  We fix a universe $\SET$, and we call a set (resp.~a category) \textit{$\SET$-small} if it is an object of $\SET$ (resp.~a category internal to $\SET$). Correspondingly, we speak of \textit{$\SET$-small (co)limits}. In many applications, one will also have at hand a $\SET$-small universe $\Set$ of \textit{small sets}, but we employ $\Set$ only in discussing certain examples. We fix also a universe $\SET'$ for which $\SET$ is $\SET'$-small, and we refer to the objects of $\SET'$ as \textit{huge sets}. We use the unqualified term \textit{category} as a synonym for $\SET$-small category, unless otherwise specified, and we denote the 2-category of such categories by $\CAT$. We refer to categories internal to $\SET'$ as \textit{huge categories} and denote the 2-category of these by $\CAT'$. We call lax monoidal functors and op-lax monoidal functors simply \textit{monoidal functors} and \textit{opmonoidal functors} respectively. We use the unqualified term \textit{monoidal category} to mean $\SET$-small monoidal category. Monoidal (resp.~opmonoidal) functors between such monoidal categories are the 1-cells of a 2-category $\MCAT$ (resp.~$\MCAT_\oplax$) whose 2-cells are monoidal (resp. opmonoidal) transformations. We write $\MCAT_\strong$ for the locally full sub-2-category of $\MCAT$ spanned by the strong monoidal functors. We have analogously defined 2-categories $\MCAT',\MCAT'_\oplax,\MCAT'_\strong$ whose objects are huge monoidal categories. We write $\TWOCAT$ for the 2-category of 2-categories internal to $\SET'$, and henceforth we call the objects of $\TWOCAT$ simply \textit{2-categories}. 

We write $\cdot$ for composition in ordinary categories (in non-diagrammatic order). In a monoidal category $\V$, we write the monoidal product as $\otimes$, the unit object as $I$ or $I_\V$, and the associators and unitors as $a,\ell,r$.
\end{para}

\section{Background I: Enriched categories, actegories, and Day convolution}\label{sec:lr_vcats_vacts}

Let $\V$ be a monoidal category, allowing here that $\V$ may be huge (\ref{para:set_small}).

\begin{para}[\textbf{Left and right $\V$-categories}]\label{para:lr_vcats}
The classic papers of B\'enabou \cite{Ben:CatRel} and of Eilenberg and Kelly \cite{EiKe} employ the term \textit{$\V$-category} in two distinct but inter-definable senses, with the consequence that $\V$-categories in the sense of B\'enabou are precisely $\V^\rev$-categories in the sense of Eilenberg and Kelly, where $\V^\rev$ is the reverse of $\V$ (\S \ref{sec:intro}). In this paper, we frequently employ both notions, so we put them on equal footing by declaring that a \textbf{left $\V$-category} is a $\V$-category in the sense of Eilenberg and Kelly \cite{EiKe}, while a \textbf{right $\V$-category} is a $\V$-category in the sense of B\'enabou \cite{Ben:CatRel}, i.e.~a left $\V^\rev$-category. The distinction between these two notions is that a left $\V$-category $\C$ has composition morphisms of the form $m_{ABC}:\C(B,C) \otimes \C(A,B) \to \C(A,C)$ in $\V$ for each triple $A,B,C$ of objects of $\C$, while a right $\V$-category instead has composition morphisms of the form $m_{ABC}:\C(A,B) \otimes \C(B,C)  \to \C(A,C)$ in $\V$. The motivation for the choice of the ``left'' and ``right'' terminology will become apparent in \ref{para:modules_vact_vcat}, \ref{para:acts_as_vgr_cats}, \ref{para:right_vgr}, \ref{rem:rmod_in_vgr_cat}, \ref{defn:bigraded}--\ref{exa:bigr_str_hatv}, \ref{defn:env_act}, and \ref{para:cdiag_vgr_cats}. The distinction between the notions of left $\V$-category and right $\V$-category carries over straightforwardly to the more general setting of categories enriched in a bicategory $\X$, and one finds both conventions used in the literature: For example, Street \cite{Str:EnrCatsCoh} and Betti-Carboni-Street-Walters \cite{BCSW:VarEnr} employ right $\X$-categories, while Carboni-Kasangian-Walters \cite{CKW} and Gordon-Power \cite{GP,GP:GabrielUlmer} employ left $\X$-categories.

For convenience within passages involving only left $\V$-categories, we also adopt the term \textbf{$\V$-category} as a synonym for the notion of \textit{left} $\V$-category. We adopt corresponding terms pertaining to enriched functors: A \textbf{(left) $\V$-functor} is a $\V$-functor in the sense of Eilenberg and Kelly \cite{EiKe}, while a \textbf{right $\V$-functor} is a $\V$-functor in the sense of B\'enabou \cite{Ben:CatRel}, i.e.~a left $\V^\rev$-functor.  Similarly, we have the notions of \textbf{(left) $\V$-natural transformation} and \textbf{right $\V$-natural transformation}. Thus we have the 2-category $\LCAT{\V}$ of $\SET$-small left $\V$-categories and the 2-category $\RCAT{\V} = \LCAT{\V^\rev}$ of $\SET$-small right $\V$-categories; we write $\LCAT{\V}'$ (resp. $\textnormal{CAT}_\V'$) for the 2-category of huge left (resp. right) $\V$-categories. Henceforth we use the term \textit{(left) $\V$-category} to refer to $\SET$-small (left) $\V$-categories, and similarly for right $\V$-categories.

In the special case where $\V$ is symmetric monoidal, there is an isomorphism of 2-categories $\LCAT{\V} \cong \RCAT{\V}$ under which the notions of left and right $\V$-category may be identified. More generally, if $\V$ is braided monoidal, then the identity functor on $\V$ underlies a monoidal isomorphism $\V \cong \V^\rev$ \cite[Example 2.5]{JoyStr} with monoidal structure provided by the braiding; in general one can construct multiple such isomorphisms using the braiding, since the braiding is in general distinct from its inverse.

For an arbitrary monoidal category $\V$, one reason to consider both left and right $\V$-categories is the following observation, which appears in \cite[p.~32]{Wood:thesis}, \cite[\S 2.9]{KLSS:CatsEnrTwoSides}, \cite{GP:GabrielUlmer}: Every left $\V$-category $\C$ determines a right $\V$-category $\C^\circ$ that has the same objects as $\C$ but has hom-objects $\C^\circ(A,B) = \C(B,A)$ and composition morphisms $\C^\circ(A,B) \otimes \C^\circ(B,C) \to \C^\circ(A,C)$ defined as the composition morphisms $m_{CBA}:\C(B,A) \otimes \C(C,B) \to \C(C,A)$ for $\C$. We call the right $\V$-category $\C^\circ$ the \textbf{formal opposite} of the left $\V$-category $\C$. There is an isomorphism of 2-categories $(-)^\circ:\LCAT{\V}^\co \to \RCAT{\V}$. When $\V$ is symmetric monoidal, the opposite $\C^\op$ of $\C$ as defined in \cite[\S 1.4]{Ke:Ba} is the \textit{left} $\V$-category that corresponds to $\C^\circ$ under the isomorphism $\RCAT{\V} \cong \LCAT{\V}$. For example, for an arbitrary monoidal category $\V$, the distinction between one-object left $\V$-categories and one-object right $\V$-categories is purely formal, as both are just monoids in $\V$; given a monoid $R = (R,m,e)$ in $\V$, we adopt the convention of writing $R$ also to denote the left $\V$-category thus obtained; the formal opposite $R^\circ$ is then exactly the same monoid $(R,m,e)$, but regarded as a right $\V$-category; contrastingly, if $\V$ is symmetric then $R^\op$ is a genuinely different monoid $(R,m \cdot c_{RR},e)$, where $c_{RR}:R \otimes R \to R \otimes R$ is the symmetry.
\end{para}

\begin{para}[\textbf{Left and right $\V$-actegories}]\label{para:actegories}
A \textbf{(left) $\V$-actegory} is a category $\C$ equipped with a strong monoidal functor $\Phi:\V \to \CAT(\C,\C)$, written as a left action $X.A = (\Phi X)A$ $(X \in \V,A \in \C)$, while we say that the $\V$-actegory $\C$ is \textbf{strict} if $\Phi$ is strict monoidal; see \S \ref{sec:intro} regarding the origins of this concept and its name. A \textbf{right $\V$-actegory} is a category $\C$ equipped with a strong monoidal functor $\V^\rev \to \CAT(\C,\C)$, written as right action $A.X = (\Phi X)A$. Thus, right $\V$-actegories are equivalently left $\V^\rev$-actegories. Every monoidal category $\V$ carries the structure of a left $\V$-actegory and also a right $\V$-actegory, using the monoidal product as the action in each case.  

A left $\V$-actegory is equivalently given by a (strong) homomorphism of bicategories $\Sigma\V \to \CAT$, where $\Sigma\V$ is the suspension of $\V$ as a one-object bicategory.  A \textbf{lax morphism of $\V$-actegories} is then a lax transformation between such homomorphisms of bicategories and so, in concrete terms, is given by a functor $F:\C \to \D$ equipped with morphisms $\lambda_{XA}^F:X.FA \to F(X.A)$ natural in $X \in \V$, $A \in \C$, satisfying certain axioms \cite[(2.3)]{GP}.  We say that $F$ is \textbf{strong} if each $\lambda_{XA}^F$ is an isomorphism. By \cite[\S 2]{GP}, there is a 2-category $\LACT{\V}$ whose 1-cells are lax morphisms of $\V$-actegories, in which a 2-cell $\delta:F \Rightarrow G:\C \to \D$ is a natural transformation that commutes with the morphisms $\lambda_{XA}^F$, $\lambda_{XA}^G$ in the evident sense \cite[(2.4)]{GP}. We write $\LACT{\V}^\strong$ to denote the locally full sub-2-category of $\LACT{\V}$ spanned by the strong morphisms of left $\V$-actegories.

If a left $\V$-actegory $\C$ \textit{has hom-objects} in the sense that the (ordinary) functor $(-).A:\V \to \C$ has a right adjoint $\lrsub{\V}{$\shortunderline{\C}$}{}(A,-):\C \to \V$ for all $A \in \ob\C$, then the category $\C$ underlies a left $\V$-category $\lrsub{\V}{$\shortunderline{\C}$}{}$ with hom-objects $\lrsub{\V}{$\shortunderline{\C}$}{}(A,B)$ \cite[\S 2]{Jan:Actions}.  Similarly, every right $\V$-actegory $\C$ for which each $A.(-)$ has a right adjoint underlies a right $\V$-category $\shortunderline{\C}_\V$.

A monoidal category $\V$ is \textit{biclosed} \cite{Lam:DedSysII} if for each $X \in \ob\V$ the functors $X \otimes (-),(-) \otimes X:\V \to \V$ both have right adjoints, $(-)^X$ and $\tensor[^X]{(-)}{}$ respectively; thus $\V(Y,Z^X) \cong \V(X \otimes Y,Z) \cong \V(X,\tensor[^Y]{Z}{})$ naturally in $X,Y,Z \in \V$. We favour these notations $Z^X$ and $\tensor[^X]{Z}{}$ as they ensure that $Z^{X \otimes Y} \cong (Z^X)^Y$ and $\tensor[^{X \otimes Y}]{Z}{} \cong \tensor[^{\raisebox{-.5ex}{\ensuremath{\scriptstyle X}}}]{{(\tensor[^Y]{Z}{})}}{}$ naturally in $X,Y,Z \in \V$. If $\V$ is biclosed then, by the previous paragraph, $\V$ underlies a left $\V$-category $\lrsub{\V}{$\shortunderline{\V}$}{}$ with $\lrsub{\V}{$\shortunderline{\V}$}{}(X,Y) = \tensor[^X]{Y}{}$ and also a right $\V$-category $\shortunderline{\V}_\V$ with $\shortunderline{\V}_\V(X,Y) = Y^X$.
\end{para}

\begin{exa}\label{exa:actegories}\emptybox
\noindent (1). Write $\Delta$ for the \textit{augmented simplex category}, i.e.~the category of finite ordinals with monotone maps. We may regard $\Delta$ as a (non-symmetric) strict monoidal category with monoidal product $+$ given on objects by addition.  This monoidal category $\Delta$ has a well-known universal property that is discussed in \cite[VII.5]{MacL} and entails that a strict left $\Delta$-actegory is equivalently a category $\C$ equipped with a \textit{monad} $\TT = (T,\mu,\eta)$, noting that the $\Delta$-action is then given on objects by $n.A = T^n A$ $(n \in \NN = \ob\Delta)$.

\medskip

\noindent (2). Write $\Top$ for the category of (small) topological spaces, and let $\sS$ be any full subcategory of $\Top$ that has finite products that are preserved by the inclusion $\sS \hookrightarrow \Top$.  Then we may regard $\sS$ as a cartesian monoidal category and regard $\Top$ as a left $\sS$-actegory under the action given by $X.A = X \times A$ $(X \in \sS, A \in \Top)$.

\medskip

\noindent (3). Write $\Box$ for the \textit{restricted cubical site} \cite{GrMau:Cubical}. Explicitly, if we write $\Set_2$ for the full subcategory of $\Set$ spanned by the finite powers $2^n$ $(n \in \NN)$ of $2 = \{0,1\}$, then $\Set_2$ is a strict monoidal category, and $\Box$ is the smallest (non-full) subcategory of $\Set_2$ that is closed under the monoidal product and contains the unique map $\varepsilon:2 \to 1 = 2^0$ and the maps $0,1:1 \rightrightarrows 2$ that pick out the elements $0,1 \in 2$. With this notation, $\Box$ is a (non-symmetric) strict monoidal category, and by \cite[10.4]{GrMau:Cubical} a left $\Box$-actegory  is equivalently a category $\C$ equipped with a \textit{cylinder functor} $T:\C \to \C$ \cite[I.3.1]{Bau}, i.e.~an endofunctor $T$ equipped with natural transformations $\alpha:T \to 1_\C$ and $\xi_0,\xi_1:1_\C \rightrightarrows T$ such that $\alpha \cdot \xi_0 = 1 = \alpha \cdot \xi_1$. For example, $[0,1] \times (-)$ is a cylinder functor on $\Top$.
\end{exa}

\begin{para}\label{para:tensors}
Given an object $X$ of a monoidal category $\V$ and an object $A$ of a left $\V$-category $\C$, a \textbf{copower} (or \textbf{tensor}) of $A$ by $X$ is an object $X \cdot A$ equipped with a morphism $\upsilon:X \to \C(A,X \cdot A)$ in $\V$ such that for all $B \in \ob\C$ and $Y \in \ob\V$ the function $\V(Y,\C(X \cdot A,B)) \to \V(Y \otimes X,\C(A,B))$ given by $f \mapsto m_{A,X \cdot A,B} \cdot (f \otimes \upsilon)$ is a bijection \cite[Def.~3]{Pare:Mealy}. If $\V$ is biclosed, then the preceding condition on $\upsilon$ requires equivalently that $\tilde{\upsilon}_B:\C(X \cdot A,B) \to \tensor[^X]{\C(A,B)}{}$ be an isomorphism in $\V$ for each $B \in \ob\C$, where $\tilde{\upsilon}_B$ is the transpose of $m_{A,X\cdot A,B} \cdot (1 \otimes \upsilon):\C(X \cdot A,B) \otimes X \to \C(A,B)$, and thus we recover an instance of the notion of tensor in \cite[\S 3]{GP}, which generalizes \cite[(3.44)]{Ke:Ba}.
\end{para}

\begin{para}\label{para:modules_vact_vcat}
Given a monoid $R$ in $\V$, a \textbf{left $R$-module} (or \textbf{left $R$-act}) in a left $\V$-actegory $\C$ is an object $A$ of $\C$ equipped with an associative and unital left $R$-action $R.A \to A$ in $\C$.  On the other hand, a \textbf{left $R$-module} in a left $\V$-\textit{category} is instead a left $\V$-functor $R \to \C$, where we regard $R$ as a one-object left $\V$-category (\ref{para:lr_vcats}). We observe in \ref{rem:rmod_in_vgr_cat} and \ref{para:cdiag_vgr_cats} that these notions of left $R$-module admit a common generalization in left $\V$-graded categories.  Analogously, one can define notions of right $R$-module in right $\V$-actegories and right $\V$-categories (in the latter case regarding $R$ as a right $\V$-category $R^\circ$ as in \ref{para:lr_vcats}), and more generally in right $\V$-graded categories (\ref{rem:rmod_in_vgr_cat}).
\end{para}

\begin{para}[\textbf{Day convolution}]\label{sec:day_conv}
In the remainder of the paper, we assume $\V$ is a $\SET$-small monoidal category (\ref{para:set_small}). We write $\hat{\V} := \hat{\V}(\SET) := [\V^\op,\SET]$, recalling that $\hat{\V}$ carries the structure of a (huge) biclosed monoidal category, via Day convolution \cite{Day:ClCatsFunc}, in such a way that the Yoneda embedding $\mathsf{Y}:\V \to \hat{\V}$ is strong monoidal.  Explicitly, the monoidal product carried by $\hat{\V}$ is written as $\otimes$ and given by
$$(P \otimes Q)X = \int^{Y,Z \in \V} \V(X,Y \otimes Z) \times PY \times QZ\;\;\;\;\;\;\;\;\;\;(P,Q \in \hat{\V}, X \in \V),$$
while the unit object is $\mathsf{Y}(I) = \V(-,I)$. Consequently, there is a bijective correspondence, natural in $P,Q,R \in \hat{\V}$, between morphisms $\theta:P \otimes Q \to R$ in $\hat{\V}$ and families of maps $\theta^{XY}:PX \times QY \to R(X \otimes Y)$ natural in $X,Y \in \V$, under which the identity on $P \otimes Q$ corresponds to a natural family $\kappa_{PQ}^{XY}:PX \times QY \to (P \otimes Q)(X \otimes Y)$. The internal homs in $\hat{\V}$ are given by $Q^PX = \hat{\V}(P,Q(-\otimes X))$ and $\tensor[^P]{Q}{}X = \hat{\V}(P,Q(X \otimes -))$, naturally in $P,Q \in \hat{\V}$ and $X \in \V$. Consequently $Q^{\mathsf{Y}(X)} \cong Q(X\otimes -)$ and $\tensor[^{\mathsf{Y}(X)}]{Q}{} \cong Q(- \otimes X)$ naturally in $Q \in \hat{\V}$, $X \in \V$, so we write $Q^X := Q(X \otimes-)$ and $\tensor[^X]{Q}{} := Q(- \otimes X)$. The identity functor on $\hat{\V}$ underlies an isomorphism of monoidal categories $\widehat{\V^\rev} \cong \hat{\V}^\rev$.

Under the Yoneda bijections $\hat{\V}(\mathsf{Y}X,P) \cong PX$ ($X \in \V, P \in \hat{\V}$), each $p \in PX$ corresponds to a morphism in $\hat{\V}$ that we denote by $\tilde{p}:\mathsf{Y}X \to P$. Given also a morphism $\alpha:Y \to X$ in $\V$, we write $\alpha^* =  P\alpha:PX \to PY$, and for each $p \in PX$ we call $q = \alpha^*(p) \in PY$ the \textit{reindexing of $p$ along $\alpha$}, noting that $\tilde{q} = \tilde{p} \cdot \mathsf{Y}\alpha$.

By \cite[2.1]{BooSt}, $\V(-,?):\V^\op \times \V \to \SET$ is a (lax) monoidal functor when equipped with the maps $\otimes:\V(X,X') \times \V(Y,Y') \to \V(X \otimes Y, X' \otimes Y')$, and similarly for $\hat{\V}(-,?)$. The strong monoidal functor $\mathsf{Y}:\V \to \hat{\V}$ may be regarded also as a strong opmonoidal functor, whose binary opmonoidal constraints are isomorphisms $d_{XY}:\mathsf{Y}(X \otimes Y) \to \mathsf{Y}X \otimes \mathsf{Y}Y$. Consequently, the functor $\hat{\V}(\mathsf{Y}-,?):\V^\op \times \hat{\V} \to \SET$ carries a monoidal structure. By the monoidal Yoneda lemma of \cite[p.~269]{Kou:AugVDblCats}, the Yoneda bijections $\hat{\V}(\mathsf{Y}X,P) \cong PX$ ($X \in \V, P \in \hat{\V}$) constitute an invertible monoidal transformation $\hat{\V}(\mathsf{Y}-,?) \cong \Ev:\V^\op \times \hat{\V} \to \SET$ when we equip the evaluation functor $\Ev$ with the monoidal structure given by the above maps $\kappa = \kappa_{PQ}^{XY}:PX \times QY \to (P \otimes Q)(X \otimes Y)$. In particular, the Yoneda isomorphisms commute with the binary monoidal constraints, which in detail means that for all $p \in PX$ and $q \in QY$, $\kappa(p,q) \in (P \otimes Q)(X \otimes Y)$ corresponds to the composite $\mathsf{Y}(X \otimes Y) \xrightarrow{d_{XY}} \mathsf{Y}X \otimes \mathsf{Y}Y \xrightarrow{\tilde{p} \otimes \tilde{q}} P \otimes Q$. Hence if $\theta:P \otimes Q \to R$ in $\hat{\V}$ then the element $\theta^{XY}(p,q) \in R(X \otimes Y)$ corresponds via Yoneda to the composite $\mathsf{Y}(X \otimes Y) \xrightarrow{d_{XY}} \mathsf{Y}X \otimes \mathsf{Y}Y \xrightarrow{\tilde{p} \otimes \tilde{q}} P \otimes Q \xrightarrow{\theta} R$.

If $\V$ is symmetric monoidal (resp. braided monoidal) then $\hat{\V}$ is so, and the Yoneda embedding $\mathsf{Y}:\V \to \hat{\V}$ is a braided strong monoidal functor \cite[Theorem 3.6]{Day:ClCatsFunc} (resp.~\cite[Proposition 1.1]{DayPanStr}). Similarly, if $\V$ is a \textit{duoidal category} then $\hat{\V}$ is a duoidal category \cite[\S 4]{BooSt}; see \ref{exa:duoidal} and \ref{para:duoidal_str_vhat}.

Since the given $\SET$-small monoidal category $\V$ is, in particular, $\SET'$-small, we may equip $\hat{\V}(\SET') = [\V^\op,\SET']$ with its Day convolution monoidal structure, so that $\hat{\V}(\SET')$ is a monoidal category (but is not even $\SET'$-small). The inclusion $\SET \hookrightarrow \SET'$ preserves $\SET$-small limits and colimits and so induces a fully faithful, strong monoidal functor $\hat{\V}(\SET) \hookrightarrow \hat{\V}(\SET')$ that preserves all $\SET$-small limits and colimits.
\end{para}

\begin{para}\label{para:copower_cocompl}
We now recall results on free cocompletion under copowers (\ref{para:tensors}) that we use only in the proof of \ref{thm:copower_cocompl_vgr}, where we employ the free cocompletion of a $\hat{\V}$-category with respect to copowers by representables. These results are special cases of known results on free cocompletion with respect to a class of weights, which we recall in an appendix (\S \ref{sec:weighted_colims}). Let $\sfV$ be a \textbf{huge biclosed base}, by which we mean a huge biclosed monoidal category that is locally $\SET$-small and has $\SET$-small limits and colimits (such as $\hat{\V}$). Also, let $\R$ be a $\SET$-small full subcategory of $\sfV$. Given $\sfV$-categories $\D$ and $\E$ with $\R$-copowers (i.e.~copowers by objects of $\R$), let us write $\textnormal{$\R$-}\mathsf{COCTS}(\D,\E)$ for the full subcategory of $\LCAT{\sfV}(\D,\E)$ spanned by the $\sfV$-functors that preserve $\R$-copowers. Given a ($\SET$-small left) $\sfV$-category $\C$, a \textit{free cocompletion of $\C$ under $\R$-copowers} is by definition a $\sfV$-category $\R(\C)$ that has $\R$-copowers and is equipped with a $\sfV$-functor $K:\C \to \R(\C)$ such that for each $\sfV$-category $\D$ with $\R$-copowers the functor $(-) \circ K:\textnormal{$\R$-}\mathsf{COCTS}(\R(\C),\D) \to \LCAT{\sfV}(\C,\D)$ is an equivalence. As a special case of free cocompletion with respect to a class of weights, which we recall in \S \ref{sec:weighted_colims}, a free cocompletion $\R(\C)$ of $\C$ under $\R$-copowers exists and is $\SET$-small; furthermore, the $\sfV$-functor $K:\C \to \R(\C)$ is fully faithful, and the closure of the full subcategory $\K = \{KA \mid A \in \ob\C\} \hookrightarrow \R(\C)$ under $\R$-copowers is $\R(\C)$ itself (i.e.~no proper full subcategory of $\R(\C)$ contains $\K$ and is closed under $\R$-copowers).
\end{para}

\section{Left and right graded categories: An introduction}\label{sec:graded_cats}

Let $\V$ be a monoidal category (assumed $\SET$-small, but not necessarily small, \ref{para:set_small}).

\begin{defn}\label{para:graded_cats}
A \textbf{(left) \mbox{$\V$-graded} category} is, by definition, a (left) $\hat{\V}$-category, where $\hat{\V} = [\V^\op,\SET]$ is equipped with the Day convolution monoidal structure (\ref{sec:day_conv}).  Hence we call $\LGCAT{\V} := \LCAT{\hat{\V}}$ the \textit{2-category of (left) $\V$-graded categories}.  See \S\ref{sec:intro} regarding the origins of this concept and its name.
\end{defn}

In view of the discussion in \ref{sec:day_conv}, a (left) $\V$-graded category is equivalently given by the following data, as established originally in \cite[\S 2]{Wood:thesis}: (1) a $\SET$-small set $\ob\C$; (2) an assignment to each pair $A,B \in \ob\C$ a functor $\C(A,B):\V^\op \to \SET$; (3) an assignment to each triple $A,B,C \in \ob\C$ a family of maps $\circ_{ABC}^{YX}:\C(B,C)(Y) \times \C(A,B)(X) \to \C(A,C)(Y \otimes X)$ natural in $Y,X \in \V$; (4) an assignment to each $A \in \ob\C$ an element $\gid_A \in \C(A,A)(I)$; these data are required to satisfy associativity and identity axioms, which we state in an explicit, elementwise form in \ref{para:vgr_cat_conc}(III, IV). But first we introduce some notation:

\begin{para}\label{para:notn_gr}
For reasons that will become apparent when we introduce a form of commutative diagram for $\V$-graded categories in \S \ref{sec:emb}, we introduce the following new notation in (left) $\V$-graded categories $\C$: Given objects $X \in \ob\V$ and $A,B \in \ob\C$, we write
$$\C(X \gc A; B) := \C(A,B)(X)\;,$$
and we write 
$$f:X \gc A \longrightarrow B$$
to mean that $f \in \C(X \gc A;B)$, which we convey by saying that $f$ is a \textbf{graded morphism} from $A$ to $B$ with grade $X$; correspondingly we write $\C(- \gc A;B) := \C(A,B)$.  Accordingly, we regard a graded morphism $f$ as a morphism with two inputs, one whose type is an object $X$ of $\V$ and one whose type is an object $A$ of $\C$.
\end{para}

\begin{para}[\textbf{$\V$-graded categories, concretely}]\label{para:vgr_cat_conc}
A (left) $\V$-graded category $\C$ consists of
\begin{enumerate}
\item[(a)] a $\SET$-small set $\ob\C$ whose elements we call \textit{objects};
\item[(b)] for each pair $A,B \in \ob\C$ and each $X \in \ob\V$ a $\SET$-small set $\C(X \gc A;B)$ whose elements we write as $f:X \gc A \to B$ and call \textit{graded morphisms from $A$ to $B$ with grade $X$};
\item[(c)] an assignment to each graded morphism $f:X \gc A \to B$ in $\C$ and each morphism $\alpha:Y \to X$ in $\V$ a graded morphism $\alpha^*(f):Y \gc A \to B$ called the \textit{reindexing of $f$ along $\alpha$};
\item[(d)] an assignment to each pair of graded morphisms $f:X \gc A \to B$, $g:Y \gc B \to C$ a graded morphism $g \circ f:Y \otimes X \gc A \to C$;
\item[(e)] for each $A \in \ob\C$ a graded morphism $\gid_A:I \gc A \to A$ whose grade is the unit object $I$ of $\V$;
\end{enumerate}
these data are required to satisfy the following axioms:
\begin{description}
\item[(I) Functoriality of reindexing.] $1_X^*(f) = f:X \gc A \to B$ and $(\beta \cdot \alpha)^*(f) = \alpha^*(\beta^*(f))$ $:Z \gc A \to B$ for every graded morphism $f:X \gc A \to B$ and every pair of morphisms $\alpha:Z \to Y$, $\beta:Y \to X$ in $\V$;
\item[(II) Naturality of composition.] $\beta^*(g) \circ \alpha^*(f) = (\beta \otimes \alpha)^*(g \circ f):Y'\otimes X' \gc A \to C$ for all graded morphisms $f:X \gc A \to B$, $g:Y \gc B \to C$ and all morphisms $\alpha:X' \to X$ and $\beta:Y' \to Y$ in $\V$, where $\beta \otimes \alpha:Y' \otimes X' \to Y \otimes X$ is obtained by applying the monoidal product in $\V$;
\item[(III) Essential associativity.] For all graded morphisms $f:X \gc A \to B$,  $g:Y \gc B \to C$,  $h:Z \gc C \to D$, the graded morphism $(h \circ g) \circ f:(Z \otimes Y) \otimes X \gc A \to D$ is the reindexing of $h \circ (g \circ f):Z \otimes (Y \otimes X) \gc A \to D$ along the associator $a_{ZYX}:(Z \otimes Y) \otimes X \xrightarrow{\sim} Z \otimes (Y \otimes X)$;
\item[(IV) Essential identity.] For every graded morphism $f:X \gc A \to B$, the composite $f \circ \gid_A:X \otimes I \gc A \to B$ is the reindexing of $f$ along the right unitor $r_X:X \otimes I \xrightarrow{\sim} X$, and $\gid_B \circ f:I \otimes X \gc A \to B$ is the reindexing of $f$ along the left unitor $\ell_X:I \otimes X \xrightarrow{\sim} X$.
\end{description}

Note that (II) requires precisely that the composition maps
$$\circ_{ABC}^{YX}:\C(Y \gc B;C) \times \C(X \gc A;B) \to \C(Y \otimes X \gc A;C)$$
be natural in $Y,X \in \V$ for all $A,B,C \in \ob\C$. When $\V$ is strict monoidal, the essential associativity and identity axioms (III, IV) reduce to \textit{strict} associativity and identity axioms, requiring simply the equations $(h \circ g) \circ f = h \circ (g \circ f):Z \otimes Y \otimes X \gc A \to D$ and $f \circ \gid_A = f = \gid_B \circ f$.
\end{para}

For example, $1$-graded categories for the  the terminal monoidal category $1$ (the ordinal $1$) are precisely ($\SET$-small) ordinary categories. As another simple example, regard the discrete category of natural numbers $\NN$ as a monoidal category under addition; we invite the reader to characterize $\NN$-graded categories.

\begin{para}[\textbf{On size}]
We refer to the $\V$-graded categories of Definition \ref{para:graded_cats} as \textit{$\SET$-small $\V$-graded categories}, as they are $\SET$-small $\hat{\V}(\SET)$-categories. By definition, a \textbf{huge $\V$-graded category} is a $\SET'$-small $\hat{\V}(\SET')$-category with the notation of \ref{sec:day_conv}. We write $\LGCAT{\V}' = \LCAT{\hat{\V}}'$ for the 2-category of huge $\V$-graded categories.
\end{para}

\begin{para}\label{para:vgr_func}
Given $\V$-graded categories $\C$ and $\D$, a \textbf{(left) $\V$-graded functor} $F:\C \to \D$ is, by definition, a (left) $\hat{\V}$-functor and so is given by (1) an assignment to each object $A$ of $\C$ an object $FA$ of $\D$, and (2) a family of maps $F_{AB}^X:\C(X \gc A;B) \to \D(X \gc FA;FB)$ $(A,B \in \ob\C, X \in \ob\V)$, i.e.~an assignment to each graded morphism $f:X \gc A \to B$ in $\C$ a graded morphism $Ff:X \gc FA \to FB$ in $\D$, such that (i) the maps in (2) are natural in $X \in \V$, (ii) $F(g \circ f) = Fg \circ Ff:Y\otimes X \gc FA \to FC$ and $F\gid_A = \gid_{FA}:I \gc FA \to FA$ for all $f:X \gc A \to B$ and $g:Y \gc B \to C$ in $\C$.  These axioms (i) and (ii) require precisely that the assignment in (2) preserves reindexing, composition, and identities.
\end{para}

\begin{para}\label{para:underl_ord_cat}
The \textbf{underlying ordinary category} of a $\V$-graded category $\C$ is the usual underlying ordinary category $\C_0$ of the $\hat{\V}$-category $\C$.  A morphism from $A$ to $B$ in $\C_0$ is given by a graded morphism from $A$ to $B$ whose grade is $I$.
\end{para}

\begin{para}\label{para:vgr_transf}
Given $\V$-graded functors $F,G:\C \to \D$, a \textbf{(left) $\V$-graded-natural transformation} $\delta:F \Rightarrow G$ is a $\hat{\V}$-natural transformation, equivalently, a family of morphisms $\delta_A:FA \to GA$ $(A \in \ob\C)$ in $\D_0$ that are \textbf{(left) $\V$-graded natural in $A \in \C$} in the sense that for every graded morphism $f:X \gc A \to B$ in $\C$, the reindexing of $\delta_B \circ Ff: I \otimes X \gc FA \to GB$ along $\ell_X^{-1}:X \xrightarrow{\sim} I \otimes X$ is equal to the reindexing of $Gf \circ \delta_A:X \otimes I \gc FA \to GB$ along $r_X^{-1}:X \xrightarrow{\sim} X \otimes I$.  When $\V$ is strict monoidal, this amounts to the equation $\delta_B \circ Ff = Gf \circ \delta_A:X \gc FA \to GB$.
\end{para}

\begin{exa}[\textbf{$\V$-enriched categories as $\V$-graded categories}]\label{para:vcats_as_vgr_cats}
It has been known since \cite{Wood:thesis} that $\V$-enriched categories may be described equivalently as $\V$-graded categories $\C$ whose hom-objects $\C(A,B):\V^\op \to \SET$ are representable.  Indeed, by the usual change of base process for enriched categories \cite{EiKe}, the fully faithful, strong monoidal functor $\mathsf{Y}:\V \to \hat{\V}$ determines a fully faithful 2-functor $\mathsf{Y}_*:\LCAT{\V} \to \LGCAT{\V}$ that sends each (left) $\V$-category $\C$ to a $\V$-graded category $\mathsf{Y}_*(\C)$ whose hom-objects are the representable presheaves $\mathsf{Y}\C(A,B) = \V(-,\C(A,B)):\V^\op \to \SET$, so that a graded morphism $f:X \gc A \to B$ in $\mathsf{Y}_*(\C)$ is given by a morphism $f:X \to \C(A,B)$ in $\V$. By abuse of notation, we denote the resulting $\V$-graded category $\mathsf{Y}_*(\C)$ also by $\C$, so that we may write
$$\C(X \gc A;B) = \V(X,\C(A,B))\;\;\;\;(X \in \V).$$
We thus obtain a (strict) 2-equivalence $\LCAT{\V} \simeq \LGCAT{\V}^\mathrm{lrep}$, where the latter is the full sub-2-category of $\LGCAT{\V}$ spanned by those $\V$-graded categories that are \textit{locally representable} in the sense that their hom-objects are representable.
\end{exa}

\begin{exa}[\textbf{$\V$-actegories as $\V$-graded categories}]\label{para:acts_as_vgr_cats}
Every left $\V$-actegory $\C$ may be regarded as a (left) $\V$-graded category via the formula
$$\C(X \gc A;B) = \C(X.A,B)\;\;\;\;(X \in \V),$$
so that a graded morphism $f:X \gc A \to B$ in $\C$ is by definition a morphism $f:X.A \to B$ in $\C$, and the reindexing of $f$ along a morphism $\alpha:X' \to X$ in $\V$ is the composite $f \cdot (\alpha.A):X'.A \to B$ in the actegory $\C$.  Given graded morphisms $f:X \gc A \to B$ and $g:Y \gc B \to C$, the composite $g \circ f:Y \otimes X \gc A \to C$ in the $\V$-graded category $\C$ is the composite $Y \otimes X.A \xrightarrow{\sim} Y.X.A \xrightarrow{Y.f} Y.B \xrightarrow{g} C$ in the actegory $\C$, while $\gid_A:I.A \xrightarrow{\sim} A$ is the unit constraint.

As has been known since \cite{Wood:thesis}, the $\V$-graded categories that arise in the above way from left $\V$-actegories are, up to isomorphism, those that have \textit{$\V$-copowers}, in the following sense.  Given an object $A$ of a $\V$-graded category $\C$ and an object $X$ of $\V$, a \textbf{copower} (or \textbf{tensor}) of $A$ by $X$ is, by definition, a copower of $A$ by the representable $\mathsf{Y}X = \V(-,X) \in \hat{\V}$ in the left $\hat{\V}$-category $\C$ (\ref{para:tensors}).  Equivalently, a copower of $A$ by $X$ is an object $X \cdot A$ of $\C$ equipped with a graded morphism $\upsilon:X \gc A \to X \cdot A$ in $\C$ such that for all $B \in \ob\C$ and $Y \in \ob\V$ the map
$$\C(Y \gc X \cdot A;B) \to \C(Y \otimes X \gc A;B)$$
given by $f \mapsto f \circ \upsilon$ is a bijection. We also refer to copowers by objects of $\V$ as \mbox{\textbf{$\V$-copowers}}. For example, by \ref{para:vcats_as_vgr_cats}, locally representable $\V$-graded categories with $\V$-copowers are equivalently $\V$-categories with copowers (\ref{para:tensors}). We write $\LGCAT{\V}^\bigdot$ for the full sub-2-category of $\LGCAT{\V}$ spanned by the $\V$-graded categories with $\V$-copowers. By \cite[Theorem 3.4]{GP}, there is a fully faithful 2-functor $\LGCAT{\V}^\bigdot \to \LACT{\V}$ that sends each object $\C$ of $\LGCAT{\V}^\bigdot$ to $\C_0$ equipped with an action of $\V$ given on objects by $X.A = X \cdot A$.  But one finds that this 2-functor is also essentially surjective on objects by regarding each left $\V$-actegory as a $\V$-graded category as above.

This shows\footnote{Our proof is a variation on that of \cite[Proposition 14]{Ga}.} that there is a (strict) 2-equivalence $\LACT{\V} \simeq \LGCAT{\V}^\bigdot$. Also, by \cite[Corollary 3.5]{GP}, there is a 2-equivalence between $\LACT{\V}^\strong$ and the locally full sub-2-category $\LGCAT{\V}^{\bigdot\textnormal{pres}}$ of $\LGCAT{\V}^\bigdot$ spanned by those $\V$-graded functors that preserve $\V$-copowers.

Thus regarding each $\V$-actegory as a $\V$-graded category, we can also regard every \textit{full subcategory} of a $\V$-actegory as a $\V$-graded category. Note that a $\V$-actegory $\C$ has hom-objects (\ref{para:actegories}) iff $\C$ is locally representable as a $\V$-graded category (\ref{para:vcats_as_vgr_cats}).
\end{exa}

\begin{exa}[\textbf{$\V$ is a $\V$-graded category}]\label{exa:v_vgr}
The category $\V$ itself carries the structure of a left $\V$-actegory via the formula $X.A = X \otimes A$, so $\V$ underlies a left $\V$-graded category, whose homs are given by $\V(X \gc A;B) = \V(X \otimes A, B)$ $(X,A,B \in \ob\V)$.
\end{exa}

Henceforth we tacitly regard left $\V$-categories and left $\V$-actegories as left $\V$-graded categories, via \ref{para:vcats_as_vgr_cats} and \ref{para:acts_as_vgr_cats}.

\begin{exa}[\textbf{Simplicially, topologically, and cubically graded categories}]\label{exa:gr_cats}\emptybox
\noindent (1). Writing $\Delta$ for the augmented simplex category (\ref{exa:actegories}), left $\Delta$-graded categories are categories enriched in the category $\hat{\Delta}$ of \textit{augmented simplicial sets} (with the \textit{join product} \cite{EhlPor}) and can be described concretely via graded morphisms $f:n \gc A \to B$ with $n \in \NN$. By \ref{exa:actegories} and \ref{para:acts_as_vgr_cats}, every full subcategory of a category equipped with a monad $\TT$ underlies a $\Delta$-graded category in which a graded morphism $f:n \gc A \to B$ is a morphism $f:T^n A \to B$.

\medskip

\noindent (2). Let $\sS$ be any full subcategory of $\Top$ that has finite products that are preserved by the inclusion $\sS \hookrightarrow \Top$. By \ref{exa:actegories} and \ref{para:acts_as_vgr_cats}, every full subcategory $\C$ of $\Top$ underlies an $\sS$-graded category in which a graded morphism $f:X \gc A \to B$ is a continuous map $f:X \times A \to B$ with $X \in \ob\sS$ and $A,B \in \ob\C$.

\medskip

\noindent (3). Writing $\Box$ for the restricted cubical site (\ref{exa:actegories}), $\Box$-graded categories are categories enriched in the monoidal category $\hat{\Box}$ of \textit{(ordinary) cubical sets} \cite{GrMau:Cubical} (with its Day convolution product) and can be described in terms of graded morphisms $f:2^n \gc A \to B$ with $n \in \NN$. By \ref{exa:actegories} and \ref{para:acts_as_vgr_cats}, every full subcategory $\C$ of a category equipped with a cylinder functor $\TT$ underlies a $\Box$-graded category in which a graded morphism $f:2^n \gc A \to B$ is a morphism $f:T^n A \to B$. Cylinder functors provide an abstract setting for aspects of homotopy theory \cite{Bau}, and the notion of homotopy relative to a cylinder functor generalizes to the setting of $\Box$-graded categories $\C$ by declaring that for morphisms $f,g:A \to B$ in $\C_0$ a \textit{homotopy} from $f$ to $g$ is a graded morphism $h:2 \gc A \to B$ in $\C$ with $h_0 = f$ and $h_1 = g$, where we write $h_0 = 0^*(h)$ and $h_1 = 1^*(h)$ for the reindexings of $h$ along $0,1:1 \rightrightarrows 2$.  For example, in view of \ref{exa:actegories}, every full subcategory $\C$ of $\Top$ underlies a $\Box$-graded category in which the notion of homotopy is the usual one.
\end{exa}

\begin{para}\label{para:vgr_subcat}
Let $\G$ be a set of graded morphisms in a $\V$-graded category $\C$, i.e.~a subset of $\coprod_{X \in \ob\V,A,B \in \ob\C} \C(X \gc A;B)$, and write $\ob\G$ for the set of all $G \in \ob\C$ for which there exists some $f:X \gc A \to B$ in $\G$ with $G \in \{A,B\}$. We say that $\G$ is a \textit{$\V$-graded subcategory} of $\C$ if $\G$ is closed under composition and reindexing in $\C$ and contains $\gid_G$ for each $G \in \ob\G$. Every $\V$-graded subcategory of $\C$ may be regarded as $\V$-graded category; with this identification, the largest $\V$-graded subcategory of $\C$ is $\C$ itself. There is a smallest $\V$-graded subcategory $\langle \G \rangle$ of $\C$ that contains the given set $\G$. We say that $\G$ is a \textit{generating set of graded morphisms} of $\C$ if $\langle \G \rangle = \C$.
\end{para}

\begin{exa}\label{exa:loc_rep_univ_elts_gen}
Let $\C$ be a $\V$-category, regarded also as a $\V$-graded category. Then each hom-object $\C(A,B) \in \ob\V$ represents the presheaf $\C(-\gc A;B):\V^\op \to \SET$, and the counit of the representation is a graded morphism $u_{AB}:\C(A,B)\gc A \to B$. The set $\{u_{AB} \mid A,B \in \ob\C\}$ is a generating set of graded morphisms for $\C$.
\end{exa}

\begin{exa}\label{exa:two_sub_x}
Every object $P$ of $\sfV = \hat{\V}$ determines a $\sfV$-category $\tensor[_P]{\mathbbm{2}}{}$ with object set $2 = \{0,1\}$ and hom-objects $\tensor[_P]{\mathbbm{2}}{}(0,1) = P$, $\tensor[_P]{\mathbbm{2}}{}(0,0) = \mathsf{Y}(I) = \tensor[_P]{\mathbbm{2}}{}(1,1)$ (the unit object of $\sfV$), and $\tensor[_P]{\mathbbm{2}}{}(1,0) = 0$ (the initial object of $\sfV$). In particular, given an object $X$ of $\V$, we obtain a $\V$-graded category $\tensor[_X]{\mathbbm{2}}{} := \tensor[_{\mathsf{Y}(X)}]{\mathbbm{2}}{}$. The universal element of $\mathsf{Y}(X) = \V(-,X)$ is a graded morphism $u:X \gc 0 \to 1$ in $\tensor[_X]{\mathbbm{2}}{}$. Every graded morphism in $\tensor[_X]{\mathbbm{2}}{}$ is a reindexing of either $u$, $\gid_0$, or $\gid_1$ along a unique morphism in $\V$, and $\{u\}$ is a generating set of graded morphisms for $\tensor[_X]{\mathbbm{2}}{}$ (\ref{para:vgr_subcat}). Given any $\V$-graded category $\C$, there is a bijective correspondence between $\V$-graded functors $F:\tensor[_X]{\mathbbm{2}}{} \to \C$ and graded morphisms $f:X \gc A \to B$ in $\C$, as a consequence of the Yoneda lemma.
\end{exa}

\begin{defn}\label{defn:right_vgr_cat}
A \textbf{right $\V$-graded category} is a left $\V^\rev$-graded category.  Since $\widehat{\V^\rev} \cong \hat{\V}^\rev$ by \ref{sec:day_conv}, right $\V$-graded categories may be identified with right $\hat{\V}$-categories. We write $\RGCAT{\V} = \LGCAT{\V^\rev}$ to denote the 2-category of right $\V$-graded categories, and we call its 1-cells and 2-cells \textit{right $\V$-graded functors} and \textit{right $\V$-graded-natural transformations}, respectively.
\end{defn}

\begin{para}\label{para:right_vgr}
Given objects $A$ and $B$ of a right $\V$-graded category $\C$, we write the associated functor $\C(A,B):\V^\op \to \SET$ as $\C(A \gc -;B)$, and correspondingly we write
$$f\;:\;A \gc X \longrightarrow B$$
to mean that $f \in \C(A \gc X;B)$.  If $f:A \gc X \to B$ and $g:B \gc Y \to C$ in $\C$, then the composite $g \circ f$ is a graded morphism $g \circ f:A \gc X \otimes Y \to C$. Every right $\V$-category $\C$ (i.e.~every $\V^\mathsf{rev}$-category, \ref{para:lr_vcats}) can be regarded as a right $\V$-graded category, by \ref{para:vcats_as_vgr_cats}. Similarly, every right $\V$-actegory $\C$ (i.e.~every $\V^\mathsf{rev}$-actegory) can be regarded as a right $\V$-graded category, by \ref{para:acts_as_vgr_cats}.  The monoidal category $\V$ itself carries the structure of a right $\V$-actegory and so may be regarded as a right $\V$-graded category.

By \ref{para:lr_vcats}, every left $\V$-graded category $\C$ determines a right $\V$-graded category $\C^\circ$ called its \textit{formal opposite}, where a graded morphism $f:A \gc X \to B$ in $\C^\circ$ is a graded morphism $f:X \gc B \to A$ in $\C$; moreover, by \ref{para:lr_vcats} there is an isomorphism of 2-categories $(-)^\circ:\LGCAT{\V}^\co \to \RGCAT{\V}$.

If $\V$ is symmetric monoidal, then $\hat{\V}$ is symmetric monoidal (\ref{sec:day_conv}), so by \ref{para:lr_vcats} we obtain an isomorphism $\LGCAT{\V} \cong \RGCAT{\V}$ under which left and right $\V$-graded categories may be identified. If $\V$ is braided monoidal, then in general multiple such isomorphisms can be constructed, in view of \ref{para:lr_vcats}; we defer detailed consideration of this case to \ref{exa:duoidal}.
\end{para}

\begin{para}\label{rem:rmod_in_vgr_cat}
Let $R = (R,m,e)$ be a monoid in $\V$. Given a left $\V$-graded category $\C$, the following notion of left $R$-module in $\C$ subsumes the similarly named notions in left $\V$-actegories and left $\V$-categories (\ref{para:modules_vact_vcat}): A \textbf{left $R$-module} (or \textbf{left $R$-act}) in $\C$ is a left $\V$-graded functor from $R$ (regarded as a one-object left $\V$-category) to $\C$, equivalently, an object $A$ of $\C$ equipped with a graded morphism $a:R \gc A \to A$ in $\C$ that is \textit{associative} and \textit{unital} in the sense that $m^*(a) = a \circ a:R \otimes R \gc A \to A$ and $e^*(a) = \gid_A:I \gc A \to A$. In \ref{para:cdiag_vgr_cats} we discuss an equivalent diagrammatic formulation. A \textbf{right $R$-module} in a right $\V$-graded category $\C$ is a right $\V$-graded functor from the right $\V$-category $R^\circ$ (\ref{para:lr_vcats}) to $\C$, equivalently, an object $A$ of $\C$ with a graded morphism $a:A\gc R \to A$ that is associative and unital.
\end{para}

\begin{exa}\label{exa:gr_str_hatv}
By \ref{para:actegories}, the huge biclosed monoidal category $\hat{\V}$ may be regarded as a (huge) left $\hat{\V}$-category (resp. a right $\hat{\V}$-category) with hom-objects $\tensor[^P]{Q}{}$ (resp. $Q^P$), where $P,Q \in \ob\hat{\V}$. In other words, $\hat{\V}$ underlies both a (huge) left $\V$-graded category and also a right $\V$-graded category. Explicitly, a graded morphism $\phi:X \gc P \to Q$ in the left $\V$-graded category $\hat{\V}$ is an element $\phi \in \tensor[^P]{Q}{}(X)$ and so by  \ref{sec:day_conv} is precisely a natural transformation $\phi:P \Rightarrow Q^X = Q(X\otimes-) \cong Q^{\mathsf{Y}(X)}$, or equivalently a natural transformation $\mathsf{Y}(X) \otimes P \Rightarrow Q$. Similarly, a graded morphism $\psi:P \gc X \to Q$ in the right $\V$-graded category $\hat{\V}$ is a natural transformation $\phi:P \Rightarrow \tensor[^X]{Q}{} = Q(-\otimes X)$ and is equivalently given by a natural transformation $P \otimes \mathsf{Y}(X) \Rightarrow Q$.
\end{exa}

\section{Contravariant change of base for graded categories}\label{sec:ch_base_grcats}
In this section, we show that there is a 2-functor $\LGCAT{(-)}:\MCAT_\oplax^\coop \to \TWOCAT$ given on objects by $\V \mapsto \LGCAT{\V}$. Here we use the notation of \ref{para:set_small} in writing $\MCAT_\oplax$ for the 2-category of ($\SET$-small) monoidal categories with opmonoidal functors and $\TWOCAT$ for the 2-category of 2-categories. We prove this by way of the following, in which we write $\MCAT'$ for the 2-category of huge monoidal categories (\ref{para:set_small}).

\begin{lem}\label{thm:contrav_func_hat}
There is a 2-functor $\widehat{(-)}:\MCAT_\oplax^\coop \to \MCAT'$ that sends each $\SET$-small monoidal category $\V$ to the huge monoidal category $\hat{\V}$.
\end{lem}
\begin{proof}
Let $\times:\MCAT' \times \MCAT' \to \MCAT'$ be the 2-functor given by taking the (conical) product in the 2-category $\MCAT'$.  Given any $\SET$-small monoidal category $\V$, we obtain isomorphisms of categories
\begin{equation}\label{eq:univ_pr_vhat}\MCAT'(\X,\hat{\V}) \cong \MCAT'(\V^\op \times \X,\SET)\end{equation}
2-natural in $\X \in \MCAT'$, by \cite[p.~269]{Kou:AugVDblCats}, under which monoidal functors $\X \to \hat{\V}$ correspond by transposition to monoidal functors $\V^\op \times \X \to \SET$. But the right-hand side of \eqref{eq:univ_pr_vhat} is 2-functorial not only in $\X \in \MCAT'$ but also in $\V \in \MCAT_\oplax^\coop$, because there is a 2-functor $(-)^\op:\MCAT_\oplax^\co \to \MCAT$ given by $\V \mapsto \V^\op$, so that by composition we obtain a 2-functor $\MCAT'((-)^\op \times (?),\SET):\MCAT_\oplax^{\coop} \times (\MCAT')^\op \to \CAT'$. Consequently, by \cite[\S 1.10]{Ke:Ba} there is a unique 2-functor $\widehat{(-)}:\MCAT_\oplax^\coop \to \MCAT'$ that is given on objects by $\V \mapsto \hat{\V}$ and makes the isomorphisms \eqref{eq:univ_pr_vhat} 2-natural in $\V \in \MCAT_\oplax^\coop$.
\end{proof}

\begin{thm}\label{thm:backward_ch_base_2functor}
There is a 2-functor $\LGCAT{(-)}:\MCAT_\oplax^{\co\op} \to \TWOCAT$ that sends each \textnormal{(\textit{$\SET$-small})} monoidal category $\V$ to the 2-category $\LGCAT{\V}$ of left $\V$-graded categories.  Similarly, there is a 2-functor $\RGCAT{(-)}:\MCAT_{\oplax}^\coop \to \TWOCAT$.
\end{thm}
\begin{proof}
The well known processes of change of base for enriched categories \cite[6.3]{EiKe} provide a 2-functor $\LCAT{(-)}:\MCAT' \to \TWOCAT$, and by composing with the 2-functor in Lemma \ref{thm:contrav_func_hat} we obtain the needed 2-functor $\LGCAT{(-)} := \LCAT{\widehat{(-)}}$. The remaining claim follows, since there is an isomorphism of 2-categories $(-)^\rev:\MCAT_{\oplax} \xrightarrow{\sim} \MCAT_{\oplax}$.
\end{proof}

\begin{para}\label{para:cov_ch_base_along_f}
The 2-functor $\LGCAT{(-)}$ in Theorem \ref{thm:backward_ch_base_2functor} sends each opmonoidal functor $F:\V \to \W$ to a 2-functor that we denote by $F^*\;:\;\LGCAT{\W} \to \LGCAT{\V}$. Each $\W$-graded category $\C$ is sent by $F^*$ to a $\V$-graded category $F^*\C$ with the same objects, in which a graded morphism $f:X \gc A \to B$ (with grade $X \in \ob\V$) is by definition a graded morphism $f:FX \gc A \to B$ in $\C$; moreover, $(F^*\C)(- \gc A;B) = \C({F-} \gc A;B):\V^\op \to \SET$. The composition and identities in $F^*\C$ are given by employing those in $F^*\C$ and reindexing along the structural morphisms $\delta_{YX}:F(Y \otimes X) \to FY \otimes FX$ and $\varepsilon:FI_\V \to I_\W$ associated to the opmonoidal functor $F$.
\end{para}

\begin{para}\label{exa:comon_ch_base}
A comonoid $R = (R,d:R \to R \otimes R,c:R \to I)$ in a monoidal category $\V$ is equivalently an opmonoidal functor $R:1 \to \V$ and so induces a 2-functor $R^*:\LGCAT{\V} \to \LGCAT{1} = \CAT$. In particular, the unit object $I_\V$ of $\V$ underlies a comonoid in $\V$, and $I^*_\V = (-)_0:\LGCAT{\V} \to \CAT$ is given by taking the underlying ordinary category. If $F:\V \to \W$ is an opmonoidal functor then we may regard its unit constraint $\varepsilon:FI_\V \to I_\W$ as an opmonoidal transformation $\varepsilon:FI_\V \Rightarrow I_\W$, so if $F$ is \textit{normal} in the sense that $\varepsilon$ is invertible then (by \ref{thm:backward_ch_base_2functor}) the 2-functor $F^*:\LGCAT{\W} \to \LGCAT{\V}$ commutes with $(-)_0:\LGCAT{\W} \to \CAT$ and $(-)_0:\LGCAT{\V} \to \CAT$, up to a 2-natural isomorphism consisting of identity-on-objects functors; this isomorphism is an identity if $F$ is \textit{strictly normal} in the sense that $\varepsilon$ is an identity. Every strong monoidal functor may be regarded as a normal opmonoidal functor, and we employ this tacitly in the sequel.
\end{para}

\section{Bigraded categories}\label{sec:bigraded}

\begin{para}\label{para:cats_gr_prod}
Conical finite products in the 2-categories $\MCAT$, $\MCAT_\oplax$, $\MCAT_\strong$ of \ref{para:set_small} are formed by taking the product in $\CAT$ and equipping it with the pointwise monoidal structure. Given monoidal categories $\V$ and $\W$, both of whose monoidal structures we write as $\otimes,I,a,\ell,r$, we can thus form the monoidal category $\V \times \W$. The unit objects determine strong (op)monoidal functors $I:1 \to \V$ and $I:1 \to \W$, so we obtain strong (op)monoidal functors $U_\ell = (-,I):\V \to \V \times \W$ and $U_r = (I,-):\W \to \V \times \W$, which are strictly normal (\ref{exa:comon_ch_base}). By contravariant change of base (\ref{thm:backward_ch_base_2functor}), these determine 2-functors
$$U_\ell^*:\LGCAT{\V \times \W} \to \LGCAT{\V},\;\;\;\;U_r^*:\LGCAT{\V \times \W} \to \LGCAT{\W}$$
that strictly commute with the functors $(-)_0$ valued in $\CAT$, by \ref{exa:comon_ch_base}. Thus every left $(\V \times \W)$-graded category $\C$ has an underlying left $\V$-graded category $U^*_\ell\C$ and also an underlying left $\W$-graded category $U^*_r\C$, and these have the same underlying ordinary category as $\C$. Explicitly, $\ob U_\ell^*\C = \ob U_r^*\C = \ob\C$, while graded morphisms $f:X \gc A \to B$ in $U_\ell^*\C$ and $g:X' \gc A \to B$ in $U_r^*\C$ are by definition graded morphisms of the form $f:(X,I) \gc A \to B$ and $g:(I,X') \gc A \to B$ in $\C$, respectively. The symmetry $\smsub{\mathrm{C}}{\W\V}:\W \times \V \to \V \times \W$ underlies a strict monoidal isomorphism and by \ref{thm:backward_ch_base_2functor} induces an isomorphism $\mathrm{C}^*:\LGCAT{\V \times \W} \xrightarrow{\sim} \LGCAT{\W \times \V}$ with $U_\ell^*\mathrm{C}^* = U_r^*$ and $U_r^*\mathrm{C}^* = U_\ell^*$.
\end{para}

\begin{defn}\label{defn:bigraded}
Given monoidal categories $\V$ and $\W$, a \textbf{$\V$-$\W$-bigraded category} is a (left) $(\V \times \W^\mathsf{rev})$-graded category.  Given objects $A$ and $B$ of a $\V$-$\W$-bigraded category $\C$, we write
$$\C(X \gc A \gc X';B) := \C((X,X') \gc A;B)\;\;\;\;\;\;(X \in \V, X' \in \W),$$
and we write
$$f\;:\;X \gc A \gc X' \longrightarrow B$$
to mean that $f:(X,X') \gc A \to B$.  With this convention, if $f:X \gc A \gc X' \to B$ and $g:Y \gc B \gc Y' \to C$ then $g \circ f:Y \otimes X \gc A \gc X' \otimes Y' \to C$.

The \textbf{2-category of $\V$-$\W$-bigraded categories} is defined as
$$\GCAT{\V}{\W} = \LGCAT{\V \times \W^\rev}\;.$$
We refer to the 1-cells and 2-cells in $\GCAT{\V}{\W}$ as \textbf{$\V$-$\W$-bigraded functors} and \textbf{\mbox{$\V$-$\W$}-bigraded-natural transformations}, respectively. We write $\tensor[_\V]{{\textnormal{GCAT}'}}{_\W} = \LGCAT{\V \times \W^\rev}'$ for the 2-category of huge $\V$-$\W$-bigraded categories. By \ref{para:cats_gr_prod}, we have an isomorphism $\mathrm{C}^*:\GCAT{\V}{\W} \xrightarrow{\sim} \GCAT{\W^\rev}{\V^\rev}$

By \ref{para:cats_gr_prod}, every $\V$-$\W$-bigraded category $\C$ has an underlying left $\V$-graded category $\CleftV := U^*_\ell\C$ and an underlying right $\W$-graded category $\C_\W := U^*_r\C$. A graded morphism $f:X \gc A \to B$ in $\CleftV$ is a graded morphism $f:X \gc A \gc I \to B$ in $\C$, while a graded morphism $g:A \gc X' \to B$ in $\C_\W$ is a graded morphism $g:I \gc A \gc X' \to B$ in $\C$. We sometimes write $\CleftV$ and $\C_\W$ simply as $\C$, and thus we may speak of graded morphisms $f:X \gc A \to B$ and $g:A \gc X' \to B$ in $\C$.
\end{defn}

\begin{exa}\label{exa:biact}
A \textit{$\V$-$\W$-biactegory} \cite{Skoda:Biactegories,Skoda:SomeEquivariant,EGNO:TensorCategories,CapGav:Actegories} is a category $\C$ equipped with the structure of both a left $\V$-actegory and right $\W$-actegory, together with isomorphisms $X.(A.X') \cong (X.A).X'$ $(X \in \V, A \in \C, X' \in \W)$ satisfying certain axioms \cite[4.3.1]{CapGav:Actegories}. For our purposes it will be more convenient to employ the following equivalent definition \cite[4.3.6]{CapGav:Actegories}: By definition, a \textbf{$\V$-$\W$-biactegory} $\C$ is a (left) $(\V \times \W^\mathsf{rev})$-actegory, whose action we write as $X.A.X' = (X,X').A$ $(X \in \V, A \in \C, X' \in \W)$, so that this action determines a functor of the form $\V \times \C \times \W \to \C$.  By \ref{para:acts_as_vgr_cats}, every $\V$-$\W$-biactegory $\C$ may be regarded as a $\V$-$\W$-bigraded category in which a graded morphism $f:X \gc A \gc X' \to B$ is a morphism $f:X.A.X' \to B$ in $\C$. Note that every $\V$-$\W$-biactegory $\C$ has an underlying left $\V$-actegory and an underlying right $\W$-actegory, with actions given by $X.A = X.A.I$ and $A.X' = I.A.X'$ $(X \in \V, A \in \C, X' \in \W)$, respectively. With this notation, we have isomorphisms $(X.A).X' \cong X.A.X' \cong X.(A.X')$, which we sometimes omit with the understanding that the reader can readily supply these when needed.
\end{exa}

\begin{exa}\label{exa:bigr_str_hatv}
Every monoidal category $\V$ itself is a $\V$-$\V$-biactegory via $X.A.X' = X \otimes A \otimes X'$ (actually $(X \otimes A) \otimes X'$, say), so that $\V$ is a $\V$-$\V$-bigraded category.

Similarly, the huge monoidal category $\hat{\V}$ is a $\hat{\V}$-$\hat{\V}$-biactegory, so by restriction along the strong monoidal functor $\mathsf{Y} \times \mathsf{Y}^\rev:\V \times \V^\rev \to \hat{\V} \times \hat{\V}^\rev$ we find that $\hat{\V}$ is a $(\V \times \V^\rev)$-actegory, i.e. a $\V$-$\V$-biactegory, with $X.P.Y = \mathsf{Y}(X) \otimes P \otimes \mathsf{Y}(Y)$ naturally in $X,Y \in \V$, $P \in \hat{\V}$. Each functor $X.(-).Y:\hat{\V} \to \hat{\V}$ $(X,Y \in \ob\V)$ has a right adjoint $\tensor[^Y]{(-)}{^X}$ defined by $\tensor[^Y]{P}{^X} = P(X \otimes - \otimes Y) \cong (\tensor[^Y]{P}{})^X \cong \tensor[^Y]{(P^X)}{}$ $(P \in \V)$ with the notation of \ref{sec:day_conv}, because $\hat{\V}(Q,(\tensor[^Y]{P}{})^X) \cong \hat{\V}(X.Q,\tensor[^Y]{P}{}) \cong \hat{\V}(X.Q.Y,P)$ naturally in $Q \in \hat{\V}$. Thus we  may regard $\hat{\V}$ as a huge $\V$-$\V$-bigraded category, in which a graded morphism $\phi:X \gc P \gc Y \to Q$ in $\hat{\V}$ is a natural transformation $\phi: P \Rightarrow \tensor[^Y]{Q}{^X} = Q(X \otimes - \otimes Y)$.
\end{exa}

\begin{exa}\label{exa:top-box-box-biact}
By \cite[10.4]{GrMau:Cubical}, there is a strong monoidal functor $F:\Box \to \Top$ (on the restricted cubical site, \ref{exa:actegories}, \ref{exa:gr_cats}) that sends the object $2$ of $\Box$ to the unit interval $[0,1]$ and sends the generating morphisms $\varepsilon:2 \to 1 = 2^0$ and $0,1:1 \rightrightarrows 2$ to the unique map $[0,1] \to 1$ and the maps $1 \rightrightarrows [0,1]$ that pick out the endpoints $0,1 \in [0,1]$, respectively. By \ref{exa:bigr_str_hatv}, $\Top$ is a $\Top$-$\Top$-biactegory and so, by restriction along the strong monoidal functor $F \times F^\rev:\Box \times \Box^\rev \to \Top \times \Top^\rev$, $\Top$ is a $\Box$-$\Box$-biactegory and so may be regarded as a $\Box$-$\Box$-bigraded category for the non-symmetric monoidal category $\Box$.
\end{exa}

\begin{exa}\label{exa:v1_1v_bigraded}
$\V$-$1$-bigraded categories (resp.~$1$-$\V$-bigraded categories) are precisely left $\V$-graded categories (resp.~right $\V$-graded categories).
\end{exa}

We develop further general classes of examples of bigraded categories in \ref{exa:vgr_func}(5).

\begin{para}\label{para:ch_base_bigr}
There is a 2-functor $(-) \times (?)^\rev:\MCAT_\oplax \times \MCAT_\oplax \to \MCAT_\oplax$ given on objects by $(\V,\W) \mapsto \V \times \W^\rev$, so by Theorem \ref{thm:backward_ch_base_2functor} we obtain a 2-functor
$$\GCAT{(-)}{(?)} = \LGCAT{(-) \times (?)^\rev}\;:\;\MCAT_\oplax^{\coop}  \times \MCAT_\oplax^{\coop} \longrightarrow \TWOCAT.$$
\end{para}

\begin{para}\label{para:cats_gr_pr_vs_vw_bigr}
Categories graded by a product $\V \times \W$ are equally $\V$-$\W^\rev$-bigraded categories. Our motivations for focusing on $\V$-$\W$-bigraded categories rather than $\V \times \W$-graded categories are threefold. Firstly, $\V$-$\W$-bigraded categories support a convenient diagrammatic formalism (\ref{para:cdiag_vgr_cats}). Secondly, $\V$-$\W$-bigraded categories ease the expression of the notion of graded bifunctor, as we discuss in \ref{rem:bifun_vtimesw_gr}. Thirdly, for an arbitrary monoidal category $\V$, we may regard $\V$ and $\hat{\V}$ as $\V$-$\V$-bigraded categories (\ref{exa:bigr_str_hatv}) but \textit{not} in general as $(\V \times \V)$-graded categories. There is, however, an important special case in which categories graded by $\V \times \V$ abound:
\end{para}

\begin{exa}\label{exa:duoidal}
A \textit{duoidal category} $\dV = (\V,\dpr,J,\alpha,\lambda,\rho)$ consists of a monoidal category $\V$ (for which we use our usual notations) that is equipped with opmonoidal functors $\dpr:\V \times \V \to \V$, $J:1 \to \V$, and opmonoidal transformations $\alpha$, $\lambda$, $\rho$ that make $\dV = (\V,\dpr,J,\alpha,\lambda,\rho)$ a pseudomonoid in $\MCAT_\oplax$, where we regard $\MCAT_\oplax$ as a 2-category with (conical) finite products (\ref{para:cats_gr_prod}) and so as a monoidal 2-category. One says that $\dV$ is \textit{normal} if the opmonoidal functors $\dpr:\V \times \V \to \V$ and $J:1 \to \V$ are both normal (\ref{exa:comon_ch_base}). We discuss duoidal categories in more detail in \S \ref{sec:duoidal}, referring the reader to \cite[6.1.1]{AgMa} for an explicit account of the axioms and to \cite{GaLf} for further discussion of duoidal categories and their history. Given any duoidal category $\dV = (\V,\dpr,J,\alpha,\lambda,\rho)$, we can apply contravariant change of base (\ref{thm:backward_ch_base_2functor}) to obtain a 2-functor $\dpr^*:\LGCAT{\V} \to  \LGCAT{\V \times \V}$ that sends each (left) $\V$-graded category to a $(\V \times \V)$-graded category $\C_\dpr := \dpr^*\C$ that we discuss in \S \ref{sec:duoidal}.

For example, as discussed in \cite{GaLf}, every braided monoidal category $\V$ carries the structure of a normal duoidal category in which $\dpr = \otimes$ and $J = I$. But, by \cite[Example 2.5]{JoyStr}, the braiding $c = c_{XY}:X \otimes Y \to Y \otimes X$ $(X,Y \in \V)$ equips the identity functor on $\V$ with the structure of a strong opmonoidal functor $\bar{c}:\V^\rev \xrightarrow{\sim} \V$ whose unit constraint is an identity. Hence, we have an isomorphism $\bar{c}^*:\LGCAT{\V} \xrightarrow{\sim} \LGCAT{\V^\rev} = \RGCAT{\V}$ under which each left $\V$-graded category $\C$ corresponds to a right $\V$-graded category with the same objects and graded morphisms, but with composition given by first composing in $\C$ and then reindexing along $c$. By \ref{para:ch_base_bigr} we also obtain an isomorphism $\GCAT{\V}{\V} \cong \GCAT{\V}{\V^\rev} = \LGCAT{\V \times \V}$. Thus, by the above, every left $\V$-graded category $\C$ determines a $\V$-$\V$-bigraded category in which a graded morphism $f:X \gc A \gc X' \to B$ is a graded morphism $f:X \otimes X' \gc A \to B$ in $\C$.
\end{exa}

\section{Envelope diagrams for graded categories}\label{sec:emb}

Given a monoidal category $\V$, we now embed each (left) $\V$-graded category $\C$ into a (left) $\V$-actegory and then use this to define a formalism of commutative diagrams for $\V$-graded categories.

\begin{prop}\label{thm:copower_cocompl_vgr}
The forgetful 2-functor $\LGCAT{\V}^{\bigdot\textnormal{pres}} \to \LGCAT{\V}$ has a left biadjoint whose unit consists of fully faithful $\V$-graded functors that are injective on objects. Moreover, given any $\V$-graded category $\C$, there is a $\V$-graded category $\C_\bigdot$ with $\V$-copowers and a $\V$-graded functor $E:\C \to \C_\bigdot$ with the following properties: (1) for every $\V$-graded category $\D$ with $\V$-copowers, the functor $(-) \circ E:\LGCAT{\V}^{\bigdot\textnormal{pres}}(\C_\bigdot,\D) \to \LGCAT{\V}(\C,\D)$ is an equivalence, (2) $E$ is fully faithful and injective on objects, and (3) every object of $\C_\bigdot$ is a copower of some object of the form $EA$ $(A \in \ob\C)$ by some $X \in \ob\V$. 
\end{prop}
\begin{proof}
Take $\sfV = \hat{\V}$ and let $\R \hookrightarrow \sfV$ be the full subcategory spanned by the representables $\mathsf{Y} X = \V(-,X)$ $(X \in \ob\V)$, noting that $\R$ is $\SET$-small. By \ref{para:copower_cocompl}, we can form the free cocompletion $\R(\C)$ under $\R$-copowers, which is $\SET$-small, and the associated $\sfV$-functor $K:\C \to \R(\C)$ is fully faithful and has the universal property in (1). By \ref{para:copower_cocompl}, the closure $\bar{\K}$ of $\K = \{KA \mid A \in \ob\C\} \hookrightarrow \R(\C)$ under $\R$-copowers is $\R(\C)$ itself, but $\bar{\K}$ is simply the repletion (i.e.~isomorphism closure) of $\{\mathsf{Y}X \cdot KA \mid X \in \ob\V, A \in \ob\C\} \hookrightarrow \R(\C)$, because $\mathsf{Y}Y \cdot (\mathsf{Y}X \cdot KA) \cong (\mathsf{Y}Y \otimes \mathsf{Y}X) \cdot KA \cong \mathsf{Y}(Y \otimes X) \cdot KA$ for all $Y,X \in \ob\V$, $A \in \ob\C$. Therefore each object of $\R(\C)$ is a copower of some $KA$ $(A \in \ob\C)$ by some $X \in \ob\V$, so while $K$ need not be injective on objects, it follows that there is a $\V$-graded category $\C_{\bigdot}$ with object set $\ob\V \times \ob\C$ and an equivalence $\Gamma:\C_\bigdot \to \R(\C)$ given on objects by $(X,A) \mapsto X \cdot KA$. It follows that there is a $\V$-graded functor $E:\C \to \C_\bigdot$ given on objects by $A \mapsto (I,A)$ such that $\Gamma E \cong K$, and the result follows.
\end{proof}

\begin{cor}\label{thm:cor_free_lact}
The canonical 2-functor $U:\LACT{\V}^\strong \to \LGCAT{\V}$ has a left biadjoint whose unit consists of fully faithful $\V$-graded functors that are injective on objects.
\end{cor}
\begin{proof}
By \ref{para:acts_as_vgr_cats} we have a (strict) 2-equivalence $\LACT{\V}^\strong \simeq \LGCAT{\V}^{\bigdot\textnormal{pres}}$, and $U$ is simply the composite of this 2-equivalence and the forgetful 2-functor $\LGCAT{\V}^{\bigdot\textnormal{pres}} \to \LGCAT{\V}$, so this follows from \ref{thm:copower_cocompl_vgr}.
\end{proof}

\begin{defn}\label{defn:env_act}
Given a left $\V$-graded category $\C$, we write $\V \gct \C$ to denote the left $\V$-actegory to which $\C$ is sent by the left biadjoint $\LGCAT{\V} \to \LACT{\V}^\strong$ in \ref{thm:cor_free_lact}.  As usual, we also denote the $\V$-graded category $U(\V \gct \C)$ by $\V \gct \C$.  Hence the $\V$-actegory $\V \gct \C$ is equipped with a $\V$-graded functor
$$E:\C \hookrightarrow \V \gct \C$$
that is fully faithful and injective on objects. We call $\V \gct \C$ the \textbf{enveloping actegory} or \textbf{actegorical envelope} of $\C$. We write the $\V$-action carried by $\V \gct \C$ as ``$\gc$'', and via the embedding $E$ we identify $\C$ with the full $\V$-graded subcategory of $\V \gct \C$ spanned by the objects $EA$ with $A \in \ob\C$, denoting these objects simply by $A$. Thus, for each graded morphism $f:X \gc A \to B$ in $\C$ we write simply $f:X \gc A \to B$ to denote the associated morphism $Ef:X \gc EA \to EB$ in $\V \gct \C$. In particular, $\gid_A:I \gc A \to A$ in $\C$ is identified with the unit constraint $\gid_A:I \gc A \xrightarrow{\sim} A$ in $\V \gct \C$ (\ref{para:acts_as_vgr_cats}). By \ref{thm:copower_cocompl_vgr} and the proof of \ref{thm:cor_free_lact}, every object of $\V \gct \C$ is isomorphic to the object $X \gc A$ obtained by applying the $\V$-action to some $X \in \ob\V$ and $A \in \ob\C$. 
\end{defn}

\begin{para}[\textbf{Envelope diagrams}]\label{para:cdiag_vgr_cats}
The embedding of a left $\V$-graded category $\C$ into its enveloping actegory $\V \gct \C$ enables the use of commutative diagrams that denote graded morphisms in $\C$. Given graded morphisms $f:X \gc A \to B$, $g:Y \gc B \to C$ in $\C$, and a morphism $\alpha:X' \to X$ in $\V$, the composite $g \circ f:Y\otimes X \gc A \to C$ and the reindexing $\alpha^*(f):X' \gc A \to B$ are the following composites in $\V \gct \C$, respectively:
$$Y \otimes X \gc A \xrightarrow{\sim} Y\gc X\gc A \xrightarrow{Y\gc f} Y\gc B \xrightarrow{g} C\;\;\;\;\;\;\;\;\;\;X' \gc A \xrightarrow{\alpha \gc A} X \gc A \xrightarrow{f} B\;.$$
We sometimes omit the isomorphisms $Y \gc X \gc A \cong Y \otimes X \gc A$ and $I \gc A \cong A$ associated to the $\V$-action carried by $\V \gct \C$, with the understanding that the reader can readily supply these when needed.

We can thus use certain commutative diagrams in $\V \gct \C$ to express equations in a left $\V$-graded category $\C$. An \textbf{envelope diagram} in a left $\V$-graded category $\C$ is, by definition, a diagram in $\V \gct \C$ that is expressed entirely in terms of graded morphisms in $\C$ and morphisms in $\V$, using the action carried by $\V \gct \C$. For example, given a monoid $R = (R,m,e)$ in $\V$ and a left $\V$-graded category $\C$, a left $R$-module in $\C$ in the sense of \ref{rem:rmod_in_vgr_cat} is an object $A$ of $\C$ equipped with a graded morphism $a:R \gc A \to A$ in $\C$ making the following envelope diagrams commute:
$$
\xymatrix{
R \otimes R \gc A \ar[r]^\sim \ar[d]_{m \gc A} & R\gc R\gc A \ar[rr]^{R\gc a} & & R\gc A \ar[d]^{a} & I \gc A \ar[rr]^(.5){e\gc A} \ar[drr]_{\gid_A} & & R\gc A \ar[d]^{a}\\
R\gc A \ar[rrr]^{a} & & & A & & & A
}
$$

Given a right $\V$-graded category $\C$, we write $\C \gcl \V$ to denote the enveloping actegory $\V^\rev\gct \C$ of the left $\V^\rev$-graded category $\C$.  Since $\V^\rev\gct \C$ is a left $\V^\rev$-actegory, this means that $\C\gcl \V$ is a right $\V$-actegory, which we call the enveloping actegory of $\C$. We have an embedding $E:\C \hookrightarrow \C \gcl \V$ under which we identify $\C$ with a full $\V$-graded subcategory of $\C\gcl \V$, and we write ``$\gc$'' to denote the $\V$-action on $\C \gcl \V$. Thus we may regard each graded morphism $f:A \gc X \to B$ in $\C$ as a morphism in the $\V$-actegory $\C\gcl \V$. Given graded morphisms $f:A \gc X \to B$, $g:B \gc Y \to C$ in $\C$, the composite $g \circ f:A \gc X \otimes Y \to C$ in $\C$ is the composite $A \gc X \otimes Y \xrightarrow{\sim} A \gc X \gc Y \xrightarrow{f \gc Y} B \gc Y \xrightarrow{g} C$ in the right $\V$-actegory $\C\gcl \V$.

Let $\C$ be a $\V$-$\W$-bigraded category for monoidal categories $\V$ and $\W$. Then $\C$ is a left $(\V \times \W^\rev)$-graded category, so its enveloping actegory $(\V \times \W^\rev)\gcl \C$ is a left $(\V \times \W^\rev)$-actegory, i.e. a $\V$-$\W$-biactegory, which we denote by $\V \gct \C\gct \W$ and call the \textbf{enveloping biactegory} of $\C$. We write the action carried by $\V \gct \C \gct \W$ as $X \gc D \gc X'$ $(X \in \V, D \in \V \gct \C \gct \W, X' \in \W)$, and there is a fully faithful $\V$-$\W$-bigraded functor $E:\C \hookrightarrow \V \gct \C\gct \W$ that allows us to regard each graded morphism $f:X\gc A\gc X' \to B$ in $\C$ as a morphism in the $\V$-$\W$-biactegory $\V \gct \C\gct \W$. Given graded morphisms $f:X\gc A\gc X' \to B$ and $g:Y \gc B \gc Y' \to C$ in $\C$ and morphisms $\alpha:Z \to X$ in $\V$ and $\alpha':Z' \to X'$ in $\W$, the composite $g \circ f$ and reindexing $(\alpha,\alpha')^*(f)$ in $\C$ can be written as the following composites in $\V \gct \C \gct \W$, respectively:
$$Y \otimes X \gc A \gc X' \otimes Y' \xrightarrow{\sim} Y\gc X\gc A\gc X'\gc Y' \xrightarrow{Y \gc f\gc Y'} Y\gc B\gc Y' \xrightarrow{g} C$$
$$Z \gc A\gc Z' \xrightarrow{\alpha\gc A\gc \alpha'} X\gc A\gc X' \xrightarrow{f} B\;.$$
In view of \ref{defn:env_act}, every object of $\V \gct \C \gct \W$ is isomorphic to one of the form $X \gc A \gc X'$ with $X \in \ob\V$, $A \in \ob\C$, $X' \in \ob\W$. In view of \ref{exa:biact}, $\V \gct \C \gct \W$ also has an underlying left $\V$-actegory  whose left $\V$-action is given by $X \gc D = X \gc D \gc I$ $(X \in \V, D \in \V \gct \C \gct \W)$, and similarly $\V \gct \C \gct \W$ has an underlying right $\W$-actegory. In particular, given $X \in \ob\V$ and $A,B \in \ob\C$, a morphism $f$ from $X \gc A = X \gc A \gc I$ to $B$ in $\V \gct \C \gct \W$ is equivalently a graded morphism $f:X \gc A \gc I \to B$ in $\C$, i.e.~a graded morphism $f:X \gc A \to B$ in (the left $\V$-graded category underlying) $\C$ (\ref{defn:bigraded}). Similar remarks apply also to graded morphisms of the form $f:A \gc X' \to B$ in (the right $\W$-graded category underlying) $\C$. As discussed in \ref{exa:biact}, we have isomorphisms $(X \gc A) \gc X' \cong X \gc A \gc X' \cong X \gc (A \gc X')$ in $\V \gct \C \gct \W$, which we often omit.
\end{para}

In the remainder of the paper, we freely use the notations in \ref{para:cdiag_vgr_cats}.

\section{Bigraded squares in bigraded categories}\label{sec:bigr_sq}

Let $\V$ and $\W$ be monoidal categories.  The following concept will be useful in our studies of graded functor categories and bifunctors:

\begin{defn}\label{defn:bigr_sq}
Given a $\V$-$\W$-bigraded category $\C$ and objects $X \in \ob\V$ and $X' \in \ob\W$, an \textbf{$(X,X')$-bigraded square} $s = (f,g,\phi,\phi')$ in $\C$ consists of graded morphisms $f:X \gc A \to A'$, $g:X \gc B \to B'$, $\phi:A \gc X' \to B$, $\phi':A' \gc X' \to B'$ in $\C$ such that the following envelope diagram commutes:
\begin{equation}\label{eq:bigr_square}
\xymatrix{
X \gc A \gc X' \ar[d]_{f \gc X'} \ar[r]^(.55){X \gc \phi} & X \gc B \ar[d]^g\\
A' \gc X' \ar[r]^{\phi'} & B'
}
\end{equation}
The \textbf{diagonal} $\Delta s:X \gc A \gc X' \to B'$ of $s$ is the common composite in \eqref{eq:bigr_square}. We use the following schematic notation to denote the $(X,X')$-bigraded square $s = (f,g,\phi,\phi')$:
\begin{equation}\label{eq:bigr_sq_notn}
\xymatrix{
A \ar@{-}[d]_f \ar@{-}[r]^\phi="s1" & B \ar@{-}[d]^g\\
A' \ar@{-}[r]^{\phi'}="t1" & B'
\ar@{}"s1";"t1"|(0.6){X,X'}
}
\end{equation}
Note that, with the conventions of \ref{defn:bigraded} and \ref{para:cdiag_vgr_cats}, $f,g$ are graded morphisms in $\CleftV$ and $\phi,\phi'$ are graded morphisms in $\C_\W$, while \eqref{eq:bigr_square} is a diagram in $\V \gct \C \gct \W$.
\end{defn}

While we usually use envelope diagrams in working with bigraded squares, we first pause to develop a direct formulation of bigraded squares; we later use this in \S \ref{sec:duoidal} to examine what bigraded squares amount to in the duoidal case.

\begin{rem}\label{rem:bigr_square}
In the diagram \eqref{eq:bigr_square}, we omit the isomorphisms $(X \gc A) \gc X' \cong X \gc A \gc X'$ $\cong X \gc (A \gc X')$ in keeping with our convention in \ref{para:cdiag_vgr_cats}. Also, by \ref{para:cdiag_vgr_cats}, $f$ and $g$ can be written as graded morphisms of the form $f:X \gc A \gc I \to A'$ and $g:X \gc B \gc I \to B'$ in $\C$, while $\phi$ and $\phi'$ can be written as $\phi:I \gc A \gc X' \to B$ and $\phi':I \gc A' \gc X' \to B'$. Furthermore, in view of \ref{para:cdiag_vgr_cats}, the diagram \eqref{eq:bigr_square} can be rewritten in terms of the two-sided action carried by $\V \gct \C \gct \W$ as follows:
$$
\xymatrix{
X\gc A\gc X' \ar[d]_(.45)\wr \ar[r]^(.4)\sim & X \gc I \gc A \gc X' \gc I \ar[rr]^(.5){X \gc \phi \,\gc I} & & X \gc B\gc I \ar[d]^g\\
{I \gc X \gc A \gc I \gc X'} \ar[rr]^(.5){I\gc f \gc X'} & & I\gc A' \gc X' \ar[r]^(.6){\phi'} & B'
}
$$
Consequently, the commutativity of the square \eqref{eq:bigr_square} can be formulated without recourse to envelope diagrams, at some cost of elegance and economy, as follows:
\end{rem}

\begin{prop}\label{thm:direct_formulation_of_bigr_squares}
An $(X,X')$-bigraded square $s = (f,g,\phi,\phi')$ in a $\V$-$\W$-bigraded category $\C$ is given by graded morphisms $f:(X,I) \gc A \to A'$, $g:(X,I) \gc B \to B'$, $\phi:(I,X') \gc A \to B$ and $\phi':(I,X') \gc A' \to B'$ in the $(\V \times \W^\rev)$-graded category $\C$ such that the composites 
$$g \circ \phi:(X \otimes I,X' \otimes I) \gc A \to B'\;,\;\;\;\;\phi' \circ f:(I \otimes X,I \otimes X') \gc A \to B'$$
are related by the equation
\begin{equation}\label{eq:common_comp}(r_X^{-1},r_{X'}^{-1})^*(g \circ \phi) = (\ell_X^{-1},\ell_{X'}^{-1})^*(\phi' \circ f)\;:\;(X,X') \gc A \to B,\end{equation}
where $(r_X^{-1},r_{X'}^{-1}):(X,X') \to (X \otimes I,X' \otimes I)$ and $(\ell_X^{-1},\ell_{X'}^{-1}):(X,X') \to (I \otimes X,I \otimes X')$ are the morphisms in $\V \times \W^\rev$ determined by the left and right unitors in $\V$ and $\W$. The diagonal $\Delta s$ is the common value in \eqref{eq:common_comp}.
\end{prop}

\begin{para}\label{para:flip_bigraded_square}
$(f,g,\phi,\phi')$ is a bigraded square in a $\V$-$\W$-bigraded category $\C$ iff $(\phi,\phi',f,g)$ is a bigraded square in the $\W^\rev$-$\V^\rev$-bigraded category $\mathrm{C}^*\C$ of \ref{para:cats_gr_prod}.
\end{para}

\begin{para}\label{rem:bigr_sq_prod_gr}
A \textit{bigraded square} $(f,g,\phi,\phi')$ in a $(\V \times \W)$-graded category $\C$ is, by definition, a bigraded square in the $\V$-$\W^\rev$-bigraded category $\C$ and so, by \ref{thm:direct_formulation_of_bigr_squares}, consists of graded morphisms $f:(X,I) \gc A \to A'$, $g:(X,I) \gc B \to B'$, $\phi:(I,X') \gc A \to B$ and $\phi':(I,X') \gc A' \to B'$ in $\C$ such that the composites $g \circ \phi:(X \otimes I,I \otimes X') \gc A \to B'$ and $\phi' \circ f:(I \otimes X,X' \otimes I) \gc A \to B'$
satisfy the equation $(r_X^{-1},\ell_{X'}^{-1})^*(g \circ \phi) = (\ell_X^{-1},r_{X'}^{-1})^*(\phi' \circ f):(X,X') \gc A \to B'$. The added complexity of this formulation in comparison with \ref{defn:bigr_sq} is part of why we favour the formulation in \ref{defn:bigr_sq} via bigraded categories. Note that we can write $f,g,\phi,\phi'$ as graded morphisms $f:X \gc A \to A'$, $g:X \gc B \to B'$ in $U_\ell^*\C$ and $\phi:X' \gc A \to B$, $\phi':X' \gc A' \to B'$ in $U_r^*\C$. 
\end{para}

\begin{exa}\label{exa:bigr_sq_hatv}
We now characterize bigraded squares in the $\V$-$\V$-bigraded category $\hat{\V}$ (\ref{exa:bigr_str_hatv}). Let $f:X \gc P \to P'$, $g:X \gc Q \to Q'$, $\phi:P \gc X' \to Q$, and $\phi':P' \gc X' \to Q'$ be graded morphisms in $\hat{\V}$, i.e.~natural transformations $f:P \Rightarrow P'(X \otimes-)$, $g:Q \Rightarrow Q'(X \otimes-)$, $\phi:P \Rightarrow Q(-\otimes X')$, and $\phi':P' \Rightarrow Q'(-\otimes X')$. Then $(f,g,\phi,\phi')$ is a bigraded square in $\hat{\V}$ if and only if the following diagram commutes, where we omit the isomorphism $Q'(X \otimes (- \otimes X')) \cong Q'((X \otimes -) \otimes X')$ given by reindexing along the associator in $\V$:
$$
\xymatrix{
P \ar@{=>}[d]_f \ar@{=>}[r]^(.4){\phi} & Q(-\otimes X') \ar@{=>}[d]^{g_{-\otimes X'}}\\
P'(X \otimes -) \ar@{=>}[r]^(.43){\phi'_{X \otimes-}} & Q'(X \otimes - \otimes X')
}
$$
\end{exa}

For the remainder of this section, we fix a $\V$-$\W$-bigraded category $\C$, and we develop general results on bigraded squares that we employ later. Our first basic result is easily verified using either envelope diagrams or \ref{thm:direct_formulation_of_bigr_squares}:

\begin{prop}\label{thm:id_bigr_sq}
Let $f:X \gc A \to A'$ and $\phi:A \gc X' \to B$ in $\C$. Then $(f,f,\gid_A,\gid_{A'})$ and $(\gid_A,\gid_B,\phi,\phi)$ are bigraded squares in $\C$ with diagonals $f$ and $\phi$, respectively.
\end{prop}

\begin{prop}\label{thm:pres_bigr_sq}
Let $F:\C \to \D$ be a $\V$-$\W$-bigraded functor, and let $s = (f,g,\phi,\phi')$ be an $(X,X')$-bigraded square in $\C$. Then $Fs := (Ff,Fg,F\phi,F\psi)$ is an $(X,X')$-bigraded square in $\D$ with diagonal $\Delta(Fs) = F(\Delta s)$.
\end{prop}
\begin{proof}
Since $F$ preserves composition and reindexing, this follows from \ref{thm:direct_formulation_of_bigr_squares}.
\end{proof}

\begin{prop}\label{thm:reindex_bigr_sq}
Given an $(X,X')$-bigraded square $s = (f,g,\phi,\phi')$ in $\C$ and a morphism $(\alpha,\beta):(Y,Y') \to (X,X')$ in $\V \times \W^\rev$, we obtain a $(Y,Y')$-bigraded square $(\alpha,\beta)^*(s) := (\alpha^*(f),\alpha^*(g),\beta^*(\phi),\beta^*(\phi'))$ where $\alpha^*$ and $\beta^*$ denote reindexing in $\CleftV$ and $\C_\W$, respectively. Furthermore, $\Delta((\alpha,\beta)^*(s)) = (\alpha,\beta)^*(\Delta s)$ is the reindexing of $\Delta s$ along $(\alpha,\beta)$ in $\C$.
\end{prop}
\begin{proof}
Writing $s$ in the usual manner as in \ref{defn:bigr_sq}, we have an envelope diagram
\begin{equation}\label{eq:bigr_sq_reindex}
\xymatrix{
Y \gc A \gc Y' \ar[drr]|(.5){\alpha \gc A \gc \beta} \ar[d]_{\alpha \gc A \gc Y'} \ar[rr]^{Y \gc A \gc \beta} & & Y \gc A \gc X' \ar[d]^{\alpha \gc A \gc X'}\ar[rr]^{Y \gc \phi} & & Y \gc B \ar[d]^{\alpha \gc B} \\
X \gc A \gc Y' \ar[d]_{f \gc Y'} \ar[rr]_{X \gc A \gc \beta} & & X \gc A \gc X' \ar[drr]|{\Delta s} \ar[rr]^{X \gc \phi} \ar[d]^{f \gc X'}& & X \gc B \ar[d]^g\\
A' \gc Y' \ar[rr]^{A' \gc \beta} & & A' \gc X' \ar[rr]^{\phi'} & & B'
}
\end{equation}
whose periphery is precisely the following envelope diagram:
$$
\xymatrix{
Y \gc A \gc Y' \ar[d]_{\alpha^*(f) \gc Y'} \ar[rr]^{Y \gc \beta^*(\phi)} & & Y \gc B \ar[d]^{\alpha^*(g)}\\
A' \gc Y' \ar[rr]^{\beta^*(\phi')} & & B'
}
$$
\end{proof}

\begin{prop}\label{thm:paste_bigr_sq}
Suppose we are given bigraded squares of the following form in $\C$:
$$
\xymatrix{
A \ar@{-}[d]_f \ar@{-}[r]^\phi="s1" & B \ar@{-}[d]|g \ar@{-}[r]^\psi="s2" & C \ar@{-}[d]^{h}\\
A' \ar@{-}[d]_{f'} \ar@{-}[r]^{\phi'}="t1" & B' \ar@{-}[d]|{g'} \ar@{-}[r]^{\psi'}="t2" & C' \ar@{-}[d]^{h'} \\
A'' \ar@{-}[r]^{\phi''}="u1" & B'' \ar@{-}[r]^{\psi''}="u2" & C''
\ar@{}"s1";"t1"|(0.6){X,X'}
\ar@{}"s2";"t2"|(0.6){X,Y'}
\ar@{}"t1";"u1"|(0.6){Y,X'}
\ar@{}"t2";"u2"|(0.6){Y,Y'}
}
$$
Then we obtain bigraded squares
$$
\xymatrix{
A \ar@{-}[d]_f \ar@{-}[rr]^{\psi \circ \phi}="s2" & & C \ar@{-}[d]^h & & A \ar@{-}[d]_{f' \circ f}\ar@{-}[rr]^{\phi}="s3" & &B \ar@{-}[d]^{g' \circ g} & & A \ar@{-}[d]_{f' \circ f} \ar@{-}[rr]^{\psi \circ \phi}="s1" & & C \ar@{-}[d]^{h' \circ h}\\
A' \ar@{-}[rr]^{\psi' \circ \phi'}="t2" & & C' & & A'' \ar@{-}[rr]^{\phi''}="t3"& & B'' & & A'' \ar@{-}[rr]^{\psi'' \circ \phi''}="t1" & & C''
\ar@{}"s1";"t1"|(0.6){Y \otimes X,X' \otimes Y'}
\ar@{}"s2";"t2"|(0.6){X,X' \otimes Y'}
\ar@{}"s3";"t3"|(0.6){Y \otimes X,X'}
}
$$
in which $\circ$ denotes composition in $\CleftV$ and in $\C_\W$. The diagonal of the rightmost square is the composite
$\Delta(g',h',\psi',\psi'') \circ \Delta(f,g,\phi,\phi'):Y \otimes X \gc A \gc X' \otimes Y' \to C''$.
\end{prop}
\begin{proof}
We treat the rightmost bigraded square, and we leave the first two as exercises for which similar methods are applicable; also see Remark \ref{para:horiz_vert_comp}. We have a commutative envelope diagram as follows:
$$
\xymatrix{
Y \gc X \gc A \gc X' \gc Y' \ar[drr]|{Y \gc \Delta(f,g,\phi,\phi') \gc Y'} \ar[d]_(.45){Y \gc f \gc X' \gc Y'} \ar[rr]^{Y \gc X \gc \phi \gc Y'} & & Y \gc X \gc B \gc Y' \ar[d]|(.45){Y \gc g \gc Y'}\ar[rr]^{Y \gc X \gc \psi} & & Y \gc X \gc C \ar[d]^(.45){Y \gc h}\\
Y \gc A' \gc X' \gc Y' \ar[d]_{f' \gc X' \gc Y'}\ar[rr]_{Y \gc \phi' \gc Y'} & & Y \gc B' \gc Y' \ar[d]_{g' \gc Y'} \ar[rr]^{Y \gc \psi'} \ar[drr]|{\Delta(g',h',\psi',\psi'')} & & Y \gc C' \ar[d]^{h'}\\
A'' \gc X' \gc Y' \ar[rr]^{\phi'' \gc Y'} & & B'' \gc Y' \ar[rr]^{\psi''} & & C''
}
$$
By composing with the isomorphism $Y \otimes X \gc A \gc X' \otimes Y' \xrightarrow{\sim} Y \gc X \gc A \gc X' \gc Y'$, we obtain a commutative envelope diagram
$$
\xymatrix{
Y \otimes X \gc A \gc X' \otimes Y' \ar[d]_{f' \circ f \gc X' \otimes Y' }\ar[rr]^(.6){Y \otimes X \gc \psi \circ \phi} & & Y \otimes X \gc B \ar[d]^{h \circ h'}\\
A'' \gc X' \otimes Y' \ar[rr]^(.6){\psi'' \circ \phi''}& & C''
}
$$
in which the common composite is $\Delta(g',h',\psi',\psi'') \circ \Delta(f,g,\phi,\phi')$.
\end{proof}

\begin{para}\label{para:horiz_vert_comp}
We call the first and second bigraded squares constructed in \ref{thm:paste_bigr_sq} the \textbf{horizontal composite} and \textbf{vertical composite}, respectively. These operations satisfy the relevant \textit{interchange law}, namely that in the situation of \ref{thm:paste_bigr_sq} if we write $s = (f,g,\phi,\phi')$, $s' = (f',g',\phi',\phi'')$, $t = (g,h,\psi,\psi')$, $t' = (g',h',\psi',\psi'')$ and write horizontal composition as $\circ$ and vertical composition as $*$, then $(t' * t) \circ (s' * s)  = (t' \circ s') * (t \circ s)$. The common value in the latter equation is the rightmost square in \ref{thm:paste_bigr_sq}, and the statement that it is a bigraded square can be deduced using horizontal and vertical composition, but in \ref{thm:paste_bigr_sq} we treat it directly to facilitate the characterization of its diagonal. Neither horizontal nor vertical composition is strictly associative in general; we leave for future work the question whether bigraded squares are the cells of a double bicategory in the sense of \cite[\S 1.6]{Ver:PhDThesis}.
\end{para}

\section{Graded functor categories valued in bigraded categories}\label{sec:gr_func_cat}

\begin{defn}\label{defn:graded_transf}
Let $\V$ and $\W$ be monoidal categories, let $\A$ be a left $\V$-graded category, and let $\C$ be a $\V$-$\W$-bigraded category. Given left $\V$-graded functors $F,G:\A \to \C$ and an object $X' \in \ob\W$, a \textbf{graded transformation} $\phi:F \gc X' \Rightarrow G$ is a family of graded morphisms $\phi_A:FA \gc X' \to GA$ $(A \in \ob\A)$ in $\C$ that are \textbf{left $\V$-graded natural in} $A \in \A$ in the sense that the following envelope diagram commutes for every graded morphism $f:X \gc A \to B$ in $\A$:
\begin{equation}\label{eq:defn_gr_func_cat}
\xymatrix{
X \gc FA \gc X' \ar[r]^(.55){X \gc \phi_A} \ar[d]_{Ff \gc X'} & X \gc GA \ar[d]^{Gf}\\
FB \gc X' \ar[r]^{\phi_B} & GB
}
\end{equation}
We also say that $(\phi_A)_{A \in \ob\A}$ is \textbf{natural at $f$} if the latter diagram commutes.
\end{defn}

In this definition, we employ the conventions of \ref{defn:bigraded} and \ref{para:cdiag_vgr_cats}, so that $F$ and $G$ are left $\V$-graded functors valued in $\CleftV$, and the components $\phi_A$ of $\phi$ are graded morphisms in $\C_\W$. The commutativity of the envelope diagram \eqref{eq:defn_gr_func_cat} is precisely the requirement that $(Ff,Gf,\phi_A,\phi_B)$ be a bigraded square in $\C$ (\ref{defn:bigr_sq}) and so (by \ref{thm:direct_formulation_of_bigr_squares}) can be expressed as the equation
$$(r_X^{-1},r_{X'}^{-1})^*(Gf \circ \phi_A) = (\ell_X^{-1},\ell_{X'}^{-1})^*(\phi_B \circ Ff)\;:\;(X,X') \gc FA \to GB.$$
If $\V$ and $\W$ are strict monoidal, this is simply the equation $Gf \circ \phi_A = \phi_B \circ Ff$. We call $\phi_f := (Ff,Gf,\phi_A,\phi_B)$ the \textbf{naturality square} for $\phi$ at $f$ and denote it as follows:
$$
\xymatrix{
FA \ar@{-}[d]_{Ff} \ar@{-}[r]^{\phi_A}="s1" & GA \ar@{-}[d]^{Gf}\\
FB \ar@{-}[r]^{\phi_B}="t1" & GB
\ar@{}"s1";"t1"|(0.6){X,X'}
}
$$

\begin{prop}\label{thm:gr_tr_gen_set}
In the situation of \ref{defn:graded_transf}, if $\G$ is a generating set of graded morphisms of $\A$ (\ref{para:vgr_subcat}), then a family $\phi = (\phi_A:FA \gc X' \to GA)_{A \in \ob\A}$ is left $\V$-graded natural in $A \in \A$ iff $\phi$ is natural at every $f \in \G$.
\end{prop}
\begin{proof}
Let $\F$ be the set of all graded morphisms $f$ in $\A$ such that $\phi$ is natural at $f$. Then $\F$ is a $\V$-graded subcategory of $\A$, by \ref{thm:id_bigr_sq}, \ref{thm:reindex_bigr_sq}, \ref{thm:paste_bigr_sq}. Hence if $\G \subseteq \F$ then $\F = \A$ by our hypothesis on $\G$, and the result follows.
\end{proof}

\begin{thm}\label{thm:vgr_func_cat}
Let $\V$ and $\W$ be monoidal categories.  (1) If $\A$ is a left $\V$-graded category and $\C$ is a $\V$-$\W$-bigraded category, then left $\V$-graded functors from $\A$ to $\C$ are the objects of a right $\W$-graded category $\tensor[^\V]{[\A,\C]}{_\W}$ that we denote also by $[\A,\C]$, in which a graded morphism is a graded transformation. (2) Similarly, if $\B$ is a right $\W$-graded category and $\C$ is a $\V$-$\W$-bigraded category, then right $\W$-graded functors from $\B$ to $\C$ are the objects of a left $\V$-graded category $\tensor[_\V]{[\B,\C]}{^\W}$ that we denote also by $[\B,\C]$. With the notation of \ref{defn:bigraded}, $\tensor[^\V]{[\A,\C]}{_\W} = \tensor[_{\W^\rev}]{[\A,\mathrm{C}^*\C]}{^{\V^\rev}}$ and $\tensor[_\V]{[\B,\C]}{^\W} = \tensor[^{\W^\rev}]{[\B,\mathrm{C}^*\C]}{_{\V^\rev}}$.
\end{thm}
\begin{proof}
We prove (1), from which (2) follows by way of the latter equation. We define composition and reindexing in $[\A,\C]$ pointwise in terms of composition and reindexing in $\C_\W$. Explicitly, given graded transformations $\phi:F \gc X' \Rightarrow G$ and $\psi:G \gc Y' \Rightarrow H$, we define $\psi \circ \phi:F \gc X' \otimes Y'\Rightarrow H$ to consist of the composites $\psi_A \circ \phi_A:FA \gc X' \otimes Y' \to HA$ $(A \in \ob\A)$ in $\C_\W$. These constitute a graded transformation because for each $f:X \gc A \to B$ in $\A$ the horizontal composite $\psi_f \circ \phi_f = (Ff,Hf,\psi_A \circ \phi_A,\psi_B \circ \phi_B)$ is a bigraded square by \ref{thm:paste_bigr_sq}. Given $\beta:Z' \to X'$ in $\W$, we define $\beta^*(\phi):F \gc Z' \Rightarrow G$ to consist of the reindexings $\beta^*(\phi_A):FA \gc Z' \to GA$ $(A \in \ob\A)$ in $\C_\W$, which constitute a graded transformation because for each $f:X \gc A \to B$ in $\A$ we can apply \ref{thm:reindex_bigr_sq} to the naturality square $\phi_f = (Ff,Gf,\phi_A,\phi_B)$ to obtain a bigraded square $(1_X,\beta)^*(\phi_f) = (Ff,Gf,\beta^*(\phi_A),\beta^*(\phi_B))$. Identities in $[\A,\C]$ are formed pointwise. The axioms \ref{para:vgr_cat_conc}(I)--(IV) for $[\A,\C]$ follow from those for $\C_\W$.
\end{proof}

\begin{para}\label{para:fcat_huge}
By applying Theorem \ref{thm:vgr_func_cat} relative to the universe $\SET'$ rather than $\SET$, we obtain a similar result for huge $\A$, $\B$, $\C$, thus obtaining a huge right $\W$-graded category $[\A,\C]$ and a huge left $\V$-graded category $[\B,\C]$.
\end{para}

\begin{para}\label{para:vgr_func_cat_other_case}
Given a right $\W$-graded category $\B$, a $\V$-$\W$-bigraded category $\C$, and right $\W$-graded functors $F,G:\B \to \C$ as in part (2) of Theorem \ref{thm:vgr_func_cat}, a graded morphism $\phi:X \gc F \to G$ in $[\B,\C]$ is a \textbf{graded transformation} $\phi:X \gc F \Rightarrow G$, i.e.~a family of graded morphisms $\phi_A:X \gc FA \to GA$ $(A \in \ob\B)$ that is \textbf{right $\W$-graded natural} in the sense that for every graded morphism $f:A \gc X' \to B$ in $\B$ the quadruple $(\phi_A,\phi_B,Ff,Gf)$ is a bigraded square in $\C$, which we express also by saying that $\phi$ is \textbf{natural at $f$}. Indeed, this follows from \ref{para:flip_bigraded_square}.
\end{para}

\begin{para}
In the situation of \ref{thm:vgr_func_cat}, $[\A,\C]_0 = \LGCAT{\V}(\A,\CleftV)$ and $[\B,\C]_0 = \RGCAT{\W}(\B,\C_\W)$.
\end{para}

\begin{exa}\label{exa:vgr_func}
(1). Since every monoidal category $\V$ carries the structure of a $\V$-$\V$-bigraded category, every left $\V$-graded category $\A$ determines a right $\V$-graded category $[\A,\V]$, and every right $\V$-graded category $\B$ determines left $\V$-graded category $[\B,\V]$.

\medskip

\noindent (2). Given any monoid $R$ in $\V$, left $R$-modules in (the left $\V$-graded category underlying) a $\V$-$\W$-bigraded category $\C$ are the objects of a right $\W$-graded category $\tensor[_R]{\textnormal{Mod}}{}(\C) = [R,\C]$. In particular, taking $\V = \W$ and $\C = \V$, we find that left $R$-modules in $\V$ are the objects of a right $\V$-graded category $\tensor[_R]{\textnormal{Mod}}{}$. Given instead a monoid $S$ in $\W$ and a $\V$-$\W$-bigraded category $\C$, we obtain a left $\V$-graded category $\tensor{\textnormal{Mod}}{_S}(\C) = [S^\circ,\C]$ whose objects are right $S$-modules in $\C$.

\medskip

\noindent (3). Let $\C$ be a $\V$-$\W$-bigraded category, and let $X$ be an object of $\V$. 
By \ref{exa:two_sub_x} and \ref{thm:vgr_func_cat}, there is a right $\W$-graded category $[\tensor[_X]{\mathbbm{2}}{},\C]$ in which an object is given by a triple $(A,A',f)$ consisting of objects $A,A'$ of $\C$ and a graded morphism $f:X \gc A \to A'$ in $\C$. By \ref{exa:two_sub_x}, the singleton $\{u:X \gc 0 \to 1\}$ is a generating set of graded morphisms for $\tensor[_X]{\mathbbm{2}}{}$, so it follows from \ref{thm:gr_tr_gen_set} that a graded morphism $(\phi,\phi'):(A,A',f) \gc X' \to (B,B',g)$ in $[\tensor[_X]{\mathbbm{2}}{},\C]$ is given by an $(X,X')$-bigraded square $(f,g,\phi,\phi')$ in $\C$.

\medskip

For example, if $\C$ is a $\Box$-$\Box$-bigraded category for the restricted cubical site $\Box$ (\ref{exa:actegories}, \ref{exa:gr_cats}) and we take $X = 2$, then $[\tensor[_2]{\mathbbm{2}}{},\C]$ is a right $\Box$-graded category whose objects are \textit{homotopies} $f:2 \gc A \to A'$ in $\C$ in the sense of \ref{exa:gr_cats}. For example, taking $\C = \Top$ (\ref{exa:top-box-box-biact}), the objects of $[\tensor[_2]{\mathbbm{2}}{},\Top]$ are homotopies in the usual sense.

\medskip

\noindent (4). Given a monoidal category $\V$, we now consider the special case of \ref{thm:vgr_func_cat}(1) where we take $\W = 1$ to be the terminal monoidal category, for which right $1$-graded categories are ordinary categories. Given an ordinary category $\B$ and a left $\V$-graded category $\C$, we may regard $\C$ as a $\V$-$1$-bigraded category (\ref{exa:v1_1v_bigraded}), and thus we obtain a left $\V$-graded category $[\B,\C] = \tensor[_\V]{[\B,\C]}{^1}$ with $[\B,\C]_0 = \CAT(\B,\C_0)$.

\medskip

\noindent (5). Given an ordinary category $\B$ and a $\V$-$\W$-bigraded category $\C$ for monoidal categories $\V$ and $\W$, we may regard $\C$ as a left $(\V \times \W^\rev)$-graded category and apply (4) to obtain a left $(\V \times \W^\rev)$-graded category $[\B,\C] = \tensor[_{\V \times \W^\rev}]{[\B,\C]}{^1}$, so $[\B,\C]$ in this case is a $\V$-$\W$-bigraded category with $[\B,\C]_0 = \CAT(\B,\C_0)$.  Consequently, every full subcategory of $\CAT(\B,\V)$ underlies a $\V$-$\V$-bigraded category; e.g.~for any Lawvere theory $\T$, the category of $\T$-algebras in $\V$ underlies a $\V$-$\V$-bigraded category.
\end{exa}

Since a $(\V \times \W)$-graded category is a $\V$-$\W^\rev$-bigraded category, Theorem \ref{thm:vgr_func_cat} entails the following result in terms of left gradings only, with the notations of \ref{para:cats_gr_prod}:

\begin{cor}\label{thm:func_cat_for_product_graded_categories}
Let $\C$ be a (left) $(\V \times \W)$-graded category. Then (1) each $\V$-graded category $\A$ determines a (left) $\W$-graded category $[\A,\C]_r := \tensor[^\V]{[\A,\C]}{_{\W^\rev}}$ whose objects are $\V$-graded functors $F:\A \to U^*_\ell\C$, and (2) each $\W$-graded category $\B$ determines a $\V$-graded category $[\B,\C]_\ell := \tensor[_\V]{[\B,\C]}{^{\W^\rev}}$ whose objects are $\W$-graded functors $F:\B \to U^*_r\C$. With the notation of \ref{para:cats_gr_prod}, $[\A,\C]_r = [\A,\mathrm{C}^*\C]_\ell$ and $[\B,\C]_\ell = [\B,\mathrm{C}^*\C]_r$.
\end{cor}

The reason for the notations introduced in \ref{thm:func_cat_for_product_graded_categories} will be apparent in \ref{para:func_cat_rep_gbifun_for_product-graded_cat},  \ref{para:enr_func_cat_gl}--\ref{thm:venr_func_cat_gr_func_cat}.

\begin{rem}\label{rem:charn_gr_morphs_in_func_cat_prod_graded}
In the situation of \ref{thm:func_cat_for_product_graded_categories}(1), if $F,G:\A \to U^*_\ell\C$ are $\V$-graded functors, then a graded morphism $\phi:X' \gc F \to G$ in the $\W$-graded category $[\A,\C]_r$ is a family of graded morphisms $\phi_A:X' \gc FA \to GA$ in $U_r^*\C$, i.e.~$\phi_A:(I,X') \gc FA \to GA$ in $\C$ $(A \in \ob\A)$, such that for each graded morphism $f:X \gc A \to B$ in $\A$ the quadruple $(Ff,Gf,\phi_A,\phi_B)$ is a bigraded square in $\C$, equivalently (by \ref{rem:bigr_sq_prod_gr}) $(r_X^{-1},\ell_{X'}^{-1})^*(Gf \circ \phi_A) = (\ell_X^{-1},r_{X'}^{-1})^*(\phi_B \circ Ff):(X,X') \gc FA \to GB$. Similarly, in view of \ref{para:vgr_func_cat_other_case}, if $F,G:\B \to U_r^*\C$ are $\W$-graded functors, then a graded morphism $\phi:X \gc F \to G$ in $[\B,\C]_\ell$ is a family of graded morphisms $\phi_A:X \gc FA \to GA$ in $U_\ell^*\C$ $(A \in \ob\B)$ such that for every $f:X' \gc A \to B$ in $\B$, $(\phi_A,\phi_B,Ff,Gf)$ is a bigraded square in $\C$.
\end{rem}

\section{Graded bifunctors and bigraded products}\label{sec:bifunctors}

\begin{defn}\label{defn:sesqfunc}
Let $\V$ and $\W$ be monoidal categories. Let $\A$ be a left $\V$-graded category, let $\B$ be a right $\W$-graded category, and let $\C$ be a $\V$-$\W$-bigraded category. A \textbf{($\V$-$\W$-)graded sesquifunctor} $F:\A,\B \to \C$ consists of (1) an assignment to each pair of objects $A \in \ob\A$, $B \in \ob\B$ an object $F(A,B) \in \ob\C$, (2) for each $B \in \ob\B$ a left $\V$-graded functor $F(-,B):\A \to \C$ given on objects by $A \mapsto F(A,B)$, and (3) for each $A \in \ob\A$ a right $\W$-graded functor $F(A,-):\B \to \C$ given on objects by $B \mapsto F(A,B)$. In view of the conventions of \ref{defn:bigraded}, the left $\V$-graded functors $F(-,B)$ have codomain $\CleftV = U^*_\ell\C$, while the right $\W$-graded functors $F(A,-)$ have codomain $\C_\W = U^*_r\C$.
\end{defn}

\begin{defn}\label{defn:commutes_under_F}
Let $F:\A,\B \to \C$ be a $\V$-$\W$-graded sesquifunctor, and let $f:X \gc A \to A'$ in $\A$ and $g:B \gc X' \to B'$ in $\B$ be graded morphisms. We say that \textbf{$f$ commutes with $g$ under $F$} and write $f \perp_F g$ if the following envelope diagram commutes:
\begin{equation}\label{eq:bifunc_ax}
\xymatrix{
X\gc F(A,B)\gc X' \ar[rr]^{X\gc F(A,g)}\ar[d]_{F(f,B) \gc X'} & & X \gc F(A,B')\ar[d]^{F(f,B')}\\
F(A',B) \gc X' \ar[rr]^{F(A',g)} & & F(A',B')
}
\end{equation}
The commutativity of this diagram requires precisely that
$$F_{fg} := (F(f,B),F(f,B'),F(A,g),F(A',g))$$
be a bigraded square in $\C$, which we may depict as follows:
$$
\xymatrix{
F(A,B) \ar@{-}[d]_{F(f,B)} \ar@{-}[r]^{F(A,g)}="s1" & F(A,B') \ar@{-}[d]^{F(f,B')}\\
F(A',B) \ar@{-}[r]^{F(A',g)}="t1" & F(A',B')
\ar@{}"s1";"t1"|(0.6){X,X'}
}
$$
\end{defn}

The terminology and notation that we introduce in \ref{defn:commutes_under_F} is a variation on that used in the very special case of Kronecker products for enriched algebraic theories in \cite[\S 5]{Lu:Cmt}.

\begin{para}\label{para:commutes_via_nat}
In the situation of \ref{defn:commutes_under_F}, the following are equivalent (with the terminology of \ref{defn:graded_transf} and \ref{para:vgr_func_cat_other_case}): (1) $f \perp_F g$, (2) the graded morphisms $F(A,g):F(A,B) \gc X' \to F(A,B')$ $(A \in \ob\A)$ are natural at $f$, (3) the graded morphisms $F(f,B):X \gc F(A,B) \to F(A',B)$ $(B \in \ob\B)$ are natural at $g$.
\end{para}

\begin{defn}\label{defn:bifunc}
A \textbf{($\V$-$\W$-)graded bifunctor} $F:\A,\B \to \C$ is a $\V$-$\W$-graded sesquifunctor such that $f \perp_F g$ for all graded morphisms $f$ in $\A$ and $g$ in $\B$.
\end{defn}

\begin{rem}\label{rem:bifun_vtimesw_gr}
The preceding definition specializes to a notion of graded bifunctor valued in a $(\V \times \W)$-graded category $\C$, as follows: If $\A$ is a (left) $\V$-graded category and $\B$ is a (left) $\W$-graded category, then a \textit{graded sesquifunctor (resp.~bifunctor)} $F:\A,\B \to \C$ is by definition a $\V$-$\W^\rev$-graded sesquifunctor (resp.~bifunctor). We now unpack the details of this. A graded sesquifunctor $F:\A,\B \to \C$ consists of $\V$-graded functors $F(-,B):\A \to U_\ell^*\C$ $(B \in \ob\B)$ and $\W$-graded functors $F(A,-):\B \to U_r^*\C$ $(A \in \ob\A)$ that agree on objects. Here we must retain the notations $U_\ell^*\C$ and $U_r^*\C$ that are rendered unnecessary in \ref{defn:sesqfunc} through the use of right and left gradings. The graded sesquifunctor $F$ is a graded bifunctor iff for all $f:X \gc A \to A'$ in $\A$ and $g:X' \gc B \to B'$ in $\B$ the graded morphisms $F(f,B)$, $F(f,B')$ in $U_\ell^*\C$ and $F(A,g)$, $F(A',g)$ in $U_r^*\C$ constitute a bigraded square in $\C$ in the sense of \ref{rem:bigr_sq_prod_gr}. At some cost of economy\footnote{By contrast, the use of envelope diagrams for bigraded categories in \ref{defn:commutes_under_F} obviates this. Indeed, ease of expression is one of the reasons why we have chosen to formulate graded bifunctors in terms bigraded categories, among other reasons discussed in \ref{para:cats_gr_pr_vs_vw_bigr}.} we can directly express the latter in elementary terms as in \ref{rem:bigr_sq_prod_gr}.
\end{rem}

\begin{prop}\label{para:bifun_gen_sets}
Let $F:\A,\B \to \C$ be a $\V$-$\W$-graded sesquifunctor, and let $\F$ and $\G$ be generating sets of graded morphisms for $\A$ and $\B$, respectively (\ref{para:vgr_subcat}). Then $F$ is a $\V$-$\W$-graded bifunctor iff $f \perp_F g$ for all graded morphisms $f$ in $\F$ and $g$ in $\G$.
\end{prop}
\begin{proof}
By \ref{thm:gr_tr_gen_set} and \ref{para:commutes_via_nat}, $F$ is a graded bifunctor iff $f \perp_F g$ for all graded morphisms $f$ in $\F$ and $g$ in $\B$, and by applying \ref{thm:gr_tr_gen_set} and \ref{para:commutes_via_nat} again the result follows.
\end{proof}

\begin{exa}
Let $\C$ be a $\V$-$\W$-bigraded category, $R$ a monoid in $\V$, and $S$ a monoid in $\W$. By default we regard $R$ and $S$ as one-object left $\V$- and $\W$-categories, respectively, so $S^\circ$ is a one-object right $\W$-category (\ref{para:lr_vcats}). By definition, an \textbf{$R$-$S$-bimodule in $\C$} is a $\V$-$\W$-graded bifunctor $M:R,S^\circ \to \C$ and (by \ref{exa:loc_rep_univ_elts_gen},  \ref{rem:rmod_in_vgr_cat}, \ref{para:bifun_gen_sets}) is equivalently given by an object $M$ of $\C$ and graded morphisms $\lambda:R \gc M \to M$ and $\rho:M \gc S \to M$ in $\C$ such that $(M,\lambda)$ is a left $R$-module in $\C$, $(M,\rho)$ is a right $S$-module in $\C$, and $\lambda$ commutes with $\rho$ in the sense that $(\lambda,\lambda,\rho,\rho)$ is a bigraded square.
\end{exa}

\begin{para}\label{para:2fun_bifun}
A \textbf{transformation} $\theta:F \Rightarrow G:\A,\B \to \C$ of $\V$-$\W$-graded sesquifunctors $F,G$ is by definition a family of morphisms $\theta_{AB}:F(A,B) \to G(A,B)$ in $\C_0$ $(A \in \ob\A, B \in \ob\B)$ such that (i) for each $B \in \ob\B$ the morphisms $\theta_{AB}$ are left $\V$-graded natural in $A \in \A$ and (ii) for each $A \in \ob\A$ the morphisms $\theta_{AB}$ are right $\W$-graded natural in $B \in \B$. Note that (i) and (ii) are well defined since $\C_0 = (\CleftV)_0 = (\C_\W)_0$ by \ref{para:cats_gr_prod}. Graded sesquifunctors and their transformations form a category $\tensor[_\V]{\textsf{GSes}}{_\W}(\A,\B;\C)$ with a full subcategory $\tensor[_\V]{\textsf{GBif}}{_\W}(\A,\B;\C)$ spanned by the graded bifunctors; we write these also as $\textsf{GSes}(\A,\B;\C)$ and $\textsf{GBif}(\A,\B;\C)$. Note that $\tensor[_\V]{\textsf{GSes}}{_\W}(\A,\B;\C)$ is a (conical) pullback in $\CAT$ of an evident cospan:
\begin{equation}\label{eq:seqfun_pb}\LGCAT{\V}(\A,U_\ell^*\C)^{\ob\B} \to \C_0^{\ob\A \times \ob\B} \leftarrow \RGCAT{\W}(\B,U_r^*\C)^{\ob\A}.\end{equation}
For fixed $\A$ and $\B$, these expressions are 2-functorial in $\C \in \GCAT{\V}{\W}$, so we can regard \eqref{eq:seqfun_pb} as a cospan of 2-functors whose pointwise (conical) pullback is a 2-functor $\tensor[_\V]{\textsf{GSes}}{_\W}(\A,\B;-):\GCAT{\V}{\W} \to \CAT$. In particular, if $F:\A,\B \to \C$ is a graded sesquifunctor and $R:\C \to \C'$ is a $\V$-$\W$-bigraded functor then we obtain a graded sesquifunctor $RF:\A,\B \to \C'$. Moreover, if $F$ is a graded bifunctor, then $RF$ is a graded bifunctor since $R$ preserves bigraded squares (\ref{thm:pres_bigr_sq}). Thus we obtain also a 2-functor $\tensor[_\V]{\textsf{GBif}}{_\W}(\A,\B;-):\GCAT{\V}{\W} \to \CAT$. The expressions in \eqref{eq:seqfun_pb} are \textit{not} in general 2-functorial in (both) $\A$ and $\B$, even though they are functorial on 1-cells in all three variables jointly; but see \ref{thm:2fun_bifun} regarding $\tensor[_\V]{\textsf{GBif}}{_\W}$. 
\end{para}

\begin{defn}\label{defn:boxtimes}
Let $\A$ be a left $\V$-graded category, and a let $\B$ be a right $\W$-graded category. The \textbf{($\V$-$\W$-)bigraded product} of $\A$ and $\B$ is the $\V$-$\W$-bigraded category $\A \boxtimes \B$ defined as follows: Firstly, $\ob(\A \boxtimes \B) = \ob \A \times \ob \B$. Secondly, a graded morphism $(f,g):X \gc (A,B) \gc X' \to (A',B')$ in $\A \boxtimes \B$ is a pair consisting of graded morphisms $f:X \gc A \to A'$ in $\A$ and $g:B \gc X' \to B'$ in $\B$. Composition and reindexing in $\A \boxtimes \B$ are defined componentwise, in the sense that $(f',g') \circ (f,g) = (f' \circ f,g' \circ g)$ for graded morphisms $(f,g):X \gc (A,B) \gc X' \to (A',B')$ and $(f',g'):Y \gc (A',B') \gc Y' \to (A'',B'')$, and if we are given morphisms $\alpha:Z \to X$ in $\V$ and $\beta:Z' \to X'$ in $\W$ then the reindexing of $(f,g)$ along $(\alpha,\beta)$ in $\A \boxtimes \B$ is $(\alpha,\beta)^*(f,g) = (\alpha^*(f),\beta^*(g)):Z \gc (A,B) \gc Z' \to (A',B')$.  Identities in $\A \boxtimes \B$ are also given componentwise.  Note that $\A \boxtimes \B$ satisfies the axioms \ref{para:vgr_cat_conc}(I)--(IV) because $\A$ and $\B$ do.
\end{defn}

\begin{prop}
There is a 2-functor $\boxtimes:\LGCAT{\V} \times \RGCAT{\W} \to \GCAT{\V}{\W}$ that is given on objects by $(\A,\B) \mapsto \A \boxtimes \B$. In detail, (1) given 1-cells $P:\A \to \A'$ in $\LGCAT{\V}$ and $Q:\B \to \B'$ in $\RGCAT{\W}$, we obtain a $\V$-$\W$-bigraded functor $P \boxtimes Q:\A \boxtimes \B \to \A' \boxtimes \B'$ given on objects by $(A,B) \mapsto (PA,QB)$ and on graded morphisms by $(f,g) \mapsto (Pf,Qg)$, and (2) given 2-cells $\xi:P \Rightarrow P':\A \to \A'$ in $\LGCAT{\V}$ and $\zeta:Q \Rightarrow Q':\B \to \B'$ in $\RGCAT{\W}$, we obtain a 2-cell $\xi \boxtimes \zeta:P \boxtimes Q \Rightarrow P' \boxtimes Q'$ in $\GCAT{\V}{\W}$ consisting of the morphisms $(\xi_A,\zeta_B):(PA,QB) \to (P'A,Q'B)$ in $(\A' \boxtimes \B')_0 = \A'_0 \times \B'_0$ $(A \in \ob\A, B \in \ob\B)$.
\end{prop}
\begin{proof}
The verification is straightforward.
\end{proof}

\begin{prop}\label{thm:bigr_prod}
(1) In the situation of Definition \ref{defn:boxtimes}, there is a graded bifunctor
$$\mathsf{Pair} = (-,?):\A,\B \to \A \boxtimes \B$$
whose left $\V$-graded functors $(-,B):\A \to \A \boxtimes \B$ $(B \in \ob\B)$ are given on objects by $A \mapsto (A,B)$ and on graded morphisms by $f \mapsto (f,\gid_B)$, and whose right $\W$-graded functors $(A,-):\B \to \A \boxtimes \B$ $(A \in \ob\A)$ are given analogously. (2) For each graded morphism $(f,g):X \gc (A,B) \gc X' \to (A',B')$ in $\A \boxtimes \B$, the bigraded square $\mathsf{Pair}_{fg} = ((f,B),(f,B'),(A,g),(A',g))$ has diagonal $(f,g)$. (3) The set of all graded morphisms of the form $(f,B)$ or $(A,g)$ is a generating set of graded morphisms in $\A \boxtimes \B$.
\end{prop}
\begin{proof}
In view of the conventions of \ref{defn:bigraded} and \ref{defn:sesqfunc}, $(-,B)$ is required to be a left $\V$-graded functor valued in the left $\V$-graded category underlying $\A \boxtimes \B$, and this is easily seen in view of the definition of the latter $\V$-graded category (\ref{para:cats_gr_prod}, \ref{defn:bigraded}). Similar remarks apply to $(A,-)$. Let us write $(f,B):(X,I) \gc (A,B) \to (A',B)$, $(f,B'):(X,I) \gc (A,B') \to (A',B')$, $(A,g):(I,X') \gc (A,B) \to (A,B')$, $(A',g):(I,X') \gc (A',B) \to (A',B')$ as in \ref{thm:direct_formulation_of_bigr_squares}. Computing the composites $(f,B') \circ (A,g) = (f \circ \gid_A,\gid_{B'} \circ g):(X \otimes I,X' \otimes I) \gc (A,B) \to (A',B')$ and $(A',g) \circ (f,B) = (\gid_{A'} \circ f,g \circ \gid_B):(I \otimes X,I \otimes X') \gc (A,B) \to (A',B')$ and applying the essential identity axiom (\ref{para:vgr_cat_conc}), statements (1) and (2) follow from \ref{thm:direct_formulation_of_bigr_squares}. The diagonal $(f,g)$ in (2) is a reindexing of the composite $(f,B') \circ (A,g)$, so (3) holds.
\end{proof}

\begin{thm}\label{thm:bifun}
Let $\A$ be a left $\V$-graded category and $\B$ a right $\W$-graded category. Then there are isomorphisms
\begin{equation}\label{eq:bifun_rep}\GCAT{\V}{\W}(\A \boxtimes \B,\C) \cong \tensor[_\V]{\textnormal{\textsf{GBif}}}{_\W}(\A,\B;\C)\;,\end{equation}
2-natural in $\C \in \GCAT{\V}{\W}$, given by composition with $\mathsf{Pair}$. Equivalently, $\mathsf{Pair}$ is the unit of a representation of the 2-functor $\tensor[_\V]{\textnormal{\textsf{GBif}}}{_\W}(\A,\B;-):\GCAT{\V}{\W} \to \CAT$.
\end{thm}
\begin{proof}
We have a functor $\widetilde{\mathsf{Pair}} = (-)\mathsf{Pair}:\GCAT{\V}{\W}(\A \boxtimes \B,\C) \to \textnormal{\textsf{GBif}}(\A,\B;\C)$, which we first show is bijective on objects. Let $F:\A,\B \to \C$ be a graded bifunctor. Any  bigraded functor $G:\A \boxtimes \B \to \C$ with $G\mathsf{Pair} = F$ must be given on objects by $(A,B) \mapsto F(A,B)$ and, in view of \ref{thm:pres_bigr_sq} and \ref{thm:bigr_prod}, must send each graded morphism $(f,g):X \gc (A,B) \gc X' \to (A',B')$ in $\A \boxtimes \B$ to the diagonal $\Delta(F_{fg})$ of the bigraded square $F_{fg}$ of \ref{defn:commutes_under_F}, so that
$$G(f,g) = \Delta(F_{fg})\;:\;X \gc F(A,B) \gc X' \to F(A',B').$$
For the existence of $G$, let us define $G:\A \boxtimes \B \to \C$ in the latter way and show that $G$ is a bigraded functor; the statement that $G\mathsf{Pair} = F$ then follows easily, using \ref{thm:id_bigr_sq}. Firstly, $G$ sends the identity $(\gid_A,\gid_B)$ on $(A,B)$ to the identity on $G(A,B) = F(A,B)$ as a consequence of \ref{thm:id_bigr_sq} and the fact that $F(A,-)$ and $F(-,B)$ preserve identities. The reindexing $(\alpha,\beta)^*(f,g)$ of a graded morphism $(f,g):X \gc (A,B) \gc X' \to (A',B')$ in $\A \boxtimes \B$ along a morphism $(\alpha,\beta):(Y,Y) \to (X,X')$ in $\V \times \W^\rev$ is by definition $(\alpha^*(f),\beta^*(g))$ and so is sent by $G$ to the diagonal of the bigraded square
$$
\xymatrix{
F(A,B) \ar@{-}[d]_{F(\alpha^*(f),B)\,=\,\alpha^*(F(f,B))} \ar@{-}[rrrr]^{F(A,\beta^*(g))\,=\,\beta^*(F(A,g))}="s1" & & & & F(A,B') \ar@{-}[d]^{F(\alpha^*(f),B')\,=\,\alpha^*(F(f,B'))}\\
F(A',B) \ar@{-}[rrrr]^{F(A',\beta^*(g))\,=\,\beta^*(F(A',g))}="t1" & & & & F(A',B')
\ar@{}"s1";"t1"|(0.6){Y,Y'}
}
$$
which by \ref{thm:reindex_bigr_sq} is precisely the reindexing $(\alpha,\beta)^*(\Delta(F_{fg})) = (\alpha,\beta)^*(G(f,g))$. Regarding composition, suppose we are also given $(f',g'):Y \gc (A',B') \gc Y' \to (A'',B'')$ in $\A \boxtimes \B$. Then the bigraded squares $F_{fg}$, $F_{fg'}$, $F_{f'g}$, $F_{f'g'}$ take the following forms:
\begin{equation}\label{eq:pf_pres_comp}
\xymatrix{
F(A,B) \ar@{-}[d]_{F(f,B)} \ar@{-}[rr]^{F(A,g)}="s1" && F(A,B')\ar@{-}[d]|{F(f,B')} \ar@{-}[rr]^{F(A,g')}="s2" && F(A,B'') \ar@{-}[d]^{F(f,B'')}\\
F(A',B) \ar@{-}[d]_{F(f',B)} \ar@{-}[rr]^{F(A',g)}="t1" && F(A',B') \ar@{-}[d]|{F(f',B')} \ar@{-}[rr]^{F(A',g')}="t2" && F(A',B'') \ar@{-}[d]^{F(f',B'')} \\
F(A'',B) \ar@{-}[rr]^{F(A'',g)}="u1" && F(A'',B') \ar@{-}[rr]^{F(A'',g')}="u2" && F(A'',B'')
\ar@{}"s1";"t1"|(0.6){X,X'}
\ar@{}"s2";"t2"|(0.6){X,Y'}
\ar@{}"t1";"u1"|(0.6){Y,X'}
\ar@{}"t2";"u2"|(0.6){Y,Y'}
}
\end{equation}
By \ref{thm:paste_bigr_sq} these compose to produce a bigraded square
\begin{equation}\label{eq:pf_comp2}
\xymatrix{
F(A,B) \ar@{-}[d]_{F(f',B) \circ F(f,B)} \ar@{-}[rrr]^{F(A,g') \circ F(A,g)}="s1" & & & F(A,B'') \ar@{-}[d]^{F(f',B'') \circ F(f,B'')}\\
F(A'',B) \ar@{-}[rrr]^{F(A'',g') \circ F(A'',g)}="t1" & & & F(A'',B'')
\ar@{}"s1";"t1"|(0.6){Y \otimes X,X' \otimes Y'}
}
\end{equation}
whose diagonal is the composite of the diagonals in \eqref{eq:pf_pres_comp} and so is precisely $\Delta(F_{f'g'}) \circ \Delta(F_{fg}) = G(f',g') \circ G(f,g)$. But the graded sesquifunctoriality of $F$ entails that \eqref{eq:pf_comp2} is precisely the square $F_{f' \circ f,g' \circ g'}$, and hence $G(f' \circ f,g' \circ g) = \Delta(F_{f' \circ f,g' \circ g'}) = G(f',g') \circ G(f,g)$. This shows that $\widetilde{\mathsf{Pair}}$ is bijective on objects; it is also fully faithful, because if $G,H:\A \boxtimes \B \to \C$ are bigraded functors, then by \ref{thm:gr_tr_gen_set} a family of morphisms $\delta = (\delta_{(A,B)}:G(A,B) \to H(A,B))$ is a 2-cell $G \Rightarrow H$  in $\GCAT{\V}{\W} = \LGCAT{\V \times \W^\rev}$ iff $\delta$ is natural at each graded morphism in the generating set given in \ref{thm:bigr_prod}(3), iff $\delta$ is a transformation $G\mathsf{Pair} \Rightarrow H\mathsf{Pair}$.
\end{proof}

\begin{cor}\label{thm:2fun_bifun}
There is a unique 2-functor
$$\tensor[_\V]{\textnormal{\textsf{GBif}}}{_\W}\;:\;\LGCAT{\V}^\op \times \textnormal{GCAT}_\W^\op \times \GCAT{\V}{\W} \to \CAT$$
that is given on objects by $(\A,\B,\C) \mapsto \tensor[_\V]{\textnormal{\textsf{GBif}}}{_\W}(\A,\B;\C)$ and makes the isomorphisms in \eqref{eq:bifun_rep} 2-natural in $\A \in \LGCAT{\V}$, $\B \in \RGCAT{\W}$, $\C \in \GCAT{\V}{\W}$.

\end{cor}

Applying the above definitions and results relative to $\SET'$, we may consider graded bifunctors and sesquifunctors $F:\A,\B \to \C$ also when $\A,\B,\C$ are huge, and we thus obtain a 2-functor $\tensor[_\V]{\textsf{GBIF}}{_\W}$ valued in the 2-category of huge categories.

Recalling that $\LGCAT{\V} = \RGCAT{\V^\rev}$ and $\RGCAT{\W} = \LGCAT{\W^\rev}$, we also have an isomorphism $\mathrm{C}^*:\GCAT{\V}{\W} \xrightarrow{\sim} \GCAT{\W^\rev}{\V^\rev}$ by \ref{defn:bigraded}. By \ref{para:flip_bigraded_square}, each $\V$-$\W$-graded bifunctor $F:\A,\B \to \C$ determines a $\W^\rev$-$\V^\rev$-graded bifunctor $F^\mathsf{swap}:\B,\A \to \mathrm{C}^*\C$ given by $F^\mathsf{swap}(-,A) = F(A,-)$ and $F^\mathsf{swap}(B,-) = F(-,B)$. Thus we obtain

\begin{prop}\label{thm:swap_gr_bifunc}
$\tensor[_\V]{\textnormal{\textsf{GBif}}}{_\W}(\A,\B;\C) \cong \tensor[_{\W^\rev}]{\textnormal{\textsf{GBif}}}{_{\V^\rev}}(\B,\A;\mathrm{C}^*\C)$, 2-naturally in $\A \in \LGCAT{\V}$, $\B \in \RGCAT{\W}$, $\C \in \GCAT{\V}{\W}$.
\end{prop}

\section{The relation between graded functor categories and bifunctors}\label{sec:reln_gr_func_cat_bifunc}

\begin{para}[\textbf{The evaluation bifunctor}]
Let $\C$ be a $\V$-$\W$-bigraded category and $\B$ a right $\W$-graded category. By Theorem \ref{thm:vgr_func_cat} we have a left $\V$-graded category $[\B,\C] = \tensor[_\V]{[\B,\C]}{^\W}$, and we now define a graded bifunctor $\mathsf{Ev}:[\B,\C],\B \to \C$ given on objects by $\mathsf{Ev}(F,B) = FB$.  For each right $\W$-graded functor $F:\B \to \C$, we define $\mathsf{Ev}(F,-) = F$.  For each object $B$ of $\B$, we define a left $\V$-graded functor $\mathsf{Ev}(-,B):[\B,\C] \to \C$ that sends each graded morphism $\phi:X \gc F \to G$ in $[\B,\C]$ to the graded morphism $\phi_B:X \gc FB \to GB$ in $\C$. Indeed, $\mathsf{Ev}(-,B)$ preserves composition, identities, and reindexing because these are defined pointwise in $[\B,\C]$. These data satisfy axiom \eqref{eq:bifunc_ax}, because of the naturality condition in the definition of graded morphisms in $[\B,\C]$ (\ref{para:vgr_func_cat_other_case}).
\end{para}

\begin{thm}\label{thm:func_cat_rep_bifun}
Let $\B$ be a right $\W$-graded category, and let $\C$ be a $\V$-$\W$-bigraded category.  Then there are isomorphisms
$$\LGCAT{\V}(\A,\tensor[_\V]{[\B,\C]}{^\W}) \cong \tensor[_\V]{\textnormal{\textsf{GBif}}}{_\W}(\A, \B;\C)$$
2-natural in $\A \in \LGCAT{\V}$. I.e., the left $\V$-category $\tensor[_\V]{[\B,\C]}{^\W}$ represents the 2-functor $\tensor[_\V]{\textnormal{\textsf{GBif}}}{_\W}(-,\B;\C):\LGCAT{\V}^\op \to \CAT$.
\end{thm}
\begin{proof}
As the counit of the representation, we employ the evaluation bifunctor $\mathsf{Ev}:[\B,\C],\B \to \C$. Given a left $\V$-graded category $\A$, it suffices to show that the induced functor
$$\widetilde{\mathsf{Ev}}:\LGCAT{\V}(\A,[\B,\C]) \to \tensor[_\V]{\textsf{GBif}}{_\W}(\A,\B;\C)$$
is an isomorphism, where $\widetilde{\mathsf{Ev}}H = \mathsf{Ev}(H-,?):\A,\B \to \C$ for each $H:\A \to [\B,\C]$. Let $F:\A,\B \to \C$ be a graded bifunctor.  Then we obtain a left $\V$-graded functor $F^\sharp:\A \to [\B,\C]$ given on objects by $F^\sharp A = F(A,-)$ and on graded morphisms as follows.  Given $f:X \gc A \to A'$ in $\A$, the graded morphisms $F(f,B)$ $(B \in \ob\B)$ constitute a graded transformation $F^\sharp f = F(f,-):X \gc F(A,-) \Rightarrow F(A',-)$ by \ref{para:commutes_via_nat}.  The preservation of composition, identities, and reindexing by $F^\sharp$ follows from the fact that these are defined pointwise in $[\B,\C]$ and are preserved by $F$ in its first argument.  The resulting left $\V$-graded functor $F^\sharp:\A \to [\B,\C]$ clearly satisfies $\mathsf{Ev}(F^\sharp-,?) = F$ and is unique in this capacity.  Hence $\widetilde{\mathsf{Ev}}$ is bijective on objects, and it follows from the definition of morphisms in $\tensor[_\V]{\textsf{GBif}}{_\W}(\A,\B;\C)$ and the pointwise definition of composition and reindexing in $[\B,\C]$ that $\widetilde{\mathsf{Ev}}$ is fully faithful.
\end{proof}

Using \cite[\S 1.10]{Ke:Ba} we obtain

\begin{cor}\label{thm:2func_vgr_func_cat}
There is a unique 2-functor
$$\tensor[_\V]{[-,?]}{^\W}:\textnormal{GCAT}_\W^\op \times \GCAT{\V}{\W} \to \LGCAT{\V}$$
that is given on objects by $(\B,\C) \mapsto \tensor[_\V]{[\B,\C]}{^\W}$ and makes the isomorphisms in \ref{thm:func_cat_rep_bifun} 2-natural in $\A \in \LGCAT{\V}$, $\B \in \RGCAT{\W}$, and $\C \in \GCAT{\V}{\W}$. 
\end{cor}

In view of \ref{thm:vgr_func_cat} and \ref{defn:bigraded}, we thus obtain also a 2-functor
$$\tensor[^\V]{[-,?]}{_\W} = \tensor[_{\W^\rev}]{[-,\mathrm{C}^*?]}{^{\V^\rev}}:\LGCAT{\V}^\op \times \GCAT{\V}{\W} \to \RGCAT{\W}$$
given on objects by $(\A,\C) \mapsto \tensor[^\V]{[-,?]}{_\W}$. Hence, by \ref{thm:bifun}--\ref{thm:swap_gr_bifunc}, \ref{thm:func_cat_rep_bifun}, \ref{thm:2func_vgr_func_cat} we obtain

\begin{cor}\label{thm:isos_bifun_funcat}
There are isomorphisms
\begin{align*}
\LGCAT{\V}(\A,\tensor[_\V]{[\B,\C]}{^\W}) \;\cong\; \tensor[_\V]{\textnormal{\textsf{GBif}}}{_\W}&(\A,\B;\C) \;\cong\;  \RGCAT{\W}(\B,\tensor[^\V]{[\A,\C]}{_\W})\\
& \rotatebox[origin=c]{-90}{$\cong$}\\
\GCAT{\V}{\W}&(\A \boxtimes \B,\C)
\end{align*}
2-natural in $\A \in \LGCAT{\V}$, $\B \in \RGCAT{\W}$, $\C \in \GCAT{\V}{\W}$.
\end{cor}

\begin{cor}
Every $\V$-$\W$-bigraded category $\C$ determines a 2-adjunction
$$\tensor[_\V]{[-,\C]}{^\W} \dashv \tensor[^\V]{[-,\C]}{_\W}:\LGCAT{\V}^\op \to \RGCAT{\W}\;.$$
Every left $\V$-graded category $\A$ determines a 2-adjunction
$$\A \boxtimes (-) \dashv \tensor[^\V]{[\A,-]}{_\W}:\GCAT{\V}{\W} \to \RGCAT{\W}\;.$$
Every right $\W$-graded category $\B$ determines a 2-adjunction
$$(-) \boxtimes \B \dashv \tensor[_\V]{[\B,-]}{^\W}:\GCAT{\V}{\W} \to \LGCAT{\V}\;.$$
\end{cor}

\begin{para}\label{para:func_cat_rep_gbifun_for_product-graded_cat}
Given a $\V$-graded category $\A$, a $\W$-graded category $\B$, and a $(\V \times \W)$-graded category $\C$ (all three of which are left graded), we can form the category $\textsf{GBif}(\A,\B;\C) = \tensor[_\V]{\textnormal{\textsf{GBif}}}{_{\W^\rev}}(\A,\B;\C)$ of graded bifunctors $F:\A,\B \to \C$ (in the sense of \ref{rem:bifun_vtimesw_gr}) as well as the $(\V \times \W)$-graded category $\A \boxtimes \B$. By \ref{thm:isos_bifun_funcat}, we have isomorphisms
\begin{align*}
\LGCAT{\V}(\A,[\B,\C]_\ell) \;\cong\; \textsf{GBif}&(\A,\B;\C) \;\cong\;  \LGCAT{\W}(\B,[\A,\C]_r)\\
& \rotatebox[origin=c]{-90}{$\cong$}\\
\LGCAT{\V\times \W}&(\A \boxtimes \B,\C)
\end{align*}
2-natural in $\A \in \LGCAT{\V}$, $\B \in \LGCAT{\W}$, $\C \in \LGCAT{\V \times \W}$, with the notation of \ref{thm:func_cat_for_product_graded_categories}. See \ref{para:enr_func_cat_gl}--\ref{thm:venr_func_cat_gr_func_cat} for rationale for the placement of the subscripts ``$\ell$'' and ``$r$''.
\end{para}

\section{Background II: Enriched modules, presheaves, and Yoneda}\label{sec:enr_mod_psh_cocompl}

In this section we fix a huge biclosed base $\sfV$ in the sense of \ref{para:copower_cocompl}, and we recall some aspects of $\sfV$-enriched category theory that we use in the next section (and which are more often formulated with the category of small sets, $\Set$, in place of $\SET$).

We first review \textit{$\sfV$-modules} between $\sfV$-categories, as studied in \cite{Str:EnrCatsCoh,BCSW:VarEnr} in the setting of enrichment in a bicategory, and initially we consider right $\sfV$-categories (because the latter two papers do). Given huge right $\sfV$-categories $\A$ and $\B$, a \textbf{$\sfV$-module} or \textbf{$\sfV$-profunctor} $M:\A \modto \B$ is given by (1) objects $M(B,A)$ of $\sfV$ $(B \in \ob\B, A \in \ob\A)$; (2) morphisms $\lambda^{B'}_{BA}:\B(B,B') \otimes M(B',A) \to M(B,A)$ in $\sfV$ $(B,B' \in \ob\B, A \in \ob\A)$; (3) morphisms $\rho^{A'}_{BA}:M(B,A') \otimes \A(A',A) \to M(B,A)$ in $\sfV$ $(B \in \ob \B, A,A' \in \ob\A)$, such that the morphisms in (2) and (3) satisfy axioms that assert that $\lambda$ is associative and unital, $\rho$ is associative and unital, and $\lambda$ commutes with $\rho$ in the sense that for all $B,B' \in \ob\B$ and $A,A' \in \ob\A$ the following diagram (in which we omit the associator) commutes:
\begin{equation}\label{eq:vmod_actions_commute}
\xymatrix{
\B(B,B') \otimes M(B',A') \otimes \A(A',A) \ar[d]_{\lambda^{B'}_{BA'} \otimes 1} \ar[rr]^(.55){1 \otimes \rho^{A'}_{B'A}} && \B(B,B') \otimes M(B',A)  \ar[d]^{\lambda^{B'}_{BA}}\\
M(B,A') \otimes \A(A',A) \ar[rr]^{\rho^{A'}_{BA}} && M(B,A)
}
\end{equation}
We call $\lambda$ a \textit{(contravariant left) action} of $\B$ and $\rho$ a \textit{(covariant right) action} of $\A$.

Given huge right $\sfV$-categories $\A$ and $\B$, we write $\textsf{MOD}_\sfV(\A,\B)$ for the huge category in which an object is a $\sfV$-module $M:\A \modto \B$ and a morphism is a \textbf{transformation of $\sfV$-modules} $\phi:M \Rightarrow M':\A \modto \B$, i.e., a family of morphisms $\phi_{BA}:M(B,A) \to M'(B,A)$ in $\sfV$ $(B \in \ob\B, A \in \ob\A)$ that commute with the actions carried by $M$ and $M'$. Following \cite[\S 3]{Str:EnrCatsCoh}, for $\sfV$-modules $M:\A \modto \B$ and $N:\B \modto \C$ the \textit{composite} $N \otimes M:\A \modto \C$, if it exists, is the $\sfV$-module defined by taking $(N \otimes M)(C,A)$ $(C \in \ob\C, A \in \ob\A)$ to be the colimit in $\sfV$ of the (generally huge) diagram consisting of the spans
$$N(C,B) \otimes M(B,A) \xleftarrow{\rho \otimes 1} N(C,B')\otimes\B(B',B)\otimes M(B,A) \xrightarrow{1 \otimes \lambda} N(C,B') \otimes M(B',A)$$
with $B,B' \in \ob\B$. The actions $\lambda,\rho$ on $N \otimes M$ are induced by $\lambda$ for $N$ and $\rho$ for $M$. The composite $N \otimes M$ always exists if $\B$ is $\SET$-small. Given a $\sfV$-module $P:\A \modto \C$, a \textbf{bitransformation} $\beta:N,M \Rightarrow P$ (cf.~\cite[6.1]{GaShu}) is a family of morphisms $\beta_{CA}^B:N(C,B) \otimes M(B,A) \to P(C,A)$ $(C \in \ob\C, B \in \ob\B, A \in \ob\A)$ that commute with the left $\C$-action and the right $\A$-action and provide a cocone on the above diagram for each fixed pair $C,A$. If the composite $N \otimes M$ exists, then a transformation $\beta:N \otimes M \Rightarrow P$ is equivalently given by a bitransformation $\beta:N,M \Rightarrow P$.

Given right $\sfV$-functors $F:\A' \to \A$ and $G:\B' \to \B$ between huge right $\sfV$-categories and a $\sfV$-module $M:\A \modto \B$, the objects $M(GB,FA)$ $(B \in \ob\B', A \in \ob\A')$ underlie a $\sfV$-module $M(G,F):\A' \modto \B'$ with the evident actions, while $\sfV$-natural transformations $\phi:F \Rightarrow F'$, $\psi:G' \Rightarrow G$ give rise to a transformation $M(\psi,\phi):M(G,F) \Rightarrow M(G',F')$. 
Thus we obtain a 2-functor
\begin{equation}\label{eq:mod_prime_2func}\textsf{MOD}_\sfV:(\CAT_\sfV')^\op \times (\CAT_\sfV')^{\coop} \to \CAT'\end{equation}
that is given on objects by $(\A,\B) \mapsto \textsf{MOD}_\sfV(\A,\B)$ and on homs by the formula $(\textsf{MOD}_\sfV(F,G))M = M(G,F)$.

There is a bicategory $\RMOD{\sfV}$ whose objects, 1-cells, and 2-cells are ($\SET$-small) right $\sfV$-categories, $\sfV$-modules, and transformations, respectively. Given any huge right $\sfV$-category $\A$, we write simply $\A:\A \modto \A$ for the $\sfV$-module consisting of the hom-objects $\A(A',A)$ $(A',A \in \ob\A)$ with the actions given by composition, noting that if $\A$ is $\SET$-small then $\A:\A \modto \A$ is the identity 1-cell on $\A$ in $\RMOD{\sfV}$. The unitors $\B \otimes M \overset{\sim}{\Rightarrow} M,\;M \otimes \A \overset{\sim}{\Rightarrow} M:\A \modto \B$ in $\RMOD{\sfV}$ are induced by the actions $\lambda,\rho$.

The \textit{unit $\sfV$-category} $\mathbb{I}$ is the one-object right (or left) $\sfV$-category determined by the monoid $(I,\ell_I = r_I, 1_I)$ in $\V$, recalling that $\mathbb{I}^\circ = \mathbb{I}$ (\ref{para:lr_vcats}). Given a huge right $\sfV$-category $\C$, a $\sfV$-module $M:\C \modto \mathbb{I}$ is equivalently a right $\sfV$-functor $M:\C \to \sfV$, where we regard $\sfV$ as a huge right $\sfV$-category. We have isomorphisms $\textsf{MOD}_\sfV(\C,\mathbb{I}) \cong \textnormal{CAT}_\sfV'(\C,\sfV)$ 2-natural in $\C \in \textnormal{CAT}_\sfV'$. As observed in \cite{GP:GabrielUlmer}, a $\sfV$-module $M:\mathbb{I} \modto \C$ is instead a \textit{left} $\sfV$-functor $M:\C^\circ \to \sfV$, where we regard $\sfV$ as a huge left $\sfV$-category. We have isomorphisms $\textsf{MOD}_\sfV(\mathbb{I},\C) \cong \LCAT{\sfV}'(\C^\circ,\sfV)$ 2-natural in $\C \in \textnormal{CAT}_\sfV'$.  Objects $A$ of a right $\sfV$-category $\A$ may be identified with right $\sfV$-functors $A:\mathbb{I} \to \A$. Given a $\sfV$-module $M:\A \modto \B$ and objects $A \in \ob\A$, $B \in \ob\B$, we obtain $\sfV$-modules $M(B,1):\A \modto \mathbb{I}$ and $M(1,A):\mathbb{I} \modto \B$, i.e.~a right $\sfV$-functor $M(B,-):\A \to \sfV$ and a left $\sfV$-functor $M(-,A):\B^\circ \to \sfV$. Moreover, a $\sfV$-module $M:\A \modto \B$ is equivalently given by $\sfV$-modules $M(-,A):\mathbb{I} \modto \B$ $(A \in \ob\A)$ and $M(B,-):\A \modto \mathbb{I}$ $(B \in \ob\B)$ that agree on objects and whose left and right actions commute in the sense of \eqref{eq:vmod_actions_commute}. We identify $\RMOD{\sfV}(\mathbb{I},\mathbb{I})$ with $\sfV$, so that $M(B,A):\mathbb{I} \modto \mathbb{I}$ is the object $M(B,A)$ of $\sfV$.

Given huge right $\sfV$-categories $\A$, $\B$, $\C$ and $\sfV$-modules $L:\A \modto \C$ and $N:\B \modto \C$, a \textbf{right lifting} of $L$ along $N$ (cf.~\cite{Str:EnrCatsCoh,GaShu}) is a $\sfV$-module $L^N:\A \modto \B$ equipped with a bitransformation $\varepsilon:N,L^N \Rightarrow L$ that is universal in the sense that for every $\sfV$-module $M:\A \modto \B$ and every bitransformation $\beta:N, M \Rightarrow L$ there is a unique transformation $\beta':M \Rightarrow L^N$ such that $\beta^B_{CA} = \varepsilon^B_{CA} \cdot (1 \otimes \beta'_{BA}):N(C,B) \otimes M(B,A) \to L(C,A)$ for all $C \in \ob\C$, $B \in \ob\B$, $A \in \ob\A$. In the special case where $L,N:\mathbb{I} \modto \C$, the composite $N \otimes M:\mathbb{I} \modto \C$ exists for every $M:\mathbb{I} \modto \mathbb{I}$, i.e.~every $M \in \ob\V$, and is given by $(N \otimes M)C = NC \otimes M$ $(C \in \ob\C)$ with the evident action of $\C$, and a right lifting of $L$ along $N$ is precisely an object $L^N$ of $\sfV = \RMOD{\V}(\mathbb{I},\mathbb{I})$ with a transformation $\varepsilon:N \otimes L^N \Rightarrow L$ with the following property: The transposes $L^N \to LC^{NC}$ of the morphisms $\varepsilon_C:NC \otimes L^N \to LC$ $(C \in \ob\C)$ exhibit $L^N$ as a limit of the (generally huge) diagram
$$LC^{NC} \rightarrow LD^{\C(D,C) \otimes NC} \leftarrow LD^{ND}\;\;\;\;(C,D \in \ob\C)$$
in $\sfV$ induced by the left actions carried by $L$ and $N$, respectively, by way of the biclosed structure on $\sfV$. Moreover, for arbitrary $L:\A \modto \C$ and $N:\B \modto \C$, if for all $B \in \ob\B$, $A \in \ob\A$ a right lifting $L(1,A)^{N(1,B)}$ of $L(1,A):\mathbb{I} \to \C$ along $N(1,B):\mathbb{I} \to \C$ exists, then a right lifting $L^N:\A \modto \B$ exists and is formed pointwise in the sense that the universal bitransformation $\varepsilon:N,L^N \Rightarrow L$ consists of morphisms $\varepsilon_{CA}^B:N(C,B) \otimes L^N(B,A) \to L(C,A)$ $(C \in \ob\C, B \in \ob\B, A \in \ob\A)$ that exhibit $L^N(B,A)$ as a right lifting $L(1,A)^{N(1,B)}$. If $\C$ is $\SET$-small then a right lifting $L^N$ exists and is constructed pointwise, and if $\A,\B,\C$ are all $\SET$-small then $L^N$ is a right lifting of $L$ along $N$ in the bicategory $\RMOD{\sfV}$ (in the sense of \cite[\S 2]{Str:EnrCatsCoh}).

Given huge \textit{left} $\sfV$-categories $\A$ and $\B$, we write $\tensor[_\sfV]{{\textsf{MOD}}}{}(\A,\B) := \textsf{MOD}_{\sfV^\rev}(\A,\B)$, and we declare that a \textbf{$\sfV$-module} $M:\A \modto \B$ is by definition an object of the latter category. We define $\LMOD{\sfV} := \textnormal{MOD}_{\sfV^\rev}$.

\begin{para}[\textbf{Street's presheaf $\sfV$-category and Yoneda lemma}]\label{para:street_psh}
Let $\C$ be a ($\SET$-small) right $\sfV$-category.  Then there is a huge right $\sfV$-category $\cP\C$ whose objects are $\sfV$-modules $F:\mathbb{I} \modto \C$, equivalently, left $\sfV$-functors $F:\C^\circ \to \sfV$; indeed, this was shown in \cite{Str:EnrCatsCoh}, in the more general setting of categories enriched in a bicategory, and is discussed further in \cite{BCSW:VarEnr,GP:GabrielUlmer}.  As noted in \cite{BCSW:VarEnr}, the hom-object of $\cP \C$ associated to a pair of objects $F,G$ of $\cP \C$ is the right lifting $(\cP\C)(F,G) = G^F:\mathbb{I} \modto \mathbb{I}$ of $G$ along $F$.  Hence $(\cP\C)(F,G)$ is an object of $\sfV$ for which there is a bijection, natural in $X \in \sfV$, between morphisms $\phi:X \to (\cP\C)(F,G)$ and families of morphisms $\phi_A:FA \otimes X \to GA$ in $\sfV$ $(A \in \ob\C)$ that commute with the left $\C$-actions carried by $F$ and $G$.

The huge right $\sfV$-category $\cP\C$ is characterized up to isomorphism by the following universal property: By \cite[\S 3]{Str:EnrCatsCoh}, $\cP \C$ represents the 2-functor $\textsf{MOD}_\sfV(-,\C):(\CAT_\sfV')^\op \to \CAT'$, so that we have isomorphisms
\begin{equation}\label{eq:spsh}\CAT_\sfV'(\A,\cP\C) \;\;\cong\;\; \textsf{MOD}_\sfV(\A,\C)\end{equation}
2-natural in $\A \in \textnormal{CAT}_\sfV'$, and the counit of this representation is the $\sfV$-module $\mathsf{E}:\cP \C \modto \C$ given by $\mathsf{E}(A,F) = FA$ $(A \in \ob\C$, $F \in \ob\cP\C)$ with the evident actions. Hence each right $\sfV$-functor $H:\A \to \cP \C$ corresponds under this isomorphism to a $\sfV$-module $\bar{H} := \mathsf{E}(1,H):\A \modto \C$. Given right $\sfV$-functors $H:\A \to \cP \C$ and $K:\B \to \cP\C$, the $\sfV$-module $(\cP \C)(K,H):\A \modto \B$ is a right lifting of $\bar{H}:\A \modto \C$ along $\bar{K}:\B \modto \C$, i.e.~$(\cP \C)(K,H) \cong \bar{H}^{\bar{K}}$, by the pointwise nature of right liftings.

Street \cite{Str:EnrCatsCoh} established the following version of the Yoneda lemma. Given a $\sfV$-module $F:\mathbb{I} \modto \C$ the left unitor $\lambda_F:\C \otimes F \overset{\sim}{\Rightarrow} F$ in $\RMOD{\sfV}$ exhibits $F$ as a right lifting $F^\C$ of $F$ along $\C:\C \modto \C$, as its universal property is verified trivially. Hence, by the pointwise nature of right liftings, for each object $A$ of $\C$ the action morphisms $\lambda^{A}_{A'F} := \lambda^A_{A'}:\C(A',A) \otimes FA \rightarrow FA'$ $(A' \in\ob\C)$ constitute a transformation $\lambda^A_F:\C(1,A) \otimes FA \Rightarrow F$ that exhibits $FA$ as a right lifting $F^{\C(1,A)}$, so the transpose of $\lambda^A_F$ is an isomorphism
$$\lambda'_{F,A}\;:\;FA \xrightarrow{\sim} F^{\C(1,A)} = (\cP \C)(\C(1,A),F)$$
in $\sfV$.  Under the isomorphism \eqref{eq:spsh}, $\C:\C \modto \C$ corresponds to a right $\sfV$-functor $\y:\C \to \cP\C$, uniquely characterized by the equation $\bar{\y} = \C:\C \modto \C$ and given on objects by $\y A = \C(1,A)$. The right $\sfV$-functor $\y:\C \to \cP\C$ is fully faithful, because its structural morphisms $\y_{AB}:\C(A,B) \to (\cP\C)(\C(1,A),\C(1,B)) = \C(1,B)^{\C(1,A)}$ are precisely the isomorphisms $\lambda'_{\C(1,B),A}$.

A strengthening of the above Yoneda lemma is obtained as follows. While $\y$ corresponds under the isomorphism \eqref{eq:spsh} to $\bar{\y} = \C:\C \modto \C$, the identity $1_{\cP\C}$ corresponds to the counit $\bar{1}_{\cP\C} = \mathsf{E}:\cP\C \modto \C$ of the representation. Hence, by an observation above we find that $(\cP\C)(\y,1_{\cP\C}) \cong \bar{1}_{\cP\C}^{\bar{\y}} = \mathsf{E}^\C$, i.e.~$(\cP\C)(\y,1_{\cP\C}):\cP\C \modto \C$ is a right lifting of $\mathsf{E}:\cP\C \modto \C$ along $\C:\C \modto \C$. But the left action $\lambda:\C,\mathsf{E} \Rightarrow \mathsf{E}$ exhibits $\mathsf{E}$ itself as a right lifting $\mathsf{E}^\C$, because for all $A \in \ob\C$ and $F \in \ob\cP\C$ the morphisms $\lambda^A_{A'F}:\C(A',A) \otimes FA \rightarrow F$ $(A' \in\ob\C)$ exhibit $FA$ as $F^{\C(1,A)}$ as noted above. Thus we obtain an isomorphism of $\sfV$-modules
\begin{equation}\label{eq:yoneda}\lambda'\;:\;\mathsf{E} \;\overset{\sim}{\Rightarrow}\; (\cP\C)(\y,1_{\cP\C})\;:\;\cP\C \modto \C\end{equation}
with the above components $\lambda'_{F,A}$.
\end{para}

\begin{para}\label{para:psh_yoneda_left_vcats}
Applying \ref{para:street_psh} to $\sfV^\rev$, we associate to each ($\SET$-small) left $\sfV$-category $\C$ a huge left $\sfV$-category $\cP\C$ whose objects are $\sfV$-modules $F:\mathbb{I} \modto \C$, equivalently, right $\sfV$-functors $F:\C^\circ \to \sfV$, where we regard $\sfV$ as a huge right $\sfV$-category; each hom-object $(\cP\C)(F,G)$ is characterized by the existence of a bijection, natural in $X \in \sfV$, between morphisms $\phi:X \to (\cP\C)(F,G)$ and families of morphisms $\phi_A:X \otimes FA \to GA$ in $\sfV$ $(A \in \ob\C)$ that commute with the $\C$-actions carried by $F,G:\mathbb{I} \modto \C$. By the above, we have a fully faithful left $\V$-functor $\y:\C \to \cP\C$ and a Yoneda lemma for left $\sfV$-categories.
\end{para}

\section{Example: \texorpdfstring{$\V$}{V}-graded modules, presheaves, and Yoneda}\label{sec:vgr_mod_psh_yon}

For any monoidal category $\V$, we refer to $\hat{\V}$-modules in the sense of \S \ref{sec:enr_mod_psh_cocompl} as \textbf{$\V$-graded modules} or \textbf{$\V$-graded profunctors}. Given huge right $\V$-graded categories (resp.~huge left $\V$-graded categories) $\A$ and $\B$, we write $\textsf{GMOD}_\V(\A,\B) := \textsf{MOD}_{\hat{\V}}(\A,\B)$ (resp.~$\tensor[_\V]{{\textsf{GMOD}}}{}(\A,\B) := \tensor[_{\hat{\V}}]{{\textsf{MOD}}}{}(\A,\B)$). By \eqref{eq:mod_prime_2func} we have a 2-functor $$\textsf{GMOD}_\V := \textsf{MOD}_{\hat{\V}}:(\textnormal{GCAT}'_\V)^\op \times (\textnormal{GCAT}'_\V)^\coop \to \CAT'\;,$$ and similarly a 2-functor $\tensor[_\V]{{\textsf{GMOD}}}{} := \tensor[_{\hat{\V}}]{{\textsf{MOD}}}{}$. We have a bicategory $\RGMOD{\V} := \RMOD{\hat{\V}}$ (resp.~$\LGMOD{\V} := \LMOD{\hat{\V}}$) whose 1-cells are $\V$-graded modules between ($\SET$-small) right $\V$-graded categories (resp.~left $\V$-graded categories). 
 
With this terminology, $\V$-graded modules are examples of graded bifunctors, by the following result, in which we regard $\hat{\V}$ as a huge $\V$-$\V$-bigraded category as in \ref{exa:bigr_str_hatv}.

\begin{thm}\label{thm:mod_as_bifun}
There are isomorphisms
\begin{equation}\label{eq:mod_as_bifun}\textnormal{\textsf{GMOD}}_\V(\A,\B) \cong \tensor[_\V]{\textnormal{\textsf{GBIF}}}{_\V}(\B^\circ,\A;\hat{\V})\end{equation}
2-natural in $\A,\B \in \textnormal{GCAT}_\V'$. Also, there are isomorphisms
$$\tensor[_\V]{{\textnormal{\textsf{GMOD}}}}{}(\A,\B) \cong \tensor[_\V]{\textnormal{\textsf{GBIF}}}{_\V}(\A,\B^\circ;\hat{\V})$$
2-natural in $\A,\B \in \LGCAT{\V}'$.
\end{thm}
\begin{proof}
Given huge right $\V$-graded categories $\A$ and $\B$, $\tensor[_\V]{\textnormal{\textsf{GSES}}}{_\V}(\B^\circ,\A;\hat{\V})$ is the pullback of the cospan
$$\LGCAT{\V}'(\B^\circ,\hat{\V})^{\ob\A} \to \hat{\V}^{\ob\B \times \ob\A} \leftarrow \textnormal{GCAT}_\V'(\A,\hat{\V})^{\ob\B}$$
in $\CAT'$. But by \S \ref{sec:enr_mod_psh_cocompl}, we have isomorphisms
\begin{equation}\label{eq:psh_isos}\textsf{GMOD}_\V(\mathbb{I},\B) \cong \LGCAT{\V}'(\B^\circ,\hat{\V}),\;\;\textsf{GMOD}_\V(\A,\mathbb{I}) \cong \textnormal{GCAT}_\V'(\A,\hat{\V})\end{equation}
2-natural in $\A,\B \in \textnormal{GCAT}_\V'$. Thus $\tensor[_\V]{\textnormal{\textsf{GSES}}}{_\V}(\B^\circ,\A;\hat{\V})$ may be identified with the pullback of the cospan
$$\textsf{GMOD}_\V(\mathbb{I},\B)^{\ob\A} \to \hat{\V}^{\ob\B \times \ob\A} \leftarrow \textsf{GMOD}_\V(\A,\mathbb{I})^{\ob\B},$$
in $\CAT'$, so that an object $M$ of $\tensor[_\V]{\textnormal{\textsf{GSES}}}{_\V}(\B^\circ,\A;\hat{\V})$ is equivalently given by $\V$-graded modules $M(-,A):\mathbb{I} \modto \B$ $(A \in \ob\A)$ and $M(B,-):\A \modto \mathbb{I}$ $(B \in \ob\B)$ that agree on objects. Thus, by \S \ref{sec:enr_mod_psh_cocompl}, $\textnormal{\textsf{GMOD}}_\V(\A,\B)$ may be identified with the full subcategory of $\tensor[_\V]{\textnormal{\textsf{GSES}}}{_\V}(\B^\circ,\A;\hat{\V})$ spanned by the objects $M$ whose actions
$$\lambda_{B'A}^{B}:\B(B',B) \otimes M(B,A) \to M(B',A),\;\;\;\;\rho^{A}_{BA'}:M(B,A) \otimes \A(A,A') \to M(B,A')$$
in $\hat{\V}$ $(A,A' \in \ob\A, B,B' \in \ob\B)$ commute in the sense of \eqref{eq:vmod_actions_commute}, but we now show that $M$ satisfies the latter condition \eqref{eq:vmod_actions_commute} if and only if $M$ is a graded bifunctor. To this end, note that for each $B \in \ob\B$ the 
right $\V$-graded functor $M(B,-):\A \to \hat{\V}$ sends each graded morphism $g:A \gc X' \to A'$ in $\A$ to a graded morphism $M(B,g):M(B,A) \gc X' \to M(B,A')$ in $\hat{\V}$, i.e.~a natural transformation $M(B,g):M(B,A) \Rightarrow M(B,A')(-\otimes X'):\V^\op \to \SET$. The left $\V$-graded functor $M(-,A):\B^\circ \to \hat{\V}$ sends each graded morphism $f:B' \gc X \to B$ in $\B$ to a graded morphism $M(f,A):X \gc M(B,A) \to M(B',A)$ in $\hat{\V}$, i.e.~a natural transformation $M(f,A):M(B,A) \Rightarrow M(B',A)(X \otimes -)$. With this in mind, the above actions are induced (via \ref{sec:day_conv}) by the maps
$$M^A:\B(B'\gc X;B) \times M(B,A)Y \to M(B',A)(X \otimes Y)$$
$$M_B:M(B,A)Y \times \A(A \gc X';A') \to M(B,A')(Y \otimes X')$$
with $X,Y,X' \in \ob\V$ given by $M^A(f,\mu) = M(f,A)_Y(\mu)$ and $M_B(\mu,g) = M(B,g)_Y(\mu)$. Consequently, the commutativity of \eqref{eq:vmod_actions_commute} is equivalently the commutativity of the diagrams of the following form, in which we omit the use of the associator in $\V$:
$$
\xymatrix{
\B(B' \gc X;B) \times M(B,A)Y \times \A(A \gc X';A') \ar[d]_{M^A \times 1} \ar[r]^(.55){1 \times M_B} & \B(B'\gc X;B) \times M(B,A')(Y \otimes X') \ar[d]^{M^{A'}}\\
M(B',A)(X \otimes Y) \times \A(A \gc X';A') \ar[r]^{M_{B'}} & M(B',A')(X \otimes Y \otimes X')
}
$$
This requires that for all $f:B' \gc X \to B$ in $\B$ and $g:A \gc X' \to A'$ in $\A$ the following diagram commutes for each $Y \in \ob\V$:
$$
\xymatrix{
M(B,A)Y \ar[d]_{M(f,A)_Y} \ar[rr]^(.4){M(B,g)_Y} && M(B,A')(Y\otimes X') \ar[d]^{M(f,A')_{Y\otimes X'}}\\
M(B',A)(X \otimes Y) \ar[rr]^(.45){M(B',g)_{X \otimes Y}} && M(B',A')(X \otimes Y \otimes X')
}
$$
But by \ref{exa:bigr_sq_hatv} this is equivalently the requirement that the graded morphisms
$M(f,A)$, $M(f,A')$, $M(B,g)$, $M(B',g)$ form a bigraded square in $\hat{\V}$.

Thus we obtain isomorphisms of the form \eqref{eq:mod_as_bifun}, and their 2-naturality in each variable separately is entailed by the 2-naturality of the isomorphisms in \eqref{eq:psh_isos}. Thus the first claim is proved, and the second follows, since
\begin{align*}
\tensor[_\V]{{\textnormal{\textsf{GMOD}}}}{}(\A,\B) & = \textnormal{\textsf{GMOD}}_{\V^\rev}(\A,\B) \cong \tensor[_{\V^\rev}]{\textnormal{\textsf{GBIF}}}{_{\V^\rev}}(\B^\circ,\A;\widehat{\V^\rev})\\
& \cong \tensor[_{\V^\rev}]{\textnormal{\textsf{GBIF}}}{_{\V^\rev}}(\B^\circ,\A;\mathrm{C}^*\hat{\V}) \cong \tensor[_\V]{\textnormal{\textsf{GBIF}}}{_\V}(\A,\B^\circ;\hat{\V})
\end{align*}
2-na\-tur\-al\-ly in $\A,\B \in \LGCAT{\V}' = \textnormal{GCAT}'_{\V^\rev}$ by \ref{thm:swap_gr_bifunc}, since $\mathrm{C}^*\hat{\V} \cong \widehat{\V^\rev}$. 
\end{proof}

Applying Theorem \ref{thm:bifun} relative to $\SET'$, this entails the following:

\begin{cor}
There are isomorphisms $\textnormal{\textsf{GMOD}}_\V(\A,\B) \cong \tensor[_\V]{{\textnormal{GCAT}'}}{_\V}(\B^\circ \boxtimes \A,\hat{\V})$ 2-natural in $\A,\B \in \textnormal{GCAT}_\V'$. Similarly, there are isomorphisms $\tensor[_\V]{{\textnormal{\textsf{GMOD}}}}{}(\A,\B) \cong \tensor[_\V]{{\textnormal{GCAT}'}}{_\V}(\A \boxtimes \B^\circ,\hat{\V})$ 2-natural in $\A,\B \in \LGCAT{\V}'$.
\end{cor}

\begin{exa}[\textbf{$\V$-graded hom-functors}]\label{exa:hom_bifunc}
Let $\B$ be a right $\V$-graded category. Under the isomorphism of Theorem \ref{thm:mod_as_bifun}, the identity $\V$-graded module $\B:\B \modto \B$ corresponds to a graded bifunctor $\B(-,?):\B^\circ,\B \to \hat{\V}$ consisting of right $\V$-graded functors $\B(A,-):\B \to \hat{\V}$ $(A \in \ob\B)$ and left $\V$-graded functors $\B(-,B):\B^\circ \to \hat{\V}$ $(B \in \ob\B)$. Explicitly, $\B(A,-)$ sends each graded morphism $g:B \gc X' \to B'$ in $\B$ to the graded morphism $\B(A,g):\B(A,B) \gc X' \to \B(A,B')$ in $\hat{\V}$ given by the family of maps $\B(A \gc Y;g):\B(A \gc Y;B) \to \B(A \gc Y\otimes X';B')$ $(Y \in \ob\V)$ sending each graded morphism $f:A \gc Y \to B$ to the composite $g \circ f:A \gc Y \otimes X' \to B'$. On the other hand, $\B(-,B)$ sends each graded morphism $f:A \gc X \to A'$ in $\B$ to the graded morphism $\B(f,B):X \gc \B(A',B) \to \B(A,B)$ in $\hat{\V}$ given by the maps $\B(f \gc Y;B):\B(A' \gc Y;B) \to \B(A \gc X \otimes Y;B)$ $(Y \in \ob\V)$ sending each graded morphism $g:A' \gc Y \to B$ to the composite $g \circ f:A \gc X \otimes Y \to B$.
\end{exa}

\begin{exa}[\textbf{The $\V$-graded Yoneda embedding and Yoneda lemma}]
Let $\B$ be a right $\V$-graded category. Then $\B^\circ$ is a left $\V$-graded category, so by regarding $\hat{\V}$ as a huge $\V$-$\V$-bigraded category as in \ref{exa:bigr_str_hatv}, we can apply \ref{para:fcat_huge} to obtain a huge right $\V$-graded category $[\B^\circ,\hat{\V}]$. By invoking Theorem \ref{thm:mod_as_bifun} and applying Corollary \ref{thm:isos_bifun_funcat} relative to $\SET'$, we obtain isomorphisms 
\begin{equation}\label{eq:gmod_psh}\textnormal{\textsf{GMOD}}_\V(\A,\B) \cong \tensor[_\V]{\textnormal{\textsf{GBIF}}}{_\V}(\B^\circ,\A;\hat{\V}) \cong \textnormal{GCAT}_\V'(\A,[\B^\circ,\hat{\V}])\end{equation}
2-natural in $\A \in \textnormal{GCAT}_\V'$. Therefore, in view of \ref{para:street_psh}, we deduce that $[\B^\circ,\hat{\V}]$ is isomorphic to Street's $\hat{\V}$-enriched presheaf category $\cP \B$, by way of an isomorphism $[\B^\circ,\hat{\V}] \cong \cP\B$ given on objects by the usual bijection between left $\V$-graded functors $F:\B^\circ \to \hat{\V}$ and $\V$-graded modules $F:\mathbb{I} \modto \B$. Under this isomorphism, Street's Yoneda embedding $\y:\B \to \cP\B$ is identified with a fully faithful right $\V$-graded functor $\y:\B \to [\B^\circ,\hat{\V}]$ that we call the (right) \textbf{$\V$-graded Yoneda embedding}.  Explicitly, $\y$ is the transpose of the graded bifunctor $\B(-,?):\B^\circ,\B \to \hat{\V}$ of \ref{exa:hom_bifunc}, so $\y$ is given on objects by $B \mapsto \B(-,B):\B^\circ \to \hat{\V}$ with the notation of \ref{exa:hom_bifunc} and sends each graded morphism $g:B \gc X' \to B'$ in $\B$ to a graded transformation $\B(-,g):\B(-,B) \gc X' \Rightarrow \B(-,B')$ whose components are the graded morphisms $\B(A,g):\B(A,B) \gc X' \to \B(A,B')$ in $\hat{\V}$ $(A \in \ob\B)$ described in \ref{exa:hom_bifunc}. By the $\hat{\V}$-enriched Yoneda lemma \eqref{eq:yoneda} together with Theorem \ref{thm:mod_as_bifun}, we obtain an isomorphism of graded bifunctors
\begin{equation}\label{eq:graded_yoneda_lemma}[\B^\circ,\hat{\V}](\y-,?) \cong \mathsf{Ev}\;:\;\B^\circ,[\B^\circ,\hat{\V}] \to \hat{\V}\end{equation}
where $\mathsf{Ev}$ is the evaluation bifunctor, i.e.~the counit of the representation in \eqref{eq:gmod_psh}. We express this isomorphism \eqref{eq:graded_yoneda_lemma} by saying also that
$$[\B^\circ,\hat{\V}](\B(-,B),F) \cong FB$$
\textit{$\V$-graded naturally in $B \in \B$ and $F \in [\B^\circ,\hat{\V}]$}, and we call this result the (right) \textbf{$\V$-graded Yoneda lemma}.
\end{exa}

\section{Example: Categories enriched in and graded by duoidal categories}\label{sec:duoidal}

In this section, we treat the special case in which $\V$ underlies a normal duoidal category $\dV$ (\ref{exa:duoidal}), and in this case we compare our graded functor categories and bifunctors with the enriched functor categories and bifunctors defined by Garner and L\'opez Franco \cite{GaLf} relative to normal duoidal categories.

As discussed in \ref{exa:duoidal}, a duoidal category $\dV = (\V,\dpr,J,\alpha,\lambda,\rho)$ is by definition a pseudomonoid in $\MCAT_\oplax$. Equivalently, a duoidal category $\dV$ is given by a monoidal category $\V = (\V,\otimes,I,a,\ell,r)$ whose underlying ordinary category is equipped with a second monoidal structure $\dpr,J,\alpha,\lambda,\rho$ together with a family of morphisms
$$\xi = \xi_{(Y,Y')(X,X')}:(Y \otimes X) \dpr (Y' \otimes X') \to (Y \dpr Y') \otimes (X \dpr X')$$
natural in $X,X',Y,Y' \in \V$, and morphisms $\mu:I \dpr I \to I$, $\nu:J \to I$, $\gamma:J \to J \otimes J$, satisfying certain axioms, given explicitly in \cite[6.1.1]{AgMa}; we call $\xi$ the \textit{interchanger}. Thus we also write $\dV = (\V,\dpr,J,\alpha,\lambda,\rho,\xi,\mu,\nu,\gamma)$, while we occasionally indicate both monoidal products explicitly by writing $\dV = (\V,\otimes,\dpr)$. A duoidal category $\dV$ is \textit{normal} if the opmonoidal functors $\dpr:\V \times \V \to \V$ and $J:1 \to \V$ are both normal (\ref{exa:comon_ch_base}), equivalently, if $\nu:J \to I$ is invertible \cite[\S 2.2]{GaLf}.  When $\dV$ is normal duoidal, we assume without loss of generality that $J = I$ and $\nu = 1_I$, we write simply  $\dV = (\V,\dpr)$ for brevity, and we employ the following notation from \cite[(2.2)]{GaLf}: Given objects $X,X' \in \ob\V$ we write $\sigma:X \dpr X' \to X \otimes X'$ and $\tau:X \dpr X' \to X' \otimes X$ to denote the composite morphisms
\begin{equation}\label{eq:sigma_tau}
\begin{array}{l}
X \dpr X' \xrightarrow{r^{-1} \,\dpr\, \ell^{-1}} (X \otimes I) \dpr (I \otimes X') \xrightarrow{\xi} (X \dpr I) \otimes (I \dpr X') \xrightarrow{\rho \otimes \lambda} X \otimes X'\\
X \dpr X' \xrightarrow{\ell^{-1} \,\dpr\, r^{-1}} (I \otimes X) \dpr (X' \otimes I) \xrightarrow{\xi} (I\dpr X') \otimes (X \dpr I) \xrightarrow{\lambda \otimes \rho} X' \otimes X.
\end{array}
\end{equation}
 
Given duoidal categories $\dV = (\V,\otimes,\dpr)$ and $\mathrm{W} = (\W,\otimes,\dpr)$, a \textit{strong duoidal functor}  $F:\dV \to \mathrm{W}$ is a functor $F:\V \to \W$ equipped with strong monoidal structures with respect to both $\otimes$ and $\dpr$, such that these monoidal structures commute with the transformations $\xi,\mu,\nu,\gamma$ carried by $\dV$ and $\mathrm{W}$, in the evident sense \cite{BooSt,GaLf}.

\begin{exa}[\textbf{Braided monoidal categories}]\label{exa:brmon}
If a monoidal category $\V$ carries a braiding $c$, then $\V$ underlies a normal duoidal category $\dV$ in which $\dpr = \otimes$ and $J = I$, with $\xi$ obtained in the evident way using the braiding and the associators \cite[Example 2]{GaLf}. In this case $\sigma$ and $\tau$ are the identity and the braiding $c:X \otimes X' \to X' \otimes X$, respectively, by \cite[\S 2.2]{GaLf}. Thus braided monoidal categories may be identified with certain normal duoidal categories, so in the sequel we implicitly regard them as such, writing them also as $\dV = (\V,c)$.
\end{exa}

\begin{para}\label{para:cstar}
For the remainder of \S \ref{sec:duoidal}, we fix a normal duoidal category $\dV = (\V, \dpr, J, \alpha, \lambda, \rho, \xi,$ $\mu, \nu, \gamma)$, and we assume as above that $J = I$ and $\nu = 1_I$. By definition, a \textbf{$\dV$-graded category} (resp.~\textbf{$\dV$-category}) is a (left) $\V$-graded category (resp.~$\V$-category) for the monoidal category $\V = (\V,\otimes,I)$. Given a $\dV$-graded category $\C$, recall from \ref{exa:duoidal} that the 2-functor $\dpr^*:\LGCAT{\V} \to \LGCAT{\V \times \V}$ sends $\C$ to a $(\V \times \V)$-graded category $\C_\dpr$. The objects of $\C_\dpr$ are the same as those of $\C$, and a graded morphism $f:(X,X') \gc A \to B$ in $\C_\dpr$ is a graded morphism $f:X \dpr X' \gc A \to B$ in $\C$. Given also a graded morphism $g:(Y,Y') \gc B \to C$ in $\C_\dpr$, we write the composite of $f$ and $g$ in $\C_\dpr$ as $g \circ_\dpr f:(Y \otimes X,Y' \otimes X') \gc A \to C$. Explicitly, $g \circ_\dpr f = \xi^*(g \circ f):(Y \otimes X)\dpr(Y' \otimes X') \gc A \to C$ is the reindexing of the composite $g \circ f:(Y \dpr Y') \otimes (X \dpr X') \gc A \to C$ in $\C$ along the interchanger $\xi:(Y \otimes X) \dpr (Y' \otimes X') \to (Y \dpr Y') \otimes (X \dpr X')$.

Given a (left) $\dV$-graded category $\C$, the $(\V \times \V)$-graded category $\C_\dpr$ has two underlying $\V$-graded categories, namely $U_\ell^*\C_\dpr$ and $U_r^*\C_\dpr$ with the notation of \ref{para:cats_gr_prod}, which we claim are both isomorphic to $\C$. Indeed, $\otimes$ and $\dpr$ share the same unit $I = J$, and the unitors $\rho$ and $\lambda$ for $\dpr$ are opmonoidal transformations (\ref{exa:duoidal}), so $\rho$ and $\lambda$ witness that the strong opmonoidal functors $U_\ell = (-,I), U_r = (I,-):\V \rightrightarrows \V \times \V$ are both pseudo-sections of $\dpr$ in $\MCAT_\oplax$, i.e.~$\rho:\dpr\,U_\ell \Rightarrow 1_\V$ and $\lambda:\dpr\,U_r \Rightarrow 1_\V$ are invertible 2-cells therein; hence, by \ref{thm:backward_ch_base_2functor}, the 2-functors $U^*_\ell,U^*_r:\LGCAT{\V \times \V} \rightrightarrows \LGCAT{\V}$ are both pseudo-retractions of $\dpr^*$. Under the resulting isomorphisms
\begin{equation}\label{eq:two_underl_vcats_iso}U^*_\ell\C_\dpr \cong \C \cong U^*_r\C_\dpr,\end{equation}
which are the identity on objects, each graded morphism $f:X \gc A \to B$ in $\C$ corresponds to graded morphisms
$$f_\ell:X \gc A \to B\;\textnormal{in $U^*_\ell\C_\dpr$}\;,\;\;\;\;\;\;\;\;f_r:X \gc A \to B\;\textnormal{in $U^*_r\C_\dpr$}$$
obtained as follows: By definition, $f_\ell$ and $f_r$ are the graded morphisms
\begin{equation}\label{eq:fell_fr}f_\ell:(X,I) \gc A \to B,\;\;\;\;\;f_r:(I,X) \gc A \to B\;\;\;\;\textnormal{in $\C_\dpr$}\end{equation}
obtained as the reindexings $\rho^*(f):X \dpr I \gc A \to B$ and $\lambda^*(f):I \dpr X \gc A \to B$ of $f$ in $\C$ along the unitors $\rho:X \dpr I \xrightarrow{\sim} X$ and $\lambda:I \dpr X \xrightarrow{\sim} X$.

We next show that bigraded squares in $\C_\dpr$ are given by the following data:
\end{para}

\begin{defn}\label{defn:duoi_vgr_sq}
Given a $\dV$-graded category $\C$, a \textbf{$\dV$-graded square} in $\C$ is a quadruple $(f,g,\phi,\phi')$ consisting of graded morphisms $f:X \gc A \to A'$, $g:X \gc B \to B'$, $\phi:X' \gc A \to B$, and $\phi':X' \gc A' \to B'$ in $\C$ such that the composites $g \circ \phi:X \otimes X' \gc A \to B'$ and $\phi' \circ f:X' \otimes X \gc A \to B'$ are related by the equation
$$\sigma^*(g \circ \phi) = \tau^*(\phi' \circ f):X \dpr X' \gc A \to B',$$
where $\sigma$, $\tau$ are the morphisms in \eqref{eq:sigma_tau}. We depict $\dV$-graded squares $(f,g,\phi,\phi')$ by diagrams \eqref{eq:bigr_sq_notn} just as for bigraded squares.

\medskip

For example, if $\dV$ is braided monoidal (so with $\dpr = \otimes$), then $\sigma$ and $\tau$ are the identity and the braiding $c:X \otimes X' \to X' \otimes X$, respectively, so the preceding equation requires precisely that
\begin{equation}\label{eq:vgr_sq_braided}g \circ \phi = c^*(\phi' \circ f):X \otimes X' \gc A \to B'.\end{equation}
\end{defn}

\begin{prop}\label{thm:bigr_sq_normal_duoidal}
Let $\C$ be a $\dV$-graded category. There is a bijective correspondence between $\dV$-graded squares in $\C$ and bigraded squares in $\C_\dpr$, under which each $\dV$-graded square $(f,g,\phi,\phi')$ in $\C$ corresponds to the bigraded square $(f_\ell,g_\ell,\phi_r,\phi'_r)$ in $\C_\dpr$.
\end{prop}
\begin{proof}
Let $f:X \gc A \to A'$, $g:X \gc B \to B'$, $\phi:X' \gc A \to B$, and $\phi':X' \gc A' \to B'$ be graded morphisms in $\C$, with corresponding graded morphisms
$$f_\ell:(X,I) \gc A \to A',\;g_\ell:(X,I) \gc B \to B',\; \phi_r:(I,X') \gc A \to B,\;\phi'_r:(I,X') \gc A' \to B'$$
in $\C_\dpr$ as in \eqref{eq:fell_fr}, which we can express as the reindexings
$$\rho^*(f):X \dpr I \gc A \to A',\;\rho^*(g):X \dpr I \gc B \to B',\;\lambda^*(\phi):I \dpr X' \gc A \to B,\;\lambda^*(\phi'):I \dpr X' \gc A' \to B'$$
in $\C$. In view of \ref{para:cstar}, the composites
$$g_\ell \circ_\dpr \phi_r:(X \otimes I,I \otimes X') \gc A \to B',\;\;\;\;\phi'_r \circ_\dpr f_\ell:(I \otimes X,X' \otimes I) \gc A \to B'$$
in $\C_\dpr$ are by definition the reindexings $\xi^*(\rho^*(g) \circ \lambda^*(\phi)):(X \otimes I) \dpr (I \otimes X') \gc A \to B'$ and $\xi^*(\lambda^*(\phi') \circ \rho^*(f)):(I \otimes X) \dpr (X' \otimes I) \gc A \to B'$
in $\C$ along the interchangers $\xi:(X \otimes I) \dpr (I \otimes X') \to (X \dpr I) \otimes (I \dpr X')$ and $\xi:(I \otimes X) \dpr (X' \otimes I) \to (I \dpr X') \otimes (X \dpr I)$, respectively. Hence, by \ref{rem:bigr_sq_prod_gr} and the definition of reindexing in $\C_\dpr$, we find that $(f_\ell,g_\ell,\phi_r,\phi'_r)$ is a bigraded square in $\C_\dpr$ if and only if $$(r^{-1} \dpr \ell^{-1})^*(\xi^*(\rho^*(g) \circ \lambda^*(\phi))) = (\ell^{-1}\dpr r^{-1})^*(\xi^*(\lambda^*(\phi') \circ \rho^*(f))):X \dpr X' \gc A \to B',$$
which simplifies (by the functoriality of reindexing and the naturality of composition, \ref{para:vgr_cat_conc}) to the equation $((\rho \otimes \lambda) \cdot \xi \cdot (r^{-1} \dpr \ell^{-1}))^*(g \circ \phi) = ((\lambda \otimes \rho) \cdot \xi \cdot (\ell^{-1} \dpr r^{-1}))^*(\phi' \circ f)$, i.e.~the equation $\sigma^*(g \circ \phi) = \tau^*(\phi' \circ f)$.
\end{proof}

Among braided monoidal categories $\dV$, symmetric monoidal categories are characterized by the property that $\dV$-graded squares can be `flipped', as follows:

\begin{prop}\label{thm:flipped_square}
The following are equivalent for a braided monoidal category $\dV$: (1) $\dV$ is symmetric monoidal; (2) for every $\dV$-graded category $\C$ and every $\dV$-graded square $(f,g,\phi,\phi')$ in $\C$, the quadruple $(\phi,\phi',f,g)$ is a $\dV$-graded square in $\C$.
\end{prop}
\begin{proof}
If (1) holds, then the characterization of $\dV$-graded squares in \eqref{eq:vgr_sq_braided} entails (2), since $c$ is its own inverse. Conversely, suppose (2). Letting $X,X' \in \ob\V$, we may regard the isomorphisms $f = r_X:X \otimes I \to X$, $g = c_{XX'}:X \otimes X' \to X' \otimes X$, $\phi = r_{X'}:X' \otimes I \to X'$, and $\phi' = 1_{X' \otimes X}:X' \otimes X \to X' \otimes X$ in $\V$ as graded morphisms $f:X \gc I \to X$, $g:X \gc X' \to X' \otimes X$, $\phi:X' \gc I \to X'$, $\phi':X' \gc X \to X' \otimes X$ in the (left) $\V$-graded category $\C := \V$. To show that the equation \eqref{eq:vgr_sq_braided} holds, first note that the composite graded morphisms $g \circ \phi:X \otimes X' \gc I \to X' \otimes X$ and $\phi' \circ f:X' \otimes X  \gc I \to X' \otimes X$ can be expressed (using coherence for monoidal categories) as the morphisms $g \circ \phi = c_{XX'} \cdot r_{X\otimes X'}:(X \otimes X') \otimes I \to X' \otimes X$ and $\phi' \circ f = r_{X' \otimes X}:(X' \otimes X) \otimes I \to X' \otimes X$ in $\V$, so that $c_{XX'}^*(\phi' \circ f) = r_{X'\otimes X} \cdot (c_{XX'} \otimes I) = c_{XX'} \cdot r_{X\otimes X'} = g \circ \phi$ by the naturality of $r$. Hence $(f,g,\phi,\phi')$ is a $\dV$-graded square, so $(\phi,\phi',f,g)$ is a $\dV$-graded square by (2). As an instance of \eqref{eq:vgr_sq_braided}, this means that $\phi' \circ f = c_{X'X}^*(g \circ \phi):X' \otimes X \gc I \to X' \otimes X$. Hence $r_{X' \otimes X} = c_{XX'} \cdot r_{X \otimes X'} \cdot (c_{X'X} \otimes I) = c_{XX'} \cdot c_{X'X} \cdot r_{X' \otimes X}:(X' \otimes X) \otimes I \to X' \otimes X$ by the naturality of $r$. From this it follows that $c_{XX'} \cdot c_{X'X} = 1_{X' \otimes X}$, and (1) is proved.
\end{proof}

For a normal duoidal category $\dV = (\V,\dpr)$, we next show that graded bifunctors valued in $\C_\dpr$ are equivalently given by the following notion:

\begin{defn}\label{para:duoidally_vgr_bifun}
Given $\dV$-graded categories $\A$, $\B$, $\C$, a \textbf{$\dV$-graded bifunctor} $F:\A,\B \to \C$ consists of $\V$-graded functors $F(-,B):\A \to \C$ $(B \in \ob\B)$ and $F(A,-):\B \to \C$ $(A \in \ob\A)$ that agree on objects and satisfy the following condition: For all $f:X \gc A \to A'$ in $\A$ and $g:X' \gc B \to B'$ in $\B$ the graded morphisms
$$F(f,B):X \gc F(A,B) \to F(A',B),\;\;F(f,B'):X \gc F(A,B') \to F(A',B'),$$
$$F(A,g):X' \gc F(A,B) \to F(A,B'),\;\;F(A',g):X' \gc F(A',B) \to F(A',B').$$
form a $\dV$-graded square $(F(f,B),F(f,B'),F(A,g),F(A',g))$ in $\C$, i.e.
$$\sigma^*(F(f,B') \circ F(A,g)) = \tau^*(F(A',g) \circ F(f,B)):X \dpr X' \gc F(A,B) \to F(A',B').$$
Through a straightforward variation on \ref{para:2fun_bifun}, $\dV$-graded bifunctors $F:\A,\B \to \C$ are the objects of a category $\textnormal{$\dV$-\textsf{GBif}}(\A,\B;\C)$.
\end{defn}

\begin{thm}\label{thm:duoidally_vgr_bifun}
Given a normal duoidal category $\dV = (\V,\dpr)$ and $\dV$-graded categories $\A$, $\B$, $\C$, there is a bijective correspondence between graded bifunctors $F:\A,\B \to \C_\dpr$ in the sense of \ref{rem:bifun_vtimesw_gr} and $\dV$-graded bifunctors $F:\A,\B \to \C$. Moreover, there is an isomorphism $\textnormal{\textsf{GBif}}(\A,\B;\C_\dpr) \cong \textnormal{$\dV$-\textsf{GBif}}(\A,\B;\C)$.
\end{thm}
\begin{proof}
Given a $\dV$-graded bifunctor $F:\A,\B \to \C$, the corresponding graded bifunctor $\A,\B \to \C_\dpr$ is obtained by composing the $\V$-graded functors $F(-,B):\A \to \C$ $(B \in \ob\B)$ and $F(A,-):\B \to \C$ $(A \in \ob\A)$ with the isomorphisms $\C \cong U^*_\ell\C_\dpr$ and $\C \cong U^*_r\C_\dpr$ of \eqref{eq:two_underl_vcats_iso}, respectively. Indeed, in view of \ref{rem:bifun_vtimesw_gr} and \ref{thm:bigr_sq_normal_duoidal}, this assignment provides the needed bijective correspondence, and the result follows straightforwardly.
\end{proof}

\begin{para}\label{para:2func_vgbif}
By \ref{thm:2fun_bifun} and \ref{thm:duoidally_vgr_bifun}, there is a 2-functor $\textnormal{$\dV$-\textsf{GBif}}:\LGCAT{\V}^\op \times \LGCAT{\V}^\op \times \LGCAT{\V} \to \CAT$ that makes the isomorphisms in \ref{thm:duoidally_vgr_bifun} 2-natural in $\A,\B,\C \in \LGCAT{\V}$.
\end{para}

Having shown that $\dV$-graded bifunctors for normal duoidal categories $\dV$ are examples of the graded bifunctors of \S \ref{sec:bifunctors}, we now construct $\dV$-graded functor categories as instances of the graded functor categories of \S \ref{sec:gr_func_cat}. In \ref{thm:duoi_venr_bifun}--\ref{thm:venr_func_cat_gr_func_cat} we compare these notions with the enriched bifunctors and functor categories of Garner and L\'opez Franco.

\begin{thm}\label{thm:func_cat_repn_duoidal}
Given a normal duoidal category $\dV = (\V,\dpr)$, there are 2-functors $[-,?]_r,[-,?]_\ell:\LGCAT{\V}^\op \times \LGCAT{\V} \rightrightarrows \LGCAT{\V}$ such that
$$\LGCAT{\V}(\A,[\B,\C]_\ell) \cong \textnormal{$\dV$-\textsf{GBif}}(\A,\B;\C) \cong \LGCAT{\V}(\B,[\A,\C]_r)$$
2-naturally in $\A,\B,\C \in \LGCAT{\V}$. Explicitly, these 2-functors send each pair of $\dV$-graded categories $\A,\C$ to $\dV$-graded categories $[\A,\C]_r$ and $[\A,\C]_\ell$ whose objects are (in both cases) $\V$-graded functors $F:\A \to \C$. A graded morphism $\phi:X' \gc F \to G$ in $[\A,\C]_r$ (resp. $[\A,\C]_\ell$) is a family of graded morphisms $\phi_A:X' \gc FA \to GA$ $(A \in \ob\A)$ in $\C$ such that for each graded morphism $f:X \gc A \to B$ in $\A$, $(Ff,Gf,\phi_A,\phi_B)$ (resp. $(\phi_A,\phi_B,Ff,Gf)$) is a $\dV$-graded square in $\C$. Furthermore, $[\A,\C]_r \cong [\A,\C_\dpr]_r$ and $[\A,\C]_\ell \cong [\A,\C_\dpr]_\ell$ 2-naturally in $\A,\C \in \LGCAT{\V}$, with the notations of \ref{thm:func_cat_for_product_graded_categories} and \ref{para:cstar}.
\end{thm}
\begin{proof}
Let $\A$ and $\C$ be $\V$-graded categories. Since $\C_\dpr$ is a $(\V \times \V)$-graded category, we can apply \ref{thm:func_cat_for_product_graded_categories} to form the (left) $\V$-graded category $[\A,\C_\dpr]_r$, whose objects are $\V$-graded functors $\A \to U^*_\ell\C_\dpr$ and whose graded morphisms are certain families of graded morphisms in $U_r^*\C_\dpr$ (\ref{rem:charn_gr_morphs_in_func_cat_prod_graded}). But by \eqref{eq:two_underl_vcats_iso} we have identity-on-objects isomorphisms $L:\C \xrightarrow{\sim} U^*_\ell\C_\dpr$ and $R:\C \xrightarrow{\sim} U^*_r\C$ given on graded morphisms by $Lf = f_\ell$ and $Rf = f_r$, respectively.  Hence each object of $[\A,\C_\dpr]_r$ is of the form $LF$ for a unique $\V$-graded functor $F:\A \to \C$, and if $F,G:\A \rightrightarrows \C$ are $\V$-graded functors, then by \ref{rem:charn_gr_morphs_in_func_cat_prod_graded}, a graded morphism $X' \gc LF \to LG$ in $[\A,\C_\dpr]_r$ is equivalently given by graded morphisms $\phi_A:X' \gc FA \to GA$ in $\C$ $(A \in \ob\A)$ such that for each graded morphism $f:X \gc A \to B$ in $\A$ the quadruple $((Ff)_\ell,(Gf)_\ell,(\phi_A)_r,(\phi_B)_r)$ is a bigraded square in $\C_\dpr$, equivalently, by \ref{thm:bigr_sq_normal_duoidal}, $(Ff,Gf,\phi_A,\phi_B)$ is a $\dV$-graded square in $\C$. Thus $\V$-graded functors $F:\A \to \C$ are the objects of a $\V$-graded category $[\A,\C]_r$ isomorphic to $[\A,\C_\dpr]_r$ and, similarly, a $\V$-graded category $[\A,\C]_\ell$ isomorphic to $[\A,\C_\dpr]_\ell$. The remaining claims now follow, by \ref{para:func_cat_rep_gbifun_for_product-graded_cat} and \ref{thm:duoidally_vgr_bifun}.
\end{proof}

\begin{exa}\label{exa:gr_func_cat_symm_br}
If $\dV = (\V,c)$ is a symmetric monoidal category, regarded as duoidal, then $[\A,\C]_r = [\A,\C]_\ell$, by \ref{thm:flipped_square} and \ref{thm:func_cat_repn_duoidal}, and in view of the characterization of $\dV$-graded squares in \eqref{eq:vgr_sq_braided} we recover the $\V$-graded functor category $[\A,\C]$ that was studied by Wood in \cite[\S 1.6]{Wood:thesis} and coincides with the $\hat{\V}$-enriched functor category that is obtained as an instance of \cite{DayKe:EnrFuncCats} (and \cite[\S 2.2]{Ke:Ba}) in this case since $\hat{\V}$ is a symmetric monoidal closed category with $\SET$-small limits (\ref{sec:day_conv}).

But if $\dV = (\V,c)$ is a braided monoidal category that is not symmetric, then $[\A,\C]_r$ and $[\A,\C]_\ell$ are in general distinct despite having the same objects: For example, by \ref{exa:vgr_func}(3), \ref{thm:bigr_sq_normal_duoidal}, \ref{thm:func_cat_repn_duoidal}, an object of $[\tensor[_X]{\mathbbm{2}}{},\C]_r$ or $[\tensor[_X]{\mathbbm{2}}{},\C]_\ell$ is given by a graded morphism $f:X \gc A \to A'$ in $\C$, while a graded morphism $(\phi,\phi'):X' \gc f \to g$ in $[\tensor[_X]{\mathbbm{2}}{},\C]_r$ (resp.~$[\tensor[_X]{\mathbbm{2}}{},\C]_\ell$) is a pair of graded morphisms $\phi$, $\phi'$ of grade $X'$ in $\C$ such that $(f,g,\phi,\phi')$ (resp.~$(\phi,\phi',f,g)$) is a $\dV$-graded square, so the graded morphisms in $[\tensor[_X]{\mathbbm{2}}{},\C]_r$ and $[\tensor[_X]{\mathbbm{2}}{},\C]_\ell$ are in general distinct, by \ref{thm:flipped_square}.
\end{exa}
 
The following special case of the notion of $\dV$-graded square will play an important role in the remainder of this section:

\begin{defn}\label{defn:v-enr_sq}
A \textbf{$\dV$-enriched square} is a $\dV$-graded square in some $\dV$-enriched category $\C$ (regarded as $\dV$-graded by \ref{para:vcats_as_vgr_cats}).
\end{defn}

\begin{para}\label{para:venr_sq}
Unpacking the definition, a $\dV$-enriched square in a $\dV$-category $\C$ is precisely a quadruple $(f,g,\phi,\phi')$ of morphisms
\begin{equation}\label{eq:venr_sq}f:X \to \C(A,A'),\;g:X \to \C(B,B'),\;\phi:X' \to \C(A,B),\;\phi':X' \to \C(A',B')\end{equation}
in $\V$ such that the following diagram commutes:
\begin{equation}\label{eq:venr_sq_diag}
\xymatrix{
X \dpr X' \ar[d]_{\tau} \ar[r]^\sigma & X \otimes X' \ar[r]^(.35){g \otimes \phi} & \C(B,B') \otimes \C(A,B) \ar[d]^{m_{ABB'}}\\
X' \otimes X \ar[r]^(.35){\phi' \otimes f} & \C(A',B') \otimes \C(A,A') \ar[r]^(.6){m_{AA'B'}} & \C(A,B')
}
\end{equation}
\end{para}

\begin{para}\label{para:star-stable_epi-sink}
A family of morphisms $\alpha = (\alpha_k:Y_k \to X)_{k \in K}$ in a category $\X$ is an \textit{epi-sink} \cite{AHS} provided that for all $\beta,\gamma:X \rightrightarrows Z$ in $\X$, if $\beta \cdot \alpha_k = \gamma \cdot \alpha_k$ for all $k \in K$ then $\beta = \gamma$. For example, every colimit cocone is an epi-sink. A \textit{$\dpr$-stable epi-sink} in $\dV = (\V,\dpr)$ is a family of morphisms $\alpha = (\alpha_k:Y_k \to X)_{k \in K}$ in $\V$ such that for all $W \in \ob\V$ the families $(W \dpr \alpha_k)_{k \in K}$ and $(\alpha_k \dpr W)_{k \in K}$ are epi-sinks in $\V$. For each object $X$ of $\V$, the singleton family given by $1_X$ is a $\dpr$-stable epi-sink, as is the family of all morphisms with codomain $X$. Every left adjoint functor preserves epi-sinks, so if $\dV$ is \textit{$\dpr$-biclosed} \cite{GaLf} in the sense that the monoidal category $(\V,\dpr,I,\alpha,\lambda,\rho)$ is biclosed, then every epi-sink in $\V$ is $\dpr$-stable.
\end{para}

\begin{lem}\label{thm:star-stable_epi-sink}
Let $(\alpha_s:Y_s \to X)_{s \in S}$ and $(\beta_t:Z_t \to X')_{t \in T}$ be $\dpr$-stable epi-sinks in $\dV$, and let $\C$ be a $\dV$-category. Then a quadruple of morphisms $(f,g,\phi,\phi')$ as in \eqref{eq:venr_sq} is a $\dV$-enriched square iff $(f \cdot \alpha_s,g \cdot \alpha_s,\phi \cdot \beta_t, \phi' \cdot \beta_t)$ is a $\dV$-enriched square for all $s \in S$ and $t \in T$.
\end{lem}
\begin{proof}
The hypotheses entail that $(\alpha_s \dpr \beta_t:Y_s \dpr Z_t \to X \dpr X')_{(s,t) \in S \times T}$ is an epi-sink (using the fact that $\alpha_s \dpr \beta_t = (X \dpr \beta_t) \cdot (\alpha_s \dpr Z_t)$). The result now follows via \ref{para:venr_sq}, using the naturality of $\sigma$ and $\tau$.
\end{proof}

\begin{para}\label{para:duoidal_str_vhat}
Since $\dV = (\V,\dpr)$ is normal duoidal, it follows by \cite[4.8]{BooSt} that $\hat{\V}$ carries the structure of a (huge) normal duoidal category, which we denote by $\hat{\dV}$, in such a way that the Yoneda embedding $\mathsf{Y}:\V \to \hat{\V}$ underlies a strong duoidal functor $\mathsf{Y}:\dV \to \hat{\dV}$. Explicitly, the monoidal products $\otimes,\dpr:\hat{\V} \times \hat{\V} \rightrightarrows \hat{\V}$ are obtained by Day convolution from $\otimes,\dpr:\V \times \V \to \V$.
\end{para}

Next we compare $\dV$-graded squares with $\hat{\dV}$-enriched squares. We write $\tilde{f}:\mathsf{Y}X = \V(-,X) \Rightarrow \C(A,A') = \C(- \gc A;A'):\V^\op \to \SET$ for the natural transformation corresponding to a graded morphism $f \in \C(X \gc A;A')$ in a $\dV$-graded category $\C$.

\begin{lem}\label{thm:vgr-sq-vhat-enr-sq}
Let $f:X \gc A \to A'$, $g:X \gc B \to B'$, $\phi:X' \gc A \to B$, $\phi':X' \gc A' \to B'$ be graded morphisms in a $\dV$-graded category $\C$. Then $(f,g,\phi,\phi')$ is a $\dV$-graded square iff $(\tilde{f},\tilde{g},\tilde{\phi},\tilde{\phi'})$ is a $\hat{\dV}$-enriched square in the $\hat{\dV}$-category $\C$.
\end{lem}
\begin{proof} $(\tilde{f},\tilde{g},\tilde{\phi},\tilde{\phi'})$ is a $\hat{\dV}$-enriched square iff the rectangle (1) in the following diagram commutes, but since $\mathsf{Y}:\dV \to \hat{\dV}$ is a strong duoidal functor, the cells marked (2) and (3) necessarily commute, where we write $d^\otimes$ and $d^\dpr$ for the (invertible) binary opmonoidal constraints associated to $\mathsf{Y}$:
$$
\xymatrix{
\mathsf{Y}(X \dpr X') \ar[d]_{\mathsf{Y}\tau} \ar[r]^{\mathsf{Y}\sigma} \ar[dr]|{d^\dpr} & \mathsf{Y}(X \otimes X') \ar@{}[d]|{(2)} \ar[dr]^{d^\otimes} & &\\
\mathsf{Y}(X' \otimes X) \ar@{}[r]|{(3)} \ar[dr]_{d^\otimes} & \mathsf{Y}X \dpr \mathsf{Y}X' \ar[d]_\tau \ar[r]^\sigma & \mathsf{Y}X \otimes \mathsf{Y}X' \ar@{}[d]|{(1)} \ar[r]^(.35){\tilde{g} \otimes \tilde{\phi}} & \C(B,B') \otimes \C(A,B) \ar[d]^{m_{ABB'}}\\
& \mathsf{Y}X' \otimes \mathsf{Y}X \ar[r]^(.4){\tilde{\phi'} \otimes \tilde{f}} & \C(A',B') \otimes \C(A,A') \ar[r]^(.6){m_{AA'B'}} & \C(A,B')
}
$$
Also, $(f,g,\phi,\phi')$ is a $\dV$-graded square iff $\sigma^*(g \circ \phi) = \tau^*(\phi' \circ f):X \dpr X' \gc A \to B'$. But by \ref{sec:day_conv} the two composites around the periphery of the diagram are precisely $\widetilde{\sigma^*(g \circ \phi)}, \widetilde{\tau^*(\phi' \circ f)}:\mathsf{Y}(X \dpr X') \rightrightarrows \C(A,B')$.
\end{proof}

\begin{lem}\label{thm:psh_enr_sq}
Let $\C$ be a $\dV$-graded category. Then a $\hat{\dV}$-enriched square in $\C$ is a quadruple $(f,g,\phi,\phi')$ consisting of natural transformations $f:P \Rightarrow \C(- \gc A;A')$, $g:P \Rightarrow \C(- \gc B;B')$, $\phi:Q \Rightarrow \C(- \gc A;B)$, $\phi':Q \Rightarrow \C(-\gc A';B')$ such that
$$f_X(p):X \gc A \to A',\;\;g_X(p):X \gc B \to B',\;\;\phi_{X'}(q):X' \gc A \to B,\;\;\phi'_{X'}(q):X' \gc A' \to B'$$
is a $\dV$-graded square in $\C$ for all $X,X' \in \ob\V$, $p \in PX$, $q \in QX'$.
\end{lem}
\begin{proof}
By the Yoneda lemma, there is an epi-sink $\gamma^P = (\tilde{p}:\mathsf{Y}X \to P)_{X \in \ob\V\hspace{-0.3ex},\,p \in PX}$ in $\hat{\V}$ with the notation of \ref{sec:day_conv}, and $\gamma^P$ is a $\dpr$-stable epi-sink since $\hat{\dV}$ is $\dpr$-biclosed. Similarly, $\gamma^Q = (\tilde{q}:\mathsf{Y}X' \to Q)_{X' \in \ob\V\hspace{-0.3ex},\,q \in QX'}$ is a $\dpr$-stable epi-sink. The result now follows from Lemmas \ref{thm:star-stable_epi-sink} and \ref{thm:vgr-sq-vhat-enr-sq}.
\end{proof}

By the characterization of $\dV$-enriched squares in \ref{para:venr_sq}, the following is an equivalent way of defining the notion of enriched bifunctor of Garner and L\'opez Franco \cite[Definition 10]{GaLf}:

\begin{defn}\label{defn:v-enr_bifunc}
Given a normal duoidal category $\dV = (\V,\dpr)$ and $\dV$-categories $\A$, $\B$, $\C$, a \textbf{$\dV$-enriched bifunctor} $F:\A,\B \to \C$ consists of (left) $\V$-functors $F(-,B):\A \to \C$ $(B \in \ob\B)$ and $F(A,-):\B \to \C$ $(A \in \ob\A)$ that agree on objects, such that for all $A,A' \in \ob\A$ and $B,B' \in \ob\B$ the morphisms
$$F(-,B)_{AA'},\;F(-,B')_{AA'},\;F(A,-)_{BB'},\;F(A',-)_{BB'}$$
constitute a $\dV$-enriched square in $\C$ of the following form:
$$
\xymatrix{
F(A,B) \ar@{-}[d]_{F(-,B)_{AA'}} \ar@{-}[rr]^{F(A,-)_{BB'}}="s1" & & F(A,B') \ar@{-}[d]^{F(-,B')_{AA'}}\\
F(A',B) \ar@{-}[rr]_{F(A',-)_{BB'}}="t1" & & F(A',B')
\ar@{}"s1";"t1"|(0.5){\A(A,A'),\,\B(B,B')}
}
$$
By \cite[\S 3.2]{GaLf}, $\dV$-enriched bifunctors $F:\A,\B \to \C$ are the objects of a category \linebreak $\textnormal{$\dV$-\textsf{Bif}}(\A,\B;\C)$ with the evident morphisms, and this yields a 2-functor $\textnormal{$\dV$-\textsf{Bif}}:\LCAT{\V}^\op$ $\times \LCAT{\V}^\op \times \LCAT{\V} \to \CAT$. We now show that $\dV$-enriched bifunctors are equivalently $\dV$-graded bifunctors between $\dV$-enriched categories.
\end{defn}

\begin{thm}\label{thm:duoi_venr_bifun}
Let $\dV = (\V,\dpr)$ be a normal duoidal category. Given $\dV$-categories $\A,\B,\C$, a $\dV$-enriched bifunctor $F:\A,\B \to \C$ is equivalently given by a $\dV$-graded bifunctor $F:\A,\B \to \C$ (where we regard $\A,\B,\C$ as $\dV$-graded categories via \ref{para:vcats_as_vgr_cats}), or equivalently, a graded bifunctor $F:\A,\B \to \C_\dpr$. Moreover, there are isomorphisms
$\textnormal{$\dV$-\textsf{Bif}}(\A,\B;\C) \cong \textnormal{$\dV$-\textsf{GBif}}(\A,\B;\C) \cong \textnormal{\textsf{GBif}}(\A,\B;\C_\dpr)$, 2-natural in $\A,\B,\C \in \LCAT{\V}$, where we omit from our notation the 2-functor $\mathsf{Y}_*:\LCAT{\V} \to \LGCAT{\V}$ of \ref{para:vcats_as_vgr_cats}.
\end{thm}
\begin{proof}
By \ref{para:star-stable_epi-sink} and \ref{thm:star-stable_epi-sink}, $\V$-functors $F(-,B):\A \to \C$ $(B \in \ob\B)$ and $F(A,-):\B \to \C$ $(A \in \ob\A)$ that agree on objects form a $\dV$-enriched bifunctor iff
$$(F(-,B)_{AA'} \cdot f,F(-,B')_{AA'} \cdot f,F(A,-)_{BB'} \cdot g,F(A',-)_{BB'} \cdot g)$$
is a $\dV$-enriched square for all $A,A' \in \ob\A$, $B,B' \in \ob\B$, $f:X \to \C(A,A')$, and $g:X' \to \C(B,B')$ in $\V$. Since $\mathsf{Y}_*$ is fully faithful, the result follows, using Theorem \ref{thm:duoidally_vgr_bifun} and \ref{para:2func_vgbif}.
\end{proof}

In the opposite direction, we now show that $\dV$-graded bifunctors for a normal duoidal $\dV$ are precisely $\hat{\dV}$-en\-rich\-ed bifunctors:

\begin{thm}\label{thm:hatv_enr_bifun}
Let $\dV = (\V,\dpr)$ be a normal duoidal category. Given $\dV$-graded categories $\A,\B,\C$, a $\dV$-graded bifunctor $F:\A,\B \to \C$ is precisely a $\hat{\dV}$-enriched bifunctor $F:\A,\B \to \C$. Moreover $\textnormal{$\dV$-\textsf{GBif}} = \textnormal{$\hat{\dV}$-\textsf{Bif}}:\LGCAT{\V}^\op \times \LGCAT{\V}^\op \times \LGCAT{\V} \to \CAT$.
\end{thm}
\begin{proof}
This follows from Lemma \ref{thm:psh_enr_sq}, in view of Definitions \ref{para:duoidally_vgr_bifun} and \ref{defn:v-enr_bifunc}.
\end{proof}

\begin{para}\label{para:enr_func_cat_gl}
Garner and L\'opez Franco \cite{GaLf} show that if $\dV = (\V,\dpr)$ is a normal duoidal category that is complete and $\dpr$-biclosed (\ref{para:star-stable_epi-sink}), then the 2-functors $$\textnormal{$\dV$-\textsf{Bif}}(\A,-;\C),\textnormal{$\dV$-\textsf{Bif}}(-,\A;\C):\LCAT{\V}^\op \rightrightarrows \CAT$$
are representable whenever $\A$ and $\C$ are $\dV$-categories and $\A$ is small. In this case, we write the respective representing objects as $[\A,\C]_r^{\tinydV}$ and $[\A,\C]_\ell^{\tinydV}$ and call them \textit{the $\dV$-enriched functor categories of Garner and L\'opez Franco}; we follow the convention of \cite{GaLf} for the placement of the subscripts ``$r$'' and ``$\ell$''. Explicitly, the objects of $[\A,\C]_r^{\tinydV}$ and $[\A,\C]_\ell^{\tinydV}$ are (in both cases) $\V$-functors $F:\A \to \C$. To describe the hom-objects of the former, let us follow \cite{GaLf} in writing $[X,-]_r$ for the right adjoint of $X \dpr (-):\V \to \V$ $(X \in \ob\V)$. Let $F,G:\A \rightrightarrows \C$ be $\V$-functors. For each pair $A,B \in \ob\A$, we can form the following composite morphisms in $\V$:
$$m_{FA,GA,GB} \cdot (G_{AB} \otimes 1) \cdot \sigma\;:\;\A(A,B) \dpr \C(FA,GA) \to \C(FA,GB),$$
$$m_{FA,FB,GB} \cdot (1 \otimes F_{AB}) \cdot \tau\;:\;\A(A,B) \dpr \C(FB,GB) \to \C(FA,GB).$$
Taking the transposes of these morphisms, we obtain a diagram of the form
$$\C(FA,GA) \to [\A(A,B),\C(FA,GB)]_r \leftarrow \C(FB,GB)\;\;\;\;(A,B \in \ob\A)$$
in $\V$ whose limit serves as the hom-object $[\A,\C]_r^{\tinydV}(F,G)$. Composition in $[\A,\C]_r^{\tinydV}$ is defined using the universal property of such limits and the composition in $\C$, and $[\A,\C]_\ell^{\tinydV}$ is constructed analogously, using instead the right adjoints $[X,-]_\ell$ to the functors $(-) \dpr X$. 
\end{para}

\begin{para}\label{para:vgr_func_cat_gl}
We now show that for an arbitrary normal duoidal category $\dV$, the $\dV$-graded functor categories of \ref{thm:func_cat_repn_duoidal} can be regarded as examples of the enriched functor categories of Garner and L\'opez Franco relative to the base of enrichment $\hat{\dV}$. The (huge) normal duoidal category $\hat{\dV}$ has all $\SET$-small limits and is $\dpr$-biclosed, so if $\A$ and $\C$ are ($\SET$-small) $\dV$-graded categories then by applying \ref{para:enr_func_cat_gl} relative to the base of enrichment $\hat{\dV}$ and the universe $\SET$, we can form the $\hat{\dV}$-enriched functor categories $[\A,\C]_r^{\scriptscriptstyle\hat{\dV}}$ and $[\A,\C]_\ell^{\scriptscriptstyle\hat{\dV}}$ of Garner and L\'opez Franco.  By \ref{para:enr_func_cat_gl} and Theorems \ref{thm:hatv_enr_bifun} and \ref{thm:func_cat_repn_duoidal} we have isomorphisms $\LGCAT{\V}(\B,[\A,\C]_r^{\scriptscriptstyle\hat{\dV}}) \cong \textnormal{$\hat{\dV}$-\textsf{Bif}}(\A,\B;\C) = \textnormal{$\dV$-\textsf{GBif}}(\A,\B;\C) \cong \LGCAT{\V}(\B,[\A,\C]_r)$ 2-natural in $\B \in \LGCAT{\V}$. Thus we obtain the following, reasoning similarly also for $[\A,\C]_\ell^{\scriptscriptstyle\hat{\dV}}$.
\end{para}

\begin{thm}\label{thm:dhat_enr_func_cat}
Let $\A$ and $\B$ be $\dV$-graded categories for a normal duoidal category $\dV$. Then $[\A,\C]_r^{\scriptscriptstyle\hat{\dV}} \cong [\A,\C]_r$ and $[\A,\C]_\ell^{\scriptscriptstyle\hat{\dV}} \cong [\A,\C]_\ell$.
\end{thm}

Turning to the case where $\dV$ is complete and $\dpr$-biclosed, we now show that the $\dV$-enriched functor categories of Garner and L\'opez Franco can be regarded as examples of the $\dV$-graded functor categories of \ref{thm:func_cat_repn_duoidal}, which in turn are examples of the graded functor categories of \ref{thm:func_cat_for_product_graded_categories} and \ref{thm:vgr_func_cat}:

\begin{thm}\label{thm:venr_func_cat_gr_func_cat}
Let $\dV = (\V,\dpr)$ be a complete and $\dpr$-biclosed normal duoidal category, let $\A$ and $\C$ be $\dV$-categories, and suppose $\A$ is small. Then $[\A,\C]_r^{\tinydV} \cong [\A,\C]_r$ and $[\A,\C]_\ell^{\tinydV} \cong [\A,\C]_\ell$ in $\LGCAT{\V}$, where we regard the $\dV$-categories $\A$, $\C$, $[\A,\C]_r^{\tinydV}$, $[\A,\C]_\ell^{\tinydV}$ as $\dV$-graded categories via \ref{para:vcats_as_vgr_cats}, thus omitting applications of $\mathsf{Y}_*:\LCAT{\V} \to \LGCAT{\V}$ from our notation.
\end{thm}
\begin{proof}
Since $\mathsf{Y}:\dV \to \hat{\dV}$ is a strong duoidal functor and preserves all limits, it is clear from the construction of the enriched functor categories in \ref{para:enr_func_cat_gl} that $\mathsf{Y}_*[\A,\C]_r^{\tinydV} \cong [\mathsf{Y}_*\A,\mathsf{Y}_*\C]_r^{\scriptscriptstyle\hat{\dV}}$ and $\mathsf{Y}_*[\A,\C]_\ell^{\tinydV} \cong [\mathsf{Y}_*\A,\mathsf{Y}_*\C]_\ell^{\scriptscriptstyle\hat{\dV}}$, and the result follows by \ref{thm:dhat_enr_func_cat}.
\end{proof}

\section{Appendix - Background III: Weighted colimits and cocompletion}\label{sec:weighted_colims}

In this section, we fix a huge biclosed base $\sfV$ in the sense of \ref{para:copower_cocompl}, and we recall background material used only in \ref{para:copower_cocompl} (and, in turn, in \ref{thm:copower_cocompl_vgr}).

\begin{para}[\textbf{Weighted colimits}]\label{para:weighted_colims}
Street \cite{Str:EnrCatsCoh} generalized the theory of enriched weighted colimits \cite{Ke:Ba} to the setting of enrichment in bicategories, employing general enriched modules as weights. We now recall the monoidal case of Gordon and Power's \cite{GP:GabrielUlmer} version of Street's theory, i.e.~the version for left $\sfV$-categories in which only presheaves are used as the weights. A \textbf{weight} (for colimits in left $\sfV$-categories) is a right $\sfV$-functor $W:\K^\circ \to \sfV$ for a given ($\SET$-small) left $\sfV$-category $\K$; thus a weight is equivalently a $\sfV$-module $W:\mathbb{I} \modto \K$ (\S \ref{sec:enr_mod_psh_cocompl}). Given also a huge left $\sfV$-category $\C$ and a left $\sfV$-functor $D:\K \to \C$, a \textbf{weighted colimit} $W * D$ is an object of $\C$ equipped with a right $\sfV$-natural transformation $\gamma:W \Rightarrow \C(D-,W * D):\K^\circ \to \sfV$, equivalently a 2-cell $\gamma:W \Rightarrow \C(D,W*D):\mathbb{I} \modto \K$ in $\LMOD{\sfV}$ such that for each object $A$ of $\C$ the induced morphism $\tilde{\gamma}_A:\C(W*D,A) \to (\cP\K)(W,\C(D-,A))$ is an isomorphism in $\sfV$.  Here $\tilde{\gamma}_A$ is induced (via \ref{para:psh_yoneda_left_vcats}) by the composite morphisms 
$\C(W*D,A) \otimes WK \xrightarrow{1 \otimes \gamma_K} \C(W*D,A) \otimes \C(DK,W*D) \xrightarrow{m} \C(DK,A)$ in $\sfV$ $(K \in \ob\K)$. 

Copowers (or tensors) are a special case of weighted colimits: Given a huge left $\sfV$-category $\C$, a left $\sfV$-functor $A:\mathbb{I} \to \C$ is given by an object $A$ of $\C$, while a right $\sfV$-functor $X:\mathbb{I}^\circ = \mathbb{I} \to \sfV$ is given by an object $X$ of $\sfV$. With these identifications, a weighted colimit $X * A$ is precisely a copower $X \cdot A$ in $\C$.
\end{para}

\begin{para}[\textbf{Cocompletion for a class of weights}]\label{para:cocompl}
Kelly studied free cocompletions of enriched categories with respect to a class of weights \cite[\S 5.7]{Ke:Ba}, and Power, Cattani, and Winskel \cite{PCW:Cocompl} extended Kelly's methods to the setting of enrichment in non-symmetric monoidal categories; a treatment in an even more general setting, with explicit proofs, is given in \cite[\S 12]{GaShu}. In particular, by \cite[Theorem 6]{PCW:Cocompl}, if $\C$ is a ($\SET$-small) left $\sfV$-category and $\Phi$ is a $\SET$-small set of weights, then there is a left $\sfV$-category $\Phi(\C)$ that is \textit{$\Phi$-cocomplete} (i.e.~admits all weighted colimits whose weights lie in $\Phi$) and is equipped with a fully faithful (left) $\V$-functor $K:\C \to \Phi(\C)$ such that for every $\Phi$-cocomplete left $\sfV$-category $\D$ the functor $(-) \circ K:\textnormal{$\Phi$-COCTS}(\Phi(\C),\D) \to \LCAT{\sfV}(\C,\D)$ is an equivalence, where $\textnormal{$\Phi$-COCTS}(\Phi(\C),\D)$ is the full subcategory of $\LCAT{\sfV}(\Phi(\C),\D)$ spanned by the $\sfV$-functors preserving $\Phi$-weighted colimits. Explicitly, $\Phi(\C)$ is the full sub-$\sfV$-category of $\cP\C$ obtained as the closure of the representables under $\Phi$-weighted colimits \cite[Theorem 6]{PCW:Cocompl}, and $K:\C \to \Phi(\C)$ sends each object $A$ of $\C$ to the representable right $\sfV$-functor $\C(-,A):\C^\circ \to \sfV$. The $\SET$-smallness of $\Phi(\C)$ follows from that of $\Phi$ and of $\C$, because $\Phi(\C)$ can be constructed by transfinite recursion just as in \cite[\S 3.5, 5.7]{Ke:Ba} (taking $\Phi$-colimits at successor steps and unions at limit steps). We refer readers also to \cite[Theorem 12.1]{GaShu} for an extension of this material on free $\Phi$-cocompletion to an even more general setting, with an explicit proof.
\end{para}

\bibliographystyle{amsplain}
\bibliography{bib}

\end{document}